\documentclass[11pt,reqno]{amsart}

\usepackage[T1]{fontenc}
\usepackage[utf8]{inputenc}
\usepackage{lmodern}
\usepackage{microtype}
\usepackage{geometry}
\usepackage{amsmath,amssymb,amsfonts,mathtools,mathrsfs}
\usepackage{bm}
\usepackage{enumitem}
\usepackage{booktabs}
\usepackage{longtable}
\usepackage{array}
\usepackage{tikz-cd}
\usepackage{hyperref}
\usepackage[nameinlink,capitalize,noabbrev]{cleveref}
\usepackage{xcolor}
\usepackage{setspace}
\hypersetup{colorlinks=true,linkcolor=blue!55!black,citecolor=green!35!black,urlcolor=blue!55!black}
\allowdisplaybreaks
\newtheorem{theorem}{Theorem}[section]
\newtheorem{proposition}[theorem]{Proposition}
\newtheorem{lemma}[theorem]{Lemma}
\newtheorem{corollary}[theorem]{Corollary}

\newtheorem{question}[theorem]{Question}
\newtheorem{problem}[theorem]{Problem}
\theoremstyle{definition}
\newtheorem{definition}[theorem]{Definition}
\newtheorem{example}[theorem]{Example}

\newtheorem{assumption}[theorem]{Assumption}
\theoremstyle{remark}
\newtheorem{remark}[theorem]{Remark}
\newtheorem{warning}[theorem]{Warning}

\newcommand{\C}{\mathbb C}
\newcommand{\R}{\mathbb R}

\newcommand{\cA}{\mathcal A}

\newcommand{\cC}{\mathcal C}
\newcommand{\cD}{\mathcal D}

\newcommand{\cF}{\mathcal F}

\newcommand{\cH}{\mathcal H}

\newcommand{\cO}{\mathcal O}
\newcommand{\cP}{\mathcal P}
\newcommand{\cQ}{\mathcal Q}

\newcommand{\cS}{\mathcal S}
\newcommand{\cT}{\mathcal T}

\newcommand{\cZ}{\mathcal Z}
\newcommand{\End}{\operatorname{End}}

\newcommand{\Hom}{\operatorname{Hom}}
\newcommand{\Herm}{\operatorname{Herm}}
\newcommand{\Ext}{\operatorname{Ext}}

\newcommand{\im}{\operatorname{im}}

\newcommand{\Id}{\operatorname{Id}}
\newcommand{\tr}{\operatorname{tr}}
\newcommand{\rk}{\operatorname{rk}}
\newcommand{\Vol}{\operatorname{Vol}}

\newcommand{\Dol}{\mathrm{Dol}}
\newcommand{\dR}{\mathrm{dR}}
\newcommand{\Hod}{\mathrm{Hod}}
\newcommand{\har}{\mathrm{har}}
\newcommand{\st}{\mathrm{st}}

\newcommand{\ssm}{\mathrm{ss}}
\newcommand{\poly}{\mathrm{poly}}

\newcommand{\irr}{\mathrm{irr}}

\newcommand{\ddbar}{\overline\partial}
\newcommand{\eps}{\varepsilon}
\newcommand{\ip}[2]{\left\langle #1,#2\right\rangle}
\newcommand{\norm}[1]{\left\lVert #1\right\rVert}
\newcommand{\abs}[1]{\left\lvert #1\right\rvert}
\newcommand{\set}[1]{\left\{#1\right\}}

\newcommand{\xto}[1]{\xrightarrow{#1}}

\newcommand{\MdiffDol}{\mathscr M_{\Dol}}
\newcommand{\MdiffdR}{\mathscr M_{\dR}}
\newcommand{\MdiffHod}{\mathscr M_{\Hod}}

\newcommand{\MHDol}{\mathscr M^{\cH}_{\Dol}}
\newcommand{\MHdR}{\mathscr M^{\cH}_{\dR}}
\newcommand{\MetHar}{\mathscr M_{\mathrm{Met},\har}}
\newcommand{\Ggauge}{\mathscr G}

\title[Diffeological non-Abelian Hodge theory: metrics and deformations]{Diffeological non-Abelian Hodge theory:\\relative harmonic metrics and deformation theory}
\author{Mahmud Azam}
\email{mahmud.azam@usask.ca}
\author{Steven Rayan}
\email{rayan@math.usask.ca}
\address{Centre for Quantum Topology and Its Applications (quanTA) and Department of Mathematics and Statistics, University of Saskatchewan, SK, Canada S7N 5E6}
\date{July 2026}

\begin{document}

\begin{abstract}
Let $X$ be a compact K\"ahler manifold.  In prior work \cite{AzamRayan2026}, we constructed diffeological moduli stacks of Higgs and flat bundles on $X$, related through finite extension completion of smooth harmonic families.  Here we develop the analytic and infinitesimal geometry of that construction and distinguish the stable, polystable, and semistable family problems.  Our approach is intrinsically relative: we place the Hitchin--Simpson equation on Sobolev completions over arbitrary parameter plots, rather than first imposing a finite-dimensional smooth structure on a coarse moduli space.

Every smooth stable Higgs family satisfying the usual non-Abelian Hodge numerical conditions admits a global smooth harmonic metric.  A global Hermitian--Einstein determinant metric always exists, and prescribing it makes the harmonic metric unique.  Locally, the result follows from a Banach implicit-function theorem: the metric linearization is a self-adjoint elliptic Jacobi operator whose trace-free kernel vanishes by stability.  Parameter-dependent elliptic regularity gives joint smoothness, and normalized uniqueness gives global gluing.  Thus stable families already lie in the prestack-level image of the harmonic mediator, and the non-Abelian Hodge transform is smooth on plots.  The same solution operator preserves every finite parameter regularity $C^d$ and pulls back to reduced singular parameter spaces admitting ambient smooth coefficient extensions.

For a Higgs deformation $\eta$, the normalized first metric variation satisfies
\[
  L_hs=-\mathcal S_h(\eta),
  \qquad
  s=-G_h\mathcal S_h(\eta),
\]
with an independent rank-one determinant term in the general $\mathrm{GL}_r$ case.  This yields the differential along actual plots and agrees with the classical cohomological comparison on smooth deformation loci.

For locally split polystable families of constant type we construct smooth harmonic metrics and describe their centralizer-valued ambiguity.  Explicit real-analytic examples show that an arbitrary polystable family need not admit even a continuous harmonic metric; one example also has no relative harmonic filtration and lies outside every $C^d$ extension-generated locus.  Nevertheless, in that example the singular harmonic reduction has a continuous adjoint term and produces a continuous slicewise semisimple flat family.  This gives a weak $C^0$ operator-level harmonic mediator, strictly larger than the metric-regular mediator, and shows that failure of a continuous reduction need not preclude every family-level correspondence.  We characterize the extension-generated stack by relative harmonic filtrations and develop obstruction and Kuranishi theories for assembling their invariant subbundles.  We also explain why ordinary heat flow retains only the polystable shadow of a strictly semistable object, while a filtered or renormalized limit would have to retain extension data.  Finally, we construct the smooth Hodge $\lambda$-family on the stable locus.
\end{abstract}

\maketitle
\tableofcontents

\section{Introduction}
\subsection{From a family-level correspondence to its analytic geometry}

Let $X$ be a compact K\"ahler manifold.  Classical non-Abelian Hodge theory relates three types of data: semisimple complex local systems, polystable Higgs bundles with the standard Chern-class constraints, and harmonic bundles.  The correspondence has several incarnations.  At the level of coarse moduli spaces it is a homeomorphism.  On smooth stable loci it has finer real-analytic geometry, and for Higgs bundles on curves it participates in the familiar hyperk\"ahler structure of the moduli space.  At the categorical level, Simpson's extension argument gives an equivalence between semistable Higgs bundles satisfying the appropriate Chern conditions and arbitrary flat bundles.  These statements are compatible, but they retain different information near non-polystable objects.

In \cite{AzamRayan2026} we introduced diffeological moduli stacks
\[
  \MdiffDol(X),\qquad \MdiffdR(X)
\]
of smooth families of Higgs bundles and flat bundles on $X$.  We also introduced a dg-prestack of smooth families of harmonic bundles and used finite iterated extension completion to construct substacks
\[
  \MHDol(X)\subset \MdiffDol(X),
  \qquad
  \MHdR(X)\subset \MdiffdR(X)
\]
together with an equivalence
\begin{equation}\label{eq:intro-stack-equivalence}
  \MHDol(X)\simeq \MHdR(X).
\end{equation}
The fibre over a point of the left side is the category of semistable Higgs bundles with the usual Chern-class conditions, while the point fibre of the right side is the category of all flat bundles.  The construction also contains a smooth family appearing in Simpson's counterexample to a continuous extension of the coarse correspondence to the semistable locus.  Thus \eqref{eq:intro-stack-equivalence} is not merely the classical coarse correspondence rewritten in stack language.

The earlier paper was intentionally categorical.  Its central mechanism was to pass from harmonic families to extension-generated stacks.  The analytic question left open was the following: given a smooth family of polystable Higgs bundles, can the fibrewise harmonic metrics be chosen smoothly in the parameter?  Even on the stable locus, a complete proof requires a relative elliptic set-up, a normalization removing scalar ambiguity, and a parameter-dependent nonlinear argument.  Question~5.3.4 of \cite{AzamRayan2026} asked in addition what survives for families that are only $C^d$ in the parameter direction, and whether singular parameter spaces such as $C^\infty$-schemes should be allowed.  One advantage of the analytic set-up used here is that the nonlinear moment-map equation is solved on an ambient Sobolev affine space before restriction to the possibly singular integrable Higgs locus.  This observation supplies a finite-regularity theorem on the stable locus and a pullback result for a class of reduced singular parameter spaces.  At the polystable locus the linearized operator acquires a kernel determined by automorphisms, so the stable proof cannot simply be repeated; the examples in \cref{sec:explicit-polystable-failures} show that the unrestricted assertion is in fact false already for real-analytic families.  At the semistable locus one faces a further issue: a fibrewise Jordan--H\"older filtration need not assemble into a filtration by smooth subbundles over the parameter space.

We first prove global smooth existence of normalized harmonic metrics on the stable locus from a local analytic dependence theorem and normalized gluing.  We then calculate the first variation of the metric and the differential of the non-Abelian Hodge transform, establish a theorem for locally split polystable families of constant decomposition type, and reinterpret the extension-generated semistable stack through relative harmonic filtrations and obstruction classes.

\subsection{Stable harmonic metrics and the infinitesimal transform}

Let $U$ be a finite-dimensional smooth manifold and let
\[
  (E,\cD''_E)\longrightarrow U\times X
\]
be a smooth family of Higgs bundles in the sense of \cite{AzamRayan2026}.  Thus $E$ is a smooth complex vector bundle over $U\times X$ and, in the $X$-direction,
\[
  \cD''_E=\ddbar_E+\theta
\]
is an integrable Higgs operator depending smoothly on $u\in U$.  A Hermitian metric $h$ on $E$ determines a Chern operator $\partial_{E,h}$ and the $h$-adjoint $\theta^{\dagger_h}$.  We write
\[
  \cD'_{E,h}=\partial_{E,h}+\theta^{\dagger_h}
\]
and use the Hitchin--Simpson moment map
\begin{equation}\label{eq:intro-moment-map}
  \mu(\cD'',h)
  =\sqrt{-1}\,\Lambda_\omega
  \bigl(F_{D_h}+[\theta,\theta^{\dagger_h}]\bigr).
\end{equation}
Under the NAH numerical conditions introduced in \cref{sec:numerical-conditions}, a normalized solution of $\mu=0$ is a flat harmonic metric by the standard Chern--Weil identity.  We keep this distinction explicit because the implicit-function theorem solves the moment-map equation, while flatness uses the additional topological input.

Our first main theorem is stated in a form that allows the determinant normalization to vary with the family.

\begin{theorem}[Smooth variation of stable harmonic metrics]\label{thm:intro-smooth-metrics}
Let $X$ be a compact K\"ahler manifold, $U$ a smooth manifold, and
\[
  (E,\cD''_E)\longrightarrow U\times X
\]
a smooth family of stable Higgs bundles of fixed rank.  Assume that every fibre satisfies the NAH numerical conditions \eqref{eq:NAH-condition}.  Then there is a smooth Hermitian metric $h$ on $E$ over all of $U\times X$ such that $h_u$ is harmonic for $(E_u,\cD''_{E,u})$ for every $u\in U$.

More precisely, the determinant line admits a global smooth Hermitian--Einstein metric $q$, and after $q$ is fixed there is a unique such family with $\det h=q$.  Equivalently, after fixing any smooth background metric $\kappa$ on $E$, there is a unique harmonic family satisfying the integral normalization $N_\kappa(h)_u=0$ for every $u$.  If, after a local identification of the underlying bundle, the coefficients of the Higgs family and a chosen determinant normalization are real analytic in local coordinates on $U$, then $u\mapsto h_u$ is real analytic locally with values in every fixed Sobolev completion.
\end{theorem}

The analytic construction is local on $U$, and normalized uniqueness will glue its outputs globally.  After choosing a smooth identification of the underlying bundles over a neighbourhood of $u_0$, we write
\[
  h=h_0 e^s
\]
with $s$ a self-adjoint trace-free endomorphism.  The harmonic equation becomes a nonlinear map between Sobolev spaces.  Its derivative in the $s$-direction is a self-adjoint elliptic operator
\[
  L_{h_0}=\cD''^{\,*}_{h_0}\cD''
\]
on Hermitian trace-free endomorphisms, up to the harmless universal convention factor explained in \cref{sec:linearization}.  The kernel consists of Hermitian Higgs endomorphisms.  Stability implies that these are scalar, and the trace-free normalization removes them.  Thus $L_{h_0}$ is an isomorphism between the relevant Sobolev spaces.  The Banach implicit-function theorem gives a Sobolev family; parameter-dependent elliptic regularity gives joint smoothness.

The parameter manifold is arbitrary and finite dimensional; it need not be a coordinate chart in a pre-existing moduli space.  Moreover, smooth pullback of the Higgs family pulls back the normalized harmonic metric.  The theorem is therefore formulated directly on plots of the diffeological family stack.  Its relation to existing real-analytic results on relative moduli spaces is discussed in \cref{sec:literature-positioning}.

\begin{corollary}[Stable locus of the diffeological correspondence]\label{cor:intro-stable-stack}
Every object of the stable subprestack of $\MdiffDol(X)$ satisfying the NAH numerical conditions lies in the essential image of the Dolbeault forgetful map from the harmonic-bundle prestack, before finite extension completion or stackification.  Likewise, every irreducible flat family lies in the essential image of its de Rham forgetful map.  In particular, these loci are contained in $\MHDol(X)$ and $\MHdR(X)$, respectively.  The induced non-Abelian Hodge transform on stable families is a smooth morphism of diffeological stacks, with smooth inverse on irreducible families.
\end{corollary}

This answers the stable part of Question 5.3.1 of \cite{AzamRayan2026}.  The polystable problem is subtler and is treated separately below.

The existence theorem also gives access to the infinitesimal geometry of the correspondence.  We therefore continue with the first variation in the same stable-family framework, keeping separate the trace-free metric correction and the determinant contribution.

Let $t\mapsto (\cD''_t,h_t)$ be a one-parameter family through a stable harmonic bundle.  In a smooth gauge write
\[
  \dot{\cD}''=a^{0,1}+\varphi=: \eta.
\]
The fixed-metric companion variation is not the naive adjoint of this sum: variation of the Chern $(1,0)$-part contributes a minus sign.  We therefore set
\[
  \eta^{\star_h}:=-a^{\dagger_h}+\varphi^{\dagger_h}.
\]
For the metric write $\dot h=h s_{\mathrm{tot}}$ and decompose
\[
  s_{\mathrm{tot}}=s_0+\frac{\tau}{r}\Id_E,
  \qquad \tr(s_0)=0.
\]
The scalar function $\tau$ is fixed by the independently varying harmonic metric on $\det E$; only in fixed-determinant directions is $\tau=0$.  Differentiating the trace-free moment-map equation produces the Green-operator equation for $s_0$, while differentiating the determinant Poisson equation determines $\tau$.  This separation is essential in the general $\mathrm{GL}_r$ family problem.

\begin{theorem}[Plotwise infinitesimal non-Abelian Hodge transform]\label{thm:intro-infinitesimal}
On the stable locus, the diffeological non-Abelian Hodge transform is differentiable along every smooth plot.  For a plot tangent represented by an actual one-parameter Higgs family with infinitesimal deformation $\eta=(a^{0,1},\varphi)$, let $s_0$ be the trace-free part of the logarithmic metric variation.  Then
\begin{equation}\label{eq:intro-first-variation}
  L_hs_0=-\cS_h(\eta),
\end{equation}
and, after fixing the determinant family,
\begin{equation}\label{eq:intro-green-formula}
  s_0=-G_h\cS_h(\eta).
\end{equation}
In a chosen local smooth gauge, a representative of the infinitesimal flat deformation is
\begin{equation}\label{eq:intro-dnah}
  \dot\nabla
  =\eta+\eta^{\star_h}+\cD'_h s_{\mathrm{tot}},
\end{equation}
where $\eta^{\star_h}=-a^{\dagger_h}+\varphi^{\dagger_h}$ is the fixed-metric companion variation and $s_{\mathrm{tot}}$ also contains the independently determined scalar determinant variation.  In fixed-determinant directions, $s_{\mathrm{tot}}=s_0$ and \eqref{eq:intro-green-formula} gives an explicit Green-operator formula.

For every actual plot direction, \eqref{eq:intro-dnah} is $d_{\nabla_h}$-closed because it is obtained by differentiating a family of flat connections.  Infinitesimal complex gauge paths change the representative by an exact flat gauge term.  We distinguish this plotwise differential from the classical cohomological comparison of the Dolbeault and de Rham deformation complexes: at smooth or unobstructed points, the two agree on tangent classes represented by plots, but no step of the proof assumes that an arbitrary closed degree-one deformation integrates to a genuine family.
\end{theorem}

We give several equivalent forms of the source term $\cS_h$ and compare the formula with the harmonic-representative description of tangent spaces.  The statement is plotwise: it does not impose a single finite-dimensional smooth structure on the full semistable moduli problem.  The intrinsic object is the induced tangent cohomology class; a connection-valued formula requires a local smooth gauge.

\subsection{Polystable strata, semistable filtrations, and obstructions}

For a polystable Higgs bundle, the kernel of $L_h$ is the Hermitian part of the Higgs endomorphism algebra.  Consequently there is no canonical inverse $G_h$ without choosing a slice transverse to the stabilizer.  We treat a large and geometrically natural class of polystable families: those which are locally split with constant stable decomposition type.

Suppose locally on $U$ that
\begin{equation}\label{eq:intro-poly-split}
  (E,\cD'')\cong\bigoplus_{j=1}^m (E_j,\cD''_j)\otimes W_j,
\end{equation}
where the $(E_j,\cD''_j)$ are smooth stable families, pairwise non-isomorphic fibrewise, and the $W_j$ are fixed finite-dimensional multiplicity spaces.  Applying \cref{thm:intro-smooth-metrics} to each stable factor gives:

\begin{theorem}[Constant decomposition type]\label{thm:intro-poly}
A locally split polystable family of the form \eqref{eq:intro-poly-split} admits a smooth family of harmonic metrics locally on the parameter manifold.  Once harmonic metrics on the stable factors are normalized, the space of harmonic metrics on the multiplicity factors is a smooth bundle with fibre
\[
  \prod_{j=1}^m \operatorname{Herm}^+(W_j),
\]
and the complex centralizer $\prod_j\operatorname{GL}(W_j,\C)$ acts by pullback on these forms.  The isotropy group of a chosen metric is the compact centralizer
\[
  \prod_{j=1}^m U(W_j,q_j).
\]
Equivalently, the groupoid of harmonic metrics varies smoothly and has compact isotropy determined by the polystable decomposition.
\end{theorem}

This theorem does not claim that an arbitrary polystable family admits a smooth local splitting of the form \eqref{eq:intro-poly-split}.  Indeed, jumping stabilizer type and monodromy among isomorphic stable factors are precisely the phenomena that the general problem must confront.  The explicit elliptic-curve families of \cref{sec:explicit-polystable-failures} show that the unrestricted local-existence statement is false: a real-analytic polystable family may admit no continuous harmonic metric, and may even lie outside every finite-regularity extension-generated locus.

The strongest example has a further feature that refines the meaning of this failure.  Although its determinant-normalized harmonic metrics degenerate in the parameter direction, the corresponding adjoint Higgs operators extend continuously, and so do the associated flat connections.  We use this compensated regularity to introduce a weak $C^0$ operator-level harmonic mediator.  Its objects retain continuous Dolbeault and de Rham operators while requiring the harmonic metric only slicewise.  The resulting prestack strictly enlarges the metric-regular mediator and contains the square-root family before extension completion.  It does not define a smooth transform, and it does not eliminate extension completion for strictly semistable objects; rather, it separates failure of a regular harmonic reduction from failure of every family-level non-Abelian Hodge lift.

On the semistable locus, the earlier stack $\MHDol(X)$ was defined through finite iterated extension completion.  We give it a geometric reformulation.

\begin{definition}
A \emph{relative harmonic filtration} of a smooth family $(E,\cD'')$ over $U$ is a filtration by smooth $\cD''$-invariant subbundles
\[
  0=E_0\subset E_1\subset\cdots\subset E_\ell=E
\]
such that every quotient $E_i/E_{i-1}$ is a smooth family arising from harmonic bundles.
\end{definition}

This definition turns the categorical extension condition into a property that can be checked on a family over the parameter space.  The precise equivalence used later is the next theorem.

\begin{theorem}[Relative harmonic filtration criterion]\label{thm:intro-filtration}
A smooth Higgs family is an object of $\MHDol(X)$ if and only if, locally on the parameter manifold, it admits a relative harmonic filtration.  The analogous statement holds on the de Rham side.
\end{theorem}

The theorem is formally close to the extension-completion construction, but its geometric formulation exposes the source of the remaining semistable problem: fibrewise Jordan--H\"older filtrations may fail to define smooth subbundles.  We therefore introduce an obstruction package.  Given a fibrewise invariant subspace of fixed rank, a local smooth projection $p$ onto the candidate subbundle satisfies
\[
  (1-p)\cD''p=0.
\]
The failure term
\[
  \beta_p=(1-p)\cD''p
\]
is a relative second fundamental form.  Its linearization defines a first hypercohomology class.  After a lower-order formal lift and identifications have been chosen, higher corrections give stagewise obstruction classes.  To pass from formal lifting to an actual smooth family we construct a finite-dimensional Kuranishi map for invariant subbundles.  In particular, vanishing of the relevant $\mathbf H^1$ obstruction space gives a genuine local smooth lifting criterion.

The same filtration viewpoint also clarifies a phenomenon that is sometimes hidden by the language of analytic convergence.  The Yang--Mills--Higgs heat flow naturally moves toward closed complex-gauge orbits; on the strictly semistable locus those closed orbits are represented by polystable associated graded objects.  Thus the ordinary heat-flow limit is geometric precisely where the orbit is already closed.  Away from that locus it records the polystable shadow and forgets the extension class.  This is not merely a weakness of a particular PDE proof.  A positive harmonic metric on the original semistable object would force polystability, so a heat-flow extension that remembers nontrivial extension data must be filtered, renormalized, stacky, or singular in the parameter direction.  We make this formal in \cref{sec:heat-flow-geometricity}.

The same smooth harmonic metric supplies a second, complementary deformation parameter.  We next place the construction in the Hodge $\lambda$-family, which will also clarify the relation between the present analytic results and the broader literature.

The abstract differential formalism in \cite{AzamRayan2026} was developed for $\lambda$-$d$-connections.  Once the harmonic metric varies smoothly, the standard Hodge family becomes a smooth family over $U\times\C$.  In the usual pair notation we set
\begin{equation}\label{eq:intro-lambda}
  \ddbar_{E,\lambda}=\ddbar_E+\lambda\theta^{\dagger_h},
  \qquad
  D_{E,\lambda}=\lambda\partial_{E,h}+\theta.
\end{equation}
The first operator is a varying holomorphic structure and the second is a holomorphic $\lambda$-connection for that structure.  At $\lambda=0$ the pair is the original Higgs bundle; for $\lambda\ne0$ the rescaled ordinary connection
\[
  \ddbar_{E,\lambda}+\lambda^{-1}D_{E,\lambda}
\]
is flat.  We prove pullback functoriality and smooth dependence on $(u,\lambda)$ and explain how this supplies a diffeological Hodge enhancement on the stable locus.

We make no derived-twistor claim from this construction.  Recent work develops shifted and derived twistor structures in a different setting.  Here the result is the smooth finite-dimensional-parameter Hodge family supplied by the parametric harmonic metric.

\subsection{The Hodge direction and relation to earlier work}\label{sec:literature-positioning}

The analytic existence theory behind non-Abelian Hodge theory originates in the work of Hitchin, Donaldson, Corlette, Simpson, and the Hermitian--Einstein theory of stable bundles; see \cite{Hitchin1987,Donaldson1987,Corlette1988,Simpson1988,Simpson1992,UhlenbeckYau1986}.  The local calculation in this paper belongs to that lineage and should not be read as a claim that parameter dependence of harmonic metrics has not previously been studied.  Rather, our contribution is to formulate a direct theorem for arbitrary finite-dimensional smooth parameter plots, with pullback functoriality built into the statement, and then to connect the resulting analytic lift to the diffeological stacks of \cite{AzamRayan2026}.

Several results are close to the analytic questions considered here.  Kim and Wilkin prove analytic dependence of harmonic metrics for stable parabolic Higgs bundles as the parabolic weights and stable Higgs data vary \cite{KimWilkin2018}.  Wilkin's convergence theorem for the Yang--Mills--Higgs flow on Higgs bundles over a compact Riemann surface identifies the limiting critical point with the graded object associated to the Harder--Narasimhan--Seshadri filtration \cite{Wilkin2008}.  Hu, Shi, Sun, and Zuo derive explicit first-order equations for variation of the harmonic metric in their study of the Hitchin--Simpson correspondence and isomonodromic deformation \cite{HuShiSunZuo2025}.  In a different relative setting, Hu, Sun, Yang, and Zuo prove that the relative non-Abelian Hodge correspondence is real analytic near smooth points of relative moduli spaces \cite{HuSunYangZuo2026Relative}; their Appendix~A is particularly close in analytic spirit to the stable-family theorem below.  In the present paper $X$ is fixed, the parameter manifold is an arbitrary finite-dimensional smooth manifold rather than a complex-analytic moduli base, and the theorem is organized plotwise so that it defines a morphism on the diffeological family stack.  The novelty claimed here is the combination of that parameter-space analysis with the stacky extension framework, not the first regular-dependence result for harmonic metrics.

Recent work also addresses broader variation problems in non-Abelian Hodge theory.  Hitchin studies a universal Higgs-bundle moduli space over Teichm\"uller space \cite{Hitchin2026Universal}; Hu, Sun, Yang, and Zuo study higher-order isomonodromic deformation and harmonic-metric deformation equations \cite{HuSunYangZuo2026}; and Kryczka, Tanaka, and Yau develop Lagrangian correspondences and shifted twistor structures for derived stacks of Higgs and flat perfect complexes \cite{KryczkaTanakaYau2026}.  These papers use different parameter spaces and categorical structures.  We therefore distinguish throughout between proved smooth-family statements, stack-theoretic consequences, and open extensions.

A second point of comparison concerns the $\lambda$-direction.  Our earlier work on quiver connections allows $\lambda$ to be a parameter in a stack of connections \cite{AzamRayan2025}.  In the present paper we use the classical harmonic-bundle construction to produce a smooth Hodge family from a stable plot.  We do not claim a new derived twistor structure, and we keep this construction separate from the shifted twistor theory of \cite{KryczkaTanakaYau2026}.

We close the introduction with a guide to the successive analytic, stack-theoretic, and obstruction-theoretic parts of the paper.

\Cref{sec:relative-data} fixes relative Higgs and gauge-theoretic conventions.  \Cref{sec:analytic-setup} develops the Sobolev and parameter-dependent elliptic framework.  \Cref{sec:linearization} computes the metric-direction linearization of the Hitchin--Simpson equation and identifies its kernel.  \Cref{sec:smooth-metrics} proves \cref{thm:intro-smooth-metrics}.  \Cref{sec:stack-smoothness} deduces smoothness of the transform along diffeological plots.  \Cref{sec:first-variation} derives the first-variation and Green-operator formulas.  \Cref{sec:poly} treats locally split polystable strata.  \Cref{sec:filtrations} proves the relative harmonic filtration criterion, and \cref{sec:obstructions} develops the associated obstruction theory.  \Cref{sec:finite-regularity} returns to Question~5.3.4 of \cite{AzamRayan2026}: it proves a $C^d$ stable-family theorem, a $C^d$ filtration criterion, and an ambient pullback theorem for reduced singular parameter spaces, while isolating the nonreduced $C^\infty$-scheme problem.  \Cref{sec:heat-flow-geometricity} explains the same boundary in heat-flow terms and formalizes why ordinary positive harmonic-metric limits factor through the polystable associated graded.  \Cref{sec:lambda} constructs the $\lambda$-connection family.  \Cref{sec:examples} gives the explicit polystable failures, their weak $C^0$ operator lift, and the elliptic-curve degeneration from \cite{Simpson1994,AzamRayan2026}.  The appendices collect parameter-dependent elliptic facts, normalization details, and a dictionary with the notation of our earlier paper.  Taken together, these sections separate three issues that can otherwise be conflated: regular metric selection, relative assembly of subobjects and extensions, and functoriality in the parameter space.  This separation is essential when passing from stable to semistable families.

\subsection*{Acknowledgements}
The first-named author acknowledges support from an NSERC Canada Graduate Scholarship (Doctoral) and a Canadian Department of National Defence/NSERC Supplemental Funding Award.  The second-named author acknowledges support from an NSERC Discovery Grant.

\section{Relative Higgs data and gauge-theoretic conventions}\label{sec:relative-data}

\subsection{Relative Higgs families and numerical conditions}

Throughout the paper $X$ is a connected compact K\"ahler manifold of complex dimension $n$, with K\"ahler form $\omega$.  The volume is
\[
  \Vol_\omega(X)=\int_X\frac{\omega^n}{n!}.
\]
The parameter space $U$ is always a finite-dimensional smooth manifold, not necessarily compact.  We use the standard convention that manifolds are Hausdorff and second countable, hence paracompact; in particular, smooth bundles over $U\times X$ admit global smooth Hermitian metrics.  All assertions about families are local on $U$ unless explicitly stated otherwise.

Let $\pi_X:U\times X\to X$ and $\pi_U:U\times X\to U$ denote the projections.  Write $A^k_{X/U}$ for smooth complex-valued $k$-forms in the $X$-direction and $A^{p,q}_{X/U}$ for their type decomposition.  Thus
\[
  A^k_{X/U}=\Gamma\bigl(U\times X,\pi_X^*\Lambda^kT^*X\otimes_\R\C\bigr).
\]
The partial exterior derivative and partial Dolbeault operators are denoted
\[
  d_{X/U},\qquad \partial_{X/U},\qquad\ddbar_{X/U}.
\]
They differentiate only in the $X$-direction.  The parameter derivatives will be written explicitly in local coordinates on $U$.

\begin{definition}[Smooth relative Higgs bundle]\label{def:relative-higgs}
A smooth family of Higgs bundles on $X$ parametrized by $U$ consists of a smooth complex vector bundle $E\to U\times X$ and an operator
\[
  \cD''_E=\ddbar_E+\theta:
  A^0(E)\longrightarrow
  A^{0,1}_{X/U}(E)\oplus A^{1,0}_{X/U}(E)
\]
such that:
\begin{enumerate}[label=(\roman*)]
\item $\ddbar_E$ is a partial Dolbeault operator;
\item $\theta\in A^{1,0}_{X/U}(\End E)$ is $C^\infty(U\times X)$;
\item $\ddbar_E\theta=0$ and $\theta\wedge\theta=0$.
\end{enumerate}
Equivalently, the total operator $\cD''_E$ has square zero in the Higgs total complex.
\end{definition}

The notation follows the harmonic-bundle convention in which the Higgs operator is regarded as a degree-one differential.  On a slice $\{u\}\times X$ we write
\[
  (E_u,\cD''_u)=(E_u,\ddbar_{E,u},\theta_u).
\]

\begin{remark}
The family need not be globally isomorphic to the pullback of a fixed smooth bundle on $X$.  However, after shrinking $U$ about a chosen point, a smooth connection in the $U$-direction identifies nearby fibres of $E\to U\times X$ with a fixed smooth bundle over $X$.  Every analytic argument below is performed after such a local identification.  The final statements are independent of this choice because the constructions are equivariant under smooth complex gauge transformations.
\end{remark}

With the local nature of the family fixed, we can state the stability terminology.  It is deliberately fibrewise; no constancy of decomposition data is built into the definition.

\begin{definition}[Stable, polystable, semistable family]
A smooth relative Higgs bundle is called stable, polystable, or semistable when every slice has the corresponding property with respect to $\omega$.  This is only a fibrewise condition: it does not assert constancy of Jordan--H\"older type, automorphism group, or stable decomposition.
\end{definition}

The last sentence is essential in later sections.  A fibrewise condition is weaker than the existence of a relative filtration or splitting.

Before introducing the metric equation, we isolate the topological hypotheses under which a zero of the Hitchin--Simpson moment map is promoted to a flat harmonic metric.  This distinction will be used repeatedly in the nonlinear argument.

\label{sec:numerical-conditions}

To avoid repeatedly using the phrase ``the usual Chern conditions'', we fix a convention.  For a complex vector bundle $E$ on an $n$-dimensional compact K\"ahler manifold, put
\begin{equation}\label{eq:NAH-numerical}
  \nu_1(E)=\int_X c_1(E)\wedge\frac{\omega^{n-1}}{(n-1)!},
  \qquad
  \nu_2(E)=\int_X \operatorname{ch}_2(E)\wedge\frac{\omega^{n-2}}{(n-2)!}
\end{equation}
when $n\ge2$, and retain only $\nu_1$ when $n=1$.  We call
\begin{equation}\label{eq:NAH-condition}
  \nu_1(E)=0,
  \qquad
  \nu_2(E)=0
\end{equation}
the \emph{NAH numerical conditions}.  These are the projected characteristic-number conditions used in the compact K\"ahler correspondence; one may impose the stronger assumption that the relevant rational Chern classes vanish.  The use of $\operatorname{ch}_2$ rather than $c_2$ is convenient for the Chern--Weil identity.  Our earlier paper \cite{AzamRayan2026} phrases the condition as vanishing of the appropriate first and second characteristic projections along the K\"ahler class.  Nothing in the local implicit-function argument depends on which equivalent standard normalization is adopted; the numerical condition enters when a Hermitian--Yang--Mills--Higgs solution is promoted to a flat harmonic connection.

There are correspondingly two metric notions that should not be conflated.  A metric $h$ is a \emph{Hitchin--Simpson} or \emph{Hermitian--Yang--Mills--Higgs metric of slope zero} if
\begin{equation}\label{eq:HYM-Higgs-zero}
  \sqrt{-1}\,\Lambda_\omega
  \bigl(F_{D_h}+[\theta,\theta^{\dagger_h}]\bigr)=0.
\end{equation}
Under \eqref{eq:NAH-condition}, the standard Chern--Weil identity forces the remaining curvature components of the associated connection to vanish, so that the metric is harmonic in the flat sense.  This distinction is useful because stability is used analytically to solve \eqref{eq:HYM-Higgs-zero}, whereas the second numerical condition is used only in the final flatness step.

\subsection{Hermitian geometry and the harmonic equation}

Let $h$ be a smooth Hermitian metric on $E\to U\times X$.  For every $u$, the pair $(\ddbar_{E,u},h_u)$ determines a Chern connection
\[
  D_{h,u}=\partial_{E,h,u}+\ddbar_{E,u}.
\]
These slice-wise connections assemble to a partial connection
\[
  D_h=\partial_{E,h}+\ddbar_E
\]
in the $X$-direction.  Smoothness in $u$ follows from the local matrix formula.  If $H$ is the Hermitian matrix of $h$ in a smooth frame and $A^{0,1}$ is the matrix of $\ddbar_E$, then
\begin{equation}\label{eq:chern-matrix}
  A^{1,0}_h
  =H^{-1}\partial_{X/U}H-H^{-1}(A^{0,1})^{\dagger}H,
\end{equation}
with the convention that $\dagger$ denotes conjugate transpose of matrices and the form type is transformed accordingly.

The $h$-adjoint Higgs field is
\[
  \theta^{\dagger_h}\in A^{0,1}_{X/U}(\End E),
\]
and we put
\begin{equation}\label{eq:Dprime-Ddoubleprime}
  \cD''=\ddbar_E+\theta,
  \qquad
  \cD'_h=\partial_{E,h}+\theta^{\dagger_h}.
\end{equation}
The associated total connection is
\begin{equation}\label{eq:total-connection}
  \nabla_h=\cD''+\cD'_h
  =D_h+\theta+\theta^{\dagger_h}.
\end{equation}

\begin{definition}[Harmonic metric]\label{def:harmonic-metric}
A Hermitian metric $h$ on a Higgs bundle $(E,\cD'')$ is harmonic when the connection \eqref{eq:total-connection} is flat.  We write
\begin{equation}\label{eq:HS-equation}
  \mu(\cD'',h)
  :=\sqrt{-1}\,\Lambda_\omega
  \bigl(F_{D_h}+[\theta,\theta^{\dagger_h}]\bigr).
\end{equation}
A flat harmonic metric satisfies $\mu=0$.  Conversely, for an integrable Higgs bundle satisfying the NAH numerical conditions \eqref{eq:NAH-condition}, a slope-zero Hitchin--Simpson solution $\mu=0$ is flat by the standard Chern--Weil identity of the compact K\"ahler correspondence.
\end{definition}

For nonzero slope one subtracts the central Einstein constant.  Since the diffeological correspondence in \cite{AzamRayan2026} is formulated on the numerical locus leading to flat bundles, we work primarily with the degree-zero equation.  The local analytic argument itself applies to the central Hitchin--Simpson equation with only notational changes.

\subsection{Gauge symmetry, logarithmic coordinates, and normalization}

Fix a smooth complex vector bundle $E_0\to X$.  Its complex gauge group is
\[
  \Ggauge^{\C}=\Gamma(X,\operatorname{GL}(E_0)),
\]
and, relative to a background metric $\kappa$, its unitary gauge group is
\[
  \Ggauge_\kappa=\Gamma(X,U(E_0,\kappa)).
\]
For a gauge-theoretic Sobolev index $\ell>n+1$ we use the completions
\[
  \Ggauge^{\C}_{\ell+1}=W^{\ell+1,2}(X,\operatorname{GL}(E_0)),
  \qquad
  \Ggauge_{\kappa,\ell+1}=W^{\ell+1,2}(X,U(E_0,\kappa)).
\]
The inequality $\ell>n+1$ ensures that $W^{\ell,2}$ is a Banach
algebra and that $W^{\ell+1,2}$ gauge transformations are at least
$C^1$.  This gauge-completion index is independent of the $k$ used
later for the moment-map target.

The complex gauge action on a Higgs operator is
\[
  g\cdot\cD''=g\cD''g^{-1}.
\]
The action on metrics is chosen so that $g:(E,h)\to(E,g\cdot h)$ is an isometry:
\[
  g\cdot h(v,w)=h(g^{-1}v,g^{-1}w).
\]
In matrix notation, if $h$ is represented by $H$, then
\[
  g\cdot H=(g^{-1})^\dagger H g^{-1}.
\]
The harmonic equation is gauge equivariant.

For the implicit-function argument it is convenient to express changes of metric by positive endomorphisms.  This gives logarithmic coordinates on the space of metrics and makes the trace-free normalization linear.

Given two Hermitian metrics $h_0$ and $h$ on $E_0$, there is a unique positive $h_0$-self-adjoint endomorphism $f$ such that
\[
  h(v,w)=h_0(fv,w).
\]
We write $f=e^s$ with $s=s^{\dagger_{h_0}}$.  Thus metrics near $h_0$ are parametrized by Hermitian endomorphisms.  If the determinant is fixed, then
\[
  \tr(s)=0.
\]
Let
\[
  \Herm_{h_0}(\End E_0)
  =\set{s\in\End E_0:s^{\dagger_{h_0}}=s},
\]
and similarly $\Herm_{h_0}^0$ for the trace-free part.

We use the exponential coordinate $h_s=h_0e^s$.  With the gauge action
on metrics fixed above, the same metric is written symmetrically as
\[
  h_s=e^{-s/2}\cdot h_0,
\]
because $e^{-s/2}:(E,h_0)\to(E,h_s)$ acts by pullback and $s$ is
$h_0$-self-adjoint.  The minus sign is forced by our convention
$g\cdot H=(g^{-1})^\dagger Hg^{-1}$.  The exponential description is
used for Banach charts, while the equivalent symmetric description is
useful for gauge covariance.

The remaining scalar freedom is best handled on the determinant line.  We fix that normalization now so that the later Jacobi operator is invertible on precisely the intended trace-free slice.

Stable harmonic metrics are unique only up to multiplication by a positive scalar.  A relative theorem therefore requires a normalization.  There are several equivalent choices.

\begin{definition}[Integral normalization]\label{def:integral-normalization}
Fix a smooth background metric $\kappa$ on $E$.  Define
\[
  N_\kappa(h)_u
  =\frac{1}{r\Vol_\omega(X)}
  \int_X \log\det(\kappa_u^{-1}h_u)\,\frac{\omega^n}{n!}.
\]
A Hermitian metric $h$ is integrally normalized relative to $\kappa$ if
\begin{equation}\label{eq:integral-normalization}
  N_\kappa(h)_u=0
\end{equation}
for every $u$.
\end{definition}

Alternatively, one may prescribe the determinant metric.  The latter is better suited to the trace-free implicit-function theorem.

\begin{definition}[Determinant normalization]\label{def:det-normalization}
Let $q$ be a smooth Hermitian metric on $\det E\to U\times X$.  A metric $h$ on $E$ is $q$-normalized if
\[
  \det h=q.
\]
\end{definition}

When the determinant has degree zero, equivalently $\nu_1(E_u)=0$ in the notation of \cref{sec:numerical-conditions}, a Hermitian--Einstein determinant metric exists on each fibre.  We shall prove in \cref{sec:smooth-metrics} that, for a family over a smooth parameter manifold, these metrics may be chosen as one smooth metric on the determinant line over all of $U\times X$.  Prescribing this determinant metric reduces the nonlinear equation to its trace-free part.

The term ``normalization'' refers precisely to the scalar ambiguity of a stable harmonic metric.  If $h_1$ and $h_2$ are harmonic metrics on the same stable Higgs bundle, then $h_2=c h_1$ for a positive constant $c$ on the connected manifold $X$, and therefore
\[
  \det h_2=c^r\det h_1.
\]
Thus two harmonic metrics with the same prescribed determinant are equal.  Conversely, if $q$ is a degree-zero Hermitian--Einstein metric on $\det E$, then $q$ and the determinant of any harmonic metric solve the same abelian equation; their ratio is constant on each $X$-slice, and a unique positive scalar rescaling makes the determinant equal to $q$.  An arbitrary metric on $\det E$ need not be admissible in this sense, so whenever determinant normalization is used below, $q$ is understood to be a chosen Hermitian--Einstein determinant metric.  Integral normalization removes the same scalar freedom because
\[
  N_\kappa(c h)=N_\kappa(h)+\log c.
\]
Appendix~\ref{app:normalization} gives the family-level comparison of these two normalizations and explains why they determine the same relative flat connection.

\subsection{Deformation complexes}

The gauge action is reflected infinitesimally by the standard Dolbeault and de Rham deformation complexes.  Recording them here fixes the cohomological language used later for the differential of the transform and for the obstruction theory of invariant subbundles.

For a fixed Higgs bundle $(E,\cD'')$, let
\[
  \cC^\bullet_{\Dol}(E):
  A^0(\End E)
  \xto{\cD''}
  A^1_{\Dol}(\End E)
  \xto{\cD''}
  A^2_{\Dol}(\End E)\to\cdots
\]
be the total Higgs complex, where
\[
  A^j_{\Dol}(\End E)
  =\bigoplus_{p+q=j}A^{p,q}(\End E)
  =\bigoplus_{p+q=j}
  \Gamma^\infty\!\left(
    X,\End E\otimes\Lambda^{p,q}T^*X
  \right).
\]
In degree one,
\[
  A^1_{\Dol}(\End E)
  =A^{0,1}(\End E)\oplus A^{1,0}(\End E).
\]
An infinitesimal deformation is written
\[
  \eta=a+\varphi,
  \qquad
  a\in A^{0,1}(\End E),\quad
  \varphi\in A^{1,0}(\End E),
\]
and satisfies
\begin{equation}\label{eq:linearized-Higgs}
  \cD''\eta=0.
\end{equation}
The infinitesimal complex-gauge directions in this space of closed deformations are of the form
\[
  \eta=\cD''\xi,
\]
where
\[
  \xi\in A^0(\End E)=\Gamma^\infty(X,\End E)
\]
is an infinitesimal complex gauge parameter.  With the action convention
$g\cdot\cD''=g\cD''g^{-1}$ used above, the path $g_t=e^{-t\xi}$ adds
$\cD''\xi$ to first order; the path $g_t=e^{t\xi}$ instead gives
$-\cD''\xi$.  At coefficient level $W^{\ell,2}$, the
Sobolev-completed gauge complex takes
$\xi\in W^{\ell+1,2}(\End E)$.

For a flat connection $\nabla$, the de Rham deformation complex is
\[
  \cC^\bullet_{\dR}(E,\nabla):
  A^0(\End E)
  \xto{d_\nabla}
  A^1(\End E)
  \xto{d_\nabla}
  A^2(\End E)\to\cdots.
\]
The first hypercohomology of the Higgs complex and the first de Rham cohomology control first-order deformations modulo gauge.

\subsection{Pullback and functoriality of relative objects}

Let $f:V\to U$ be smooth.  Pullback gives
\[
  f_X=f\times\Id_X:V\times X\to U\times X
\]
and a pulled-back family
\[
  f_X^*(E,\cD'').
\]
A metric $h$ pulls back to $f_X^*h$.  Since all operators in \eqref{eq:HS-equation} differentiate only in the $X$-direction,
\begin{equation}\label{eq:pullback-moment}
  \mu(f_X^*\cD'',f_X^*h)=f_X^*\mu(\cD'',h).
\end{equation}
This elementary identity is the source of the functoriality that will later turn the local analytic construction into a morphism on diffeological plots.

\begin{proposition}[Pullback of normalized solutions]\label{prop:pullback-solutions}
Suppose $h$ is a harmonic metric on a smooth family $(E,\cD'')$ over $U$.  Then $f_X^*h$ is harmonic for the pulled-back family.  If the normalization is prescribed by a determinant metric $q$, then $f_X^*h$ is normalized by $f_X^*q$.  If the integral normalization \eqref{eq:integral-normalization} is used with a pulled-back background metric, it is likewise preserved.
\end{proposition}

\begin{proof}
Work in a local frame of $E$ and write $H(u,x)$ for the matrix of
$h$.  Since $f_X$ is the identity in the $X$-direction, relative
differentiation commutes with pullback.  In particular,
\[
 (f_X^*H)^{-1}\partial_X(f_X^*H)
 =f_X^*(H^{-1}\partial_XH),
\]
so the relative Chern connection of $f_X^*h$ is the pullback of the
relative Chern connection of $h$.  The adjoint of the pulled-back Higgs
field is likewise the pullback of $\theta^{\dagger_h}$, because the
adjoint is characterized pointwise by the metric.  Consequently
\[
  \nabla_{f_X^*h}=f_X^*\nabla_h.
\]
Curvature commutes with pullback, and hence flatness of $\nabla_h$
implies flatness of $\nabla_{f_X^*h}$.  Equivalently,
\eqref{eq:pullback-moment} carries the moment-map equation to the
pulled-back family.

For the determinant normalization one has
\[
 \det(f_X^*h)=f_X^*(\det h)=f_X^*q.
\]
For the integral normalization, if $\kappa$ is the background metric,
then for every $v\in V$,
\[
 \int_X\log\det\bigl((f_X^*\kappa)^{-1}f_X^*h\bigr)
             \frac{\omega^n}{n!}
 =\int_X\log\det(\kappa^{-1}h)(f(v),x)
             \frac{\omega^n}{n!}=0.
\]
Thus both normalizations are preserved, with no change-of-variables
factor because the map on each $X$-slice is the identity.
\end{proof}

\section{Analytic set-up for parameter-dependent harmonic metrics}\label{sec:analytic-setup}

\subsection{Local trivializations and normalized metric spaces}

Fix $u_0\in U$.  Choose a coordinate ball $B\subset U$ about $u_0$ and a connection on $E$ in the $U$-direction.  Parallel transport along radial paths identifies
\[
  E|_{B\times X}\cong \pi_X^*E_0,
  \qquad E_0=E_{u_0},
\]
as smooth complex bundles.  Under this identification the family of Higgs operators becomes a smooth map
\begin{equation}\label{eq:parameter-map-D}
  B\longrightarrow \cZ^{k+1}_{\mathrm{Higgs}}(E_0)
  \subset \cA^{k+1}_{\mathrm{amb}}(E_0),
  \qquad
  u\longmapsto\cD''_u,
\end{equation}
where the ambient affine coefficient space is
\begin{equation}\label{eq:ambient-Higgs-affine-space}
  \cA^{k+1}_{\mathrm{amb}}(E_0)
  =\cD''_0+
  W^{k+1,2}\!\left(
    X,
    \bigl(\Lambda^{0,1}T^*X\oplus\Lambda^{1,0}T^*X\bigr)
    \otimes\End E_0
  \right).
\end{equation}
Thus a point of the ambient space is written
\[
  \cD''_0+a^{0,1}+\varphi,
\]
and allows both the partial Dolbeault operator and the Higgs field to vary.  The space
\begin{equation}\label{eq:integrable-Higgs-locus}
  \cZ^{k+1}_{\mathrm{Higgs}}(E_0)
  =\set{\cD''\in\cA^{k+1}_{\mathrm{amb}}(E_0):(\cD'')^2=0}
\end{equation}
is the integrable Higgs locus.  It need not be a Banach submanifold of
the ambient affine space.  The $k+1$ coefficient regularity is used
because the moment map differentiates the varying Dolbeault coefficient
once and takes values in $W^{k,2}$.  In particular,
\eqref{eq:parameter-map-D} is not merely a family of elements of
$W^{k+1,2}(\End E_0\otimes\Omega_X^{1,0})$ unless the holomorphic
structure has separately been fixed.
If $\ddbar_E$ is fixed and only the Higgs field varies, the Higgs-only
slice may use $W^{k,2}(\End E_0\otimes\Omega_X^{1,0})$, because the
moment map is algebraic in $\theta$.  The full family problem allows the
Dolbeault coefficient to vary, so we use the uniform $W^{k+1,2}$
coefficient space above.

We make the regularity precise.  Let $m=\dim_\R X=2n$ and choose an integer
\begin{equation}\label{eq:k-assumption}
  k>\frac m2+1=n+1.
\end{equation}
Then multiplication
\[
  W^{k,2}\times W^{k,2}\to W^{k,2}
\]
is continuous, and $W^{k+1,2}\hookrightarrow C^1$.  Smooth family data define a $C^\infty$ map from $B$ into every $W^{k,2}$ completion.

After the local identification of the underlying bundle, the metric variable can be placed in a fixed Sobolev space.  The determinant condition then cuts out the normalized trace-free slice used by the implicit-function theorem.

Fix a smooth background Hermitian metric $h_0$ on $E_0$.  Let
\[
  \cS_{k+2}^0
  =W^{k+2,2}\bigl(X,\Herm^0_{h_0}(\End E_0)\bigr).
\]
For $s$ in a sufficiently small open ball about zero, set
\[
  h_s=h_0e^s.
\]
The map $s\mapsto h_s$ is a smooth chart on the Banach manifold of $W^{k+2,2}$ Hermitian metrics with determinant $\det h_0$.

Let
\[
  \cT_k^0
  =W^{k,2}\bigl(X,\Herm^0_{h_0}(\End E_0)\bigr).
\]
A trace-free projection will be used repeatedly.  For an endomorphism $A$, put
\begin{equation}\label{eq:tracefree-projection}
  \operatorname{pr}_0(A)
  =A-\frac1r\tr(A)\Id,
\end{equation}
and define the trace-free moment map by
\begin{equation}\label{eq:tracefree-moment-map}
  \mu_0(\cD'',h)
  :=\operatorname{pr}_0\bigl(\mu(\cD'',h)\bigr).
\end{equation}
A small identification issue is important here.  The endomorphism $\mu_0(\cD'',h_s)$ is self-adjoint with respect to the \emph{moving} metric $h_s$, not with respect to $h_0$.  Since
\[
  h_s(v,w)=h_0(e^{s/2}v,e^{s/2}w),
\]
the isometry $e^{s/2}:(E,h_s)\to(E,h_0)$ identifies the moving Hermitian endomorphisms with the fixed $h_0$-Hermitian space.  We therefore set
\[
  \mathcal I_s(A)=e^{s/2}Ae^{-s/2}
\]
and define
\begin{equation}\label{eq:nonlinear-F}
  \cF:B\times\cS_{k+2}^0\longrightarrow\cT_k^0,
  \qquad
  \cF(u,s)=\mathcal I_s\!\left(\mu_0(\cD''_u,h_s)\right).
\end{equation}
The zero set is unchanged by $\mathcal I_s$.  Moreover, at a zero of the moment map the derivative of the conjugating factor contributes nothing, so the metric-direction linearization remains the Jacobi operator computed below.

\begin{lemma}[Smoothness of the nonlinear map]\label{lem:F-smooth}
The map \eqref{eq:nonlinear-F} is $C^\infty$ between Banach manifolds.  If $u\mapsto\cD''_u$ is real analytic as a map into the Sobolev affine space, then $\cF$ is real analytic.
\end{lemma}

\begin{proof}
In a smooth frame the Chern connection matrix is given by
\eqref{eq:chern-matrix}.  The dependence on $s$ is generated by the maps
\[
  s\mapsto e^s,
  \quad
  s\mapsto e^{-s},
  \quad
  s\mapsto \partial e^s,
\]
combined with multiplication, inversion, differentiation, and
contraction by the fixed K\"ahler form.  Because $k>m/2$, each
$W^{j,2}$ with $j\ge k$ is a Banach algebra.  The power series for the
exponential therefore converges locally in $W^{k+2,2}$ and defines a
real-analytic map; inversion is real analytic on the open set of
invertible $W^{k+2,2}$ endomorphisms.  Hence
\[
 s\longmapsto H_s^{-1}\partial H_s
\]
is analytic from $W^{k+2,2}$ to $W^{k+1,2}$.  Applying one further
$X$-derivative, and adding the quadratic connection term, places the
Chern curvature in $W^{k,2}$.  The commutator
$[\theta,\theta^{\dagger_{h_s}}]$ is algebraic in the coefficient of
$\theta$, $H_s$, and $H_s^{-1}$, so Sobolev multiplication gives the
same target regularity.  The coefficient dependence on $\cD''_u$ is
polynomial apart from these already controlled metric operations.  More
precisely, the Chern formula is algebraic in the $(0,1)$ coefficient
$A^{0,1}$, but its curvature differentiates that coefficient once.
Thus
\[
 W^{k+1,2}\ni A^{0,1}\longmapsto F_{D_{h_s}}in W^{k,2}
\]
has the required mapping property.  This is the reason for using
$\cA^{k+1}$, rather than $\cA^k$, in the universal Banach map.

Finally, $s\mapsto\mathcal I_s$ and the trace-free projection are,
respectively, analytic and bounded linear maps on the relevant
Sobolev spaces.  Composing these maps proves that $\cF$ is smooth.  If
$u\mapsto\cD''_u$ is real analytic, every operation just listed is
real analytic, which proves the final assertion.  Notice that this
argument takes place on the ambient coefficient space; integrability
of $\cD''_u$ is not needed for the mapping-property statement.
\end{proof}

We now package the preceding choices into a single nonlinear map with
fixed domain and target, separating the Higgs data from the metric.  We
abbreviate the ambient affine space
\eqref{eq:ambient-Higgs-affine-space} to $\cA^{k+1}$; no
Maurer--Cartan equation is imposed on this space.  Define the ambient map
\begin{equation}\label{eq:universal-moment-map}
  \Phi:\cA^{k+1}\times\cS_{k+2}^0\to\cT_k^0,
  \qquad
  \Phi(\cD'',s)=\mathcal I_s\!\left(\mu_0(\cD'',h_0e^s)\right).
\end{equation}
For an actual integrable family one has $\cF(u,s)=\Phi(\cD''_u,s)$.  The derivative at a harmonic point decomposes as
\begin{equation}\label{eq:derivative-splitting}
  D\Phi_{(\cD'',0)}(\eta,s)
  =\cS_h(\eta)+L_hs.
\end{equation}
Here $\cS_h$ is the fixed-metric source operator and $L_h$ is the metric-direction Jacobi operator.  The next section identifies both terms.

\subsection{Uniform elliptic control in parameter families}

Suppose $P_u$ is a smooth family of second-order differential operators on a fixed vector bundle over compact $X$, with principal symbols satisfying a uniform strong ellipticity estimate on a relatively compact parameter neighbourhood $B'\Subset B$.  Then for every integer $j\ge0$ there is a constant $C_j$ such that
\begin{equation}\label{eq:uniform-elliptic-estimate}
  \norm{v}_{W^{j+2,2}}
  \le C_j\bigl(\norm{P_uv}_{W^{j,2}}+\norm{v}_{L^2}\bigr)
\end{equation}
for all $u\in B'$.

The next proposition is the parameter consequence used below.

\begin{proposition}[Parameter-dependent elliptic bootstrap]\label{prop:parametric-bootstrap}
Let $B\subset\R^d$ be open, $P_u$ a $C^\infty$ family of strongly elliptic second-order operators on a vector bundle $V\to X$, and $f:B\to W^{k,2}(V)$ a $C^\infty$ map.  Suppose $v:B\to W^{k+2,2}(V)$ is $C^\infty$ and
\[
  P_uv(u)=f(u).
\]
If the coefficients of $P_u$ and $f(u)$ are jointly smooth on $B\times X$, then $v$ is jointly smooth on $B\times X$.

If the dependence on $u$ is real analytic in Sobolev topology and $P_u$ is invertible with locally uniformly bounded inverse, then $u\mapsto v(u)$ is real analytic in each Sobolev topology.
\end{proposition}

\begin{proof}
Fix $B'\Subset B$.  On $B'$ the coefficients and all of their parameter
derivatives are bounded in the Sobolev norms under consideration, and
the constants in \eqref{eq:uniform-elliptic-estimate} may be chosen
uniformly.  Differentiating once in a parameter direction gives
\[
  P_u(\partial_{u_i}v)
  =\partial_{u_i}f-(\partial_{u_i}P_u)v.
\]
More generally, for a multi-index $\alpha$,
\[
 P_u(\partial_u^\alpha v)
 =\partial_u^\alpha f
  -\sum_{0<\beta\le\alpha}\binom{\alpha}{\beta}
    (\partial_u^\beta P_u)
    (\partial_u^{\alpha-\beta}v).
\]
Assume inductively that all lower parameter derivatives on the right
have the desired $X$-regularity.  The coefficient multiplication
theorems and \eqref{eq:uniform-elliptic-estimate} then give two more
$X$-derivatives for $\partial_u^\alpha v$.  A second induction on the
Sobolev order proves that $u\mapsto v(u)$ is $C^\infty$ with values in
$W^{\ell,2}$ for every $\ell$.  Taking $\ell$ larger than
$m/2+a$ and using Sobolev embedding shows that every mixed derivative
$\partial_u^\alpha\partial_x^\gamma v$ of order $|\gamma|\le a$ is
continuous.  Since $a$ and $\alpha$ are arbitrary, $v$ is jointly
smooth on $B'\times X$, and hence on $B\times X$.

For the analytic assertion, fix a Sobolev level and regard
\[
 P_u:W^{\ell+2,2}(V)\longrightarrow W^{\ell,2}(V).
\]
Local uniform invertibility and the Neumann-series identity show that
$u\mapsto P_u^{-1}$ is analytic as a bounded-operator-valued map.  Thus
$v(u)=P_u^{-1}f(u)$ is analytic at this Sobolev level.  Repeating the
argument after elliptic bootstrapping gives analyticity into every
$W^{\ell,2}$; the solutions obtained at different levels agree by
uniqueness.  This proves both claims.
\end{proof}

A more detailed proof, including the non-invertible Fredholm case with a fixed kernel projection, is given in \cref{app:elliptic}.

Uniform elliptic estimates have an immediate consequence for nearby parameters: once the kernel vanishes at one point, invertibility persists after shrinking the parameter neighbourhood.  This is the bridge from the general elliptic set-up to the stable Higgs locus.

Let $L_0:W^{k+2,2}(V)\to W^{k,2}(V)$ be an isomorphism.  If $L_u$ depends continuously on $u$ in operator norm, then $L_u$ remains invertible for $u$ near $0$ and
\begin{equation}\label{eq:inverse-neumann}
  L_u^{-1}
  =L_0^{-1}\sum_{j=0}^\infty
  \bigl(-(L_u-L_0)L_0^{-1}\bigr)^j
\end{equation}
whenever the norm of $(L_u-L_0)L_0^{-1}$ is less than one.  We use this repeatedly on the stable normalized locus, where the linearization has no kernel and no moving finite-dimensional obstruction space.

The implicit-function theorem requires the $s$-derivative of \eqref{eq:nonlinear-F} to be an isomorphism.  On a stable Higgs bundle, harmonic endomorphisms are scalar, so trace-free normalization gives invertibility.  On a polystable bundle with decomposition
\[
  E\cong\bigoplus_j E_j\otimes W_j,
\]
the harmonic endomorphisms contain
\[
  \bigoplus_j \Id_{E_j}\otimes\End(W_j).
\]
Thus the linearization has a kernel whose dimension can jump in families.  This is not a technical nuisance.  It is the analytic shadow of stabilizer stratification in the moduli stack.  We therefore resist the temptation to state a single parametric theorem across the entire polystable locus.

\section{Linearization of the Hitchin--Simpson equation}\label{sec:linearization}

\subsection{Conventions and first-order variation formulas}

We now work at a fixed harmonic bundle $(E,\cD'',h)$ on $X$.  The operators
\[
  \cD''=\ddbar_E+\theta,
  \qquad
  \cD'_h=\partial_{E,h}+\theta^{\dagger_h}
\]
act on the total Higgs complex.  We use the graded commutator
\[
  [A,B]_{\mathrm{gr}}
  =AB-(-1)^{\deg A\deg B}BA.
\]
Since $\cD''$ and $\cD'_h$ have odd degree, their graded commutator is the anticommutator
\[
  [\cD'',\cD'_h]_{\mathrm{gr}}
  =\cD''\cD'_h+\cD'_h\cD''.
\]
At a harmonic metric the total connection $\nabla_h=\cD''+\cD'_h$ is flat, and therefore
\begin{equation}\label{eq:harmonic-bicomplex}
  (\cD'')^2=0,
  \qquad
  (\cD'_h)^2=0,
  \qquad
  [\cD'',\cD'_h]_{\mathrm{gr}}=0.
\end{equation}

The moment map can be represented, on endomorphism-valued data, as
\begin{equation}\label{eq:moment-graded}
  \mu(\cD'',h)
  =\sqrt{-1}\,\Lambda
  [\cD'',\cD'_h]_{\mathrm{gr}}.
\end{equation}
This agrees with \eqref{eq:HS-equation} after using Higgs integrability and the standard type decomposition.  Our normalization of $\Lambda$ is the adjoint of wedge multiplication by $\omega$.

With the graded conventions fixed, we can compute how the adjoint pieces vary.  The sign in the Chern part is important and is the source of the companion deformation used throughout the first-variation formulas.

Let $s=s^{\dagger_h}$ and set
\[
  h_t(v,w)=h(e^{ts}v,w).
\]
We shall use the next variation formula repeatedly.

\begin{lemma}[Metric variation of $\cD'_h$]\label{lem:variation-Dprime}
At $t=0$,
\begin{equation}\label{eq:variation-Dprime}
  \frac{d}{dt}\bigg|_{t=0}\cD'_{h_t}
  =\cD'_h s,
\end{equation}
where the right side denotes the endomorphism-valued degree-one form obtained by applying the induced operator $\cD'_h$ to $s$.
\end{lemma}

\begin{proof}
Choose a local holomorphic frame for $\ddbar_E$ and let $H$ denote the
matrix of $h$.  With the
present convention the matrix of $h_t$ is $H_t=He^{ts}$, while the
$(1,0)$ Chern connection matrix is $H_t^{-1}\partial H_t$.  Therefore
\begin{align*}
 \left.\frac d{dt}\right|_0(H_t^{-1}\partial H_t)
 &= -sH^{-1}\partial H
    +H^{-1}\left.\frac d{dt}\right|_0\partial(He^{ts})\\
 &=\partial s+[H^{-1}\partial H,s]
 =\partial_{E,h}s.
\end{align*}
Equivalently, in invariant notation,
\[
  \frac{d}{dt}\bigg|_0\partial_{E,h_t}=\partial_{E,h}s.
\]
For the Higgs adjoint, the defining identity
\[
  h_t(\theta v,w)=h_t(v,\theta^{\dagger_{h_t}}w)
\]
is equivalent in this frame to
$\theta^{\dagger_{h_t}}=e^{-ts}\theta^{\dagger_h}e^{ts}$.  Hence
\[
  \frac{d}{dt}\bigg|_0\theta^{\dagger_{h_t}}
  =[\theta^{\dagger_h},s].
\]
Both formulas are tensorial, so the frame computation is global.
Adding the two degree-one terms yields
\[
  \partial_{E,h}s+[\theta^{\dagger_h},s]=\cD'_hs.
\]
\end{proof}

\subsection{The Jacobi operator and its kernel}

Differentiate \eqref{eq:moment-graded} in the metric direction.  Using \cref{lem:variation-Dprime}, we obtain
\begin{equation}\label{eq:metric-linearization-raw}
  \frac{d}{dt}\bigg|_0\mu(\cD'',h_t)
  =\sqrt{-1}\Lambda\cD''\cD'_h s.
\end{equation}
The right side is the Higgs Laplacian on zero-forms.

\begin{proposition}[Jacobi operator]\label{prop:Jacobi}
At a harmonic metric define
\begin{equation}\label{eq:Jacobi-definition}
  L_hs
  :=\sqrt{-1}\Lambda\cD''\cD'_hs.
\end{equation}
Then
\begin{equation}\label{eq:Jacobi-Laplacian}
  L_h=(\cD'')^{*h}\cD''
\end{equation}
on endomorphism-valued zero-forms.  In particular, $L_h$ is nonnegative, self-adjoint, and elliptic with scalar principal symbol $\abs{\xi}^2\Id$ up to the universal Laplacian normalization.
\end{proposition}

\begin{proof}
The K\"ahler identity for the harmonic-bundle operators is
\[
  (\cD'')^{*h}=-\sqrt{-1}[\Lambda,\cD'_h].
\]
Apply this identity to the one-form $\cD''s$.  The operator $\Lambda$
lowers degree by two, and therefore $\Lambda\cD''s=0$.  It follows that
\[
  (\cD'')^{*h}\cD''s
  =-\sqrt{-1}\Lambda\cD'_h\cD''s.
\]
Flatness of the harmonic connection gives the anticommutation relation
in \eqref{eq:harmonic-bicomplex}, so
\[
  \cD'_h\cD''s=-\cD''\cD'_hs,
\]
and hence the last display is
$\sqrt{-1}\Lambda\cD''\cD'_hs=L_hs$.  This proves
\eqref{eq:Jacobi-Laplacian}.  An operator of the form $A^*A$ is
self-adjoint and nonnegative on its natural $L^2$ domain.  Its
second-order part is the ordinary Dolbeault Laplacian on $\End E$;
the Higgs commutators are of lower order.  Thus its principal symbol is
the positive scalar Laplace symbol, which proves ellipticity and the
stated principal-symbol normalization.
\end{proof}

\begin{corollary}[Energy identity]\label{cor:energy-identity}
For every smooth endomorphism $s$,
\begin{equation}\label{eq:energy-identity}
  \int_X\ip{L_hs}{s}_h\,\frac{\omega^n}{n!}
  =\norm{\cD''s}_{L^2(h,\omega)}^2.
\end{equation}
Hence
\[
  \ker L_h=\ker\cD''.
\]
\end{corollary}

\begin{proof}
By \cref{prop:Jacobi}, the Jacobi operator on zero-forms is
\[
  L_h=(\cD'')^{*h}\cD''.
\]
Here the adjoint is taken with respect to the $L^2$ inner product determined by $h$ and $\omega$.  Since $X$ is compact, there is no boundary contribution when the derivative is transferred to the first factor.  Thus
\[
  \int_X\ip{L_hs}{s}_h\,\frac{\omega^n}{n!}
  =\int_X\ip{\cD''s}{\cD''s}_h\,\frac{\omega^n}{n!},
\]
which is \eqref{eq:energy-identity}.  If $L_hs=0$, the left side
vanishes, hence $\|\cD''s\|_{L^2}=0$ and therefore $\cD''s=0$.
Conversely, $\cD''s=0$ immediately gives $L_hs=0$.  This proves the
kernel identity for smooth sections.  The energy identity extends by
density to $W^{1,2}$ sections.  In particular, a Sobolev kernel element
is weakly $\cD''$-closed and is smooth by elliptic regularity, so the
same kernel identity holds for the Sobolev realizations used below.
\end{proof}

The kernel determines whether the Jacobi operator can be inverted on the normalized slice.  We identify it with infinitesimal Hermitian automorphisms and then use stability to eliminate it.

The equation $\cD''s=0$ means that $s$ is a Higgs endomorphism.  If $s$
is Hermitian, it is therefore an infinitesimal Hermitian symmetry of
the harmonic bundle; it need not itself be invertible.

\begin{proposition}[Kernel on the stable locus]\label{prop:stable-kernel}
Let $(E,\cD'')$ be stable and $h$ harmonic.  Then
\[
  \ker\bigl(L_h:\Herm_h(\End E)\to\Herm_h(\End E)\bigr)
  =\R\Id_E.
\]
Consequently
\begin{equation}\label{eq:tracefree-invertible}
  L_h:
  W^{k+2,2}(\Herm_h^0\End E)
  \xrightarrow{\cong}
  W^{k,2}(\Herm_h^0\End E)
\end{equation}
is an isomorphism.
\end{proposition}

\begin{proof}
Let $s\in\Herm_h(\End E)$ satisfy $L_hs=0$.  By
\cref{cor:energy-identity}, $\cD''s=0$, so $s$ is a holomorphic Higgs
endomorphism of $(E,\cD'')$.  Stability implies simplicity.  For
completeness, if $f$ is a Higgs endomorphism, its kernel and saturated
image are Higgs subsheaves.  If both were nontrivial, stability together
with additivity of rank and degree would force the image to have slope
simultaneously less than and greater than the slope of $E$, a
contradiction.  Applying this observation to $f-\lambda\Id$ for a root
$\lambda$ of the characteristic polynomial, whose coefficients are
constant holomorphic functions on compact $X$, shows that
$f=\lambda\Id$.  Thus every complex Higgs endomorphism is scalar.
Hence $s=\lambda\Id_E$ for some $\lambda\in\C$.  Because $s$ is
$h$-Hermitian, $\lambda=\overline\lambda$, and therefore
$\lambda\in\R$.  This proves the kernel statement.

Restrict now to trace-free Hermitian endomorphisms.  The scalar kernel
disappears, so the restricted operator is injective.  It remains a
second-order self-adjoint elliptic operator with the same principal
symbol on the trace-free summand.  Consequently
\[
  L_h:W^{k+2,2}(\Herm_h^0\End E)
  \longrightarrow W^{k,2}(\Herm_h^0\End E)
\]
is Fredholm.  Self-adjointness identifies the cokernel with the kernel
of the adjoint, which is again the kernel of $L_h$ on the trace-free
subspace.  That kernel is zero, so the cokernel is zero as well.  Thus
the map is an isomorphism.  Elliptic regularity ensures that any Sobolev
kernel element is smooth, so no additional weak kernel appears.
\end{proof}

\begin{remark}
The use of stability occurs exactly here.  The nonlinear implicit-function argument itself does not know about slope stability.  Stability is converted, through simplicity, into invertibility of the Jacobi operator after normalization.
\end{remark}

\subsection{The fixed-metric source operator and local formulas}

Let $\eta$ be a degree-one variation of $\cD''$.  In degree one write
\[
  \eta=a+\varphi,
  \qquad
  a\in A^{0,1}(\End E),
  \quad
  \varphi\in A^{1,0}(\End E).
\]
A small but important sign enters here.  With the Hermitian metric $h$ held fixed, variation of the Dolbeault operator by $a$ changes the $(1,0)$ Chern part by $-a^{\dagger_h}$, whereas variation of the Higgs field by $\varphi$ changes its adjoint by $\varphi^{\dagger_h}$.  Thus the fixed-metric variation of the companion operator $\cD'_h$ is
\begin{equation}\label{eq:star-companion}
  \eta^{\star_h}
  :=-a^{\dagger_h}+\varphi^{\dagger_h}.
\end{equation}
The symbol $\star_h$ distinguishes this companion variation from the ordinary adjoint of the total inhomogeneous form $a+\varphi$.

\begin{definition}[Source operator]\label{def:source-operator}
At a harmonic bundle $(E,\cD'',h)$ define
\begin{equation}\label{eq:source-operator}
  \cS_h(\eta)
  :=\sqrt{-1}\Lambda\Bigl(
  [\eta,\cD'_h]_{\mathrm{gr}}
  +[\cD'',\eta^{\star_h}]_{\mathrm{gr}}
  \Bigr)_0,
\end{equation}
where the subscript $0$ denotes the trace-free part.
\end{definition}

Because both entries in each commutator have odd total degree, the brackets in \eqref{eq:source-operator} are graded anticommutators.  Gauge covariance of the moment map shows that the resulting degree-zero expression is Hermitian; the final projection makes it trace-free.

\begin{proposition}[Full linearization]\label{prop:full-linearization}
Let $\Phi$ be the universal trace-free moment map \eqref{eq:universal-moment-map}.  At a normalized harmonic point,
\begin{equation}\label{eq:full-linearization}
  D\Phi_{(\cD'',0)}(\eta,s)
  =\cS_h(\eta)+L_hs.
\end{equation}
\end{proposition}

\begin{proof}
First vary the two arguments independently.  For the metric argument,
\cref{lem:variation-Dprime,prop:Jacobi} gives
$\left.\frac d{dt}\right|_0\mu_0(\cD'',he^{ts})=L_hs$; the operator
preserves the trace-free Hermitian summand.  For the data argument, hold
$h$ fixed and write
\[
  \cD''_t=\cD''+t\eta+O(t^2).
\]
The Chern compatibility relation and the definition of the Higgs adjoint give
\[
  \cD'_{h,t}
  =\cD'_h+t\eta^{\star_h}+O(t^2).
\]
Differentiating the graded commutator in \eqref{eq:moment-graded}
therefore gives
\[
 \sqrt{-1}\Lambda\bigl(
 [\eta,\cD'_h]_{\mathrm{gr}}
 +[\cD'',\eta^{\star_h}]_{\mathrm{gr}}
 \bigr).
\]
Its trace-free part is precisely $\cS_h(\eta)$.  The derivative of the
conjugating identification $\mathcal I_s$ does not add a term: at the
base point it would be a commutator with
$\mu_0(\cD'',h)$, which vanishes.  Adding the independent metric
variation proves \eqref{eq:full-linearization}.
\end{proof}

The invariant expression for the source operator can be unpacked by type.  This local formula is useful for checking signs and for the model calculations later in the paper.

For readers who prefer the usual Higgs notation, the invariant formula can be expanded by type.  With $\eta=(a,\varphi)$, the $(1,1)$ fixed-metric variation of the Chern curvature is
\[
  \partial_{E,h}a-\ddbar_E a^{\dagger_h},
\]
while the Higgs commutator contributes
\[
  [\varphi,\theta^{\dagger_h}]
  +[\theta,\varphi^{\dagger_h}].
\]
Consequently, with the curvature convention fixed in \cref{sec:relative-data},
\begin{align}\label{eq:source-local}
  \cS_h(a,\varphi)
  =\sqrt{-1}\Lambda\Bigl(&
  \partial_{E,h}a-\ddbar_E a^{\dagger_h}
  +[\varphi,\theta^{\dagger_h}]
  +[\theta,\varphi^{\dagger_h}]
  \Bigr)_0.
\end{align}
Formula \eqref{eq:source-operator} remains the coordinate-free definition.  Writing the type formula explicitly is useful because it exposes the minus sign in the Chern variation and prevents a common but incorrect replacement of the companion variation by the naive adjoint of $a+\varphi$.

\subsection{Gauge directions and the polystable kernel}

A fundamental consistency check is the behavior under an actual complex gauge path.  Let $g_t$ be a smooth path in the complex gauge group with $g_0=\Id$, and transform both the Higgs operator and the metric by the gauge action:
\[
  \cD''_t=g_t\cdot\cD'',
  \qquad
  h_t=g_t\cdot h.
\]
Write
\[
  \eta_\xi=\left.\frac d{dt}\right|_0\cD''_t,
  \qquad
  \dot h_0(v,w)=h(s_\xi v,w),
\]
so that $s_\xi$ is the Hermitian logarithmic metric variation determined by the chosen action convention.

\begin{lemma}[Linearized gauge covariance]\label{lem:source-gauge}
At a zero of the trace-free moment map,
\begin{equation}\label{eq:source-gauge}
  \cS_h(\eta_\xi)+L_h(s_\xi)_0=0.
\end{equation}
Thus the data variation tangent to a complex gauge path is exactly cancelled, in the moment-map equation, by the induced Hermitian metric variation after taking the trace-free part.
\end{lemma}

\begin{proof}
Gauge covariance gives
\[
  \mu_0(g_t\cdot\cD'',g_t\cdot h)
  =g_t\,\mu_0(\cD'',h)\,g_t^{-1}.
\]
Differentiating the right side gives
$[\dot g_0,\mu_0(\cD'',h)]$, which is zero because the base point is a
zero of the trace-free moment map.  On the left, the operator variation
is $\eta_\xi$.  The induced metric path has logarithmic derivative
$s_\xi$; its scalar part is invisible after trace-free projection, so
the metric contribution is $L_h(s_\xi)_0$.  Applying
\cref{prop:full-linearization} therefore gives
\[
 0=\cS_h(\eta_\xi)+L_h(s_\xi)_0,
\]
as claimed.  No separate undefined ``unitary remainder'' is needed:
the combined gauge path already contains the correct variation of both
the operator and the metric.
\end{proof}

This identity is a consistency check for the source operator.  The later stack-level gauge statement, \cref{prop:gauge-compatibility}, is formulated directly for derivatives of actual plots and therefore does not require every closed infinitesimal deformation to be integrable.

The same calculation also shows exactly what fails at a polystable point: the normalized Jacobi operator acquires the Hermitian part of the reductive centralizer as kernel.  This observation motivates the slice construction of \cref{sec:poly}.

Let
\begin{equation}\label{eq:polystable-decomp-fixed}
  (E,\cD'')\cong
  \bigoplus_{j=1}^m(E_j,\cD''_j)\otimes W_j
\end{equation}
be a polystable decomposition with stable pairwise non-isomorphic factors.  A compatible harmonic metric is
\[
  h=\bigoplus_j h_j\otimes q_j,
\]
with $q_j$ a Hermitian metric on $W_j$.

\begin{proposition}[Kernel on a polystable object]\label{prop:poly-kernel}
For \eqref{eq:polystable-decomp-fixed},
\begin{equation}\label{eq:poly-kernel}
  \ker L_h\cap\Herm_h(\End E)
  \cong
  \bigoplus_{j=1}^m\Herm_{q_j}(\End W_j).
\end{equation}
The kernel is canonically the Hermitian part of the Higgs endomorphism algebra.
\end{proposition}

\begin{proof}
By \cref{cor:energy-identity}, the kernel is the space of Hermitian
Higgs endomorphisms.  The decomposition \eqref{eq:polystable-decomp-fixed}
gives
\[
 \End_{\mathrm{Higgs}}(E)
 \cong
 \bigoplus_{i,j}
 \Hom_{\mathrm{Higgs}}(E_i,E_j)
 \otimes\Hom(W_i,W_j).
\]
All stable factors have the same slope.  A nonzero Higgs morphism
between two stable bundles of that slope is an isomorphism: stability
applied to its kernel and saturated image rules out any intermediate
rank.  Because the $E_j$ are pairwise non-isomorphic, the off-diagonal
Hom spaces therefore vanish, while simplicity gives
$\End_{\mathrm{Higgs}}(E_j)=\C\Id_{E_j}$.  It follows that
\[
  \End_{\mathrm{Higgs}}(E)
  \cong\bigoplus_j\End(W_j).
\]
Under this identification the adjoint induced by
$h=\bigoplus_jh_j\otimes q_j$ is the ordinary $q_j$-adjoint on each
factor $\End(W_j)$.  Taking the fixed points of that adjoint yields
\eqref{eq:poly-kernel}.
\end{proof}

The dimension of \eqref{eq:poly-kernel} can jump when stable factors collide or multiplicities change.  This is the analytic obstruction to a uniform inverse for $L_h$ across arbitrary polystable families.

\section{Smooth dependence of harmonic metrics on the stable locus}\label{sec:smooth-metrics}

\subsection{Determinant normalization and background metrics}

We first solve the determinant equation with parameter dependence included.

Let $(L,\ddbar_L)\to U\times X$ be a smooth family of holomorphic line bundles of degree zero.  Fix a smooth Hermitian metric $q^0$ on $L$.  Any other metric has the form
\[
  q=q^0e^f,
  \qquad f\in C^\infty(U\times X,\R).
\]
The Chern curvature changes by
\[
  F_q=F_{q^0}+\ddbar_{X/U}\partial_{X/U}f
\]
with our sign convention.  Hence the Hermitian--Einstein equation in degree zero is a scalar Poisson equation
\begin{equation}\label{eq:det-poisson}
  \Delta_\omega f_u
  =-2\sqrt{-1}\Lambda F_{q^0_u},
\end{equation}
up to the fixed convention factor in the definition of the positive Laplacian.  The right side has integral zero because $\deg L_u=0$.

\begin{proposition}[Smooth Hermitian--Einstein determinant metrics]\label{prop:smooth-det}
Let $(L,\ddbar_L)$ be a smooth family of degree-zero holomorphic line bundles over $U\times X$, and fix a smooth background metric $q^0$ on $L$.  There is a smooth Hermitian metric $q$ on all of $L$ satisfying
\[
  \sqrt{-1}\Lambda F_{q_u}=0
\]
for every $u\in U$.  Relative to $q^0$, it is unique after imposing
\[
  \int_X\log(q_u/q^0_u)\,\frac{\omega^n}{n!}=0.
\]
If the coefficients of the family and the chosen background metric are real analytic in $u$ on a parameter chart, then the restriction of $q$ to that chart depends real analytically on $u$ in Sobolev topology.
\end{proposition}

\begin{proof}
Let $\Delta$ be the scalar Laplacian on $X$.  Because $X$ is connected,
its kernel consists of the constant functions.  Self-adjointness and the
Fredholm alternative therefore give an isomorphism
\[
  \Delta:W^{k+2,2}_0(X)\longrightarrow W^{k,2}_0(X)
\]
on the closed subspaces of mean-zero functions.  The curvature
$F_{q^0}$ is a globally defined relative two-form, and
\[
 \int_X\sqrt{-1}\Lambda F_{q^0_u}\,\frac{\omega^n}{n!}
 =\int_X\sqrt{-1}F_{q^0_u}\wedge
                  \frac{\omega^{n-1}}{(n-1)!}=0,
\]
up to the fixed convention factor, because $L_u$ has degree zero.
Thus the source in \eqref{eq:det-poisson} is a smooth map from $U$ into
$W^{k,2}_0(X)$.  Define
\[
  f_u=\Delta^{-1}\bigl(-2\sqrt{-1}\Lambda F_{q^0_u}\bigr).
\]
The inverse is a fixed bounded linear map, so $u\mapsto f_u$ is smooth
at the chosen Sobolev level.  Since the source is smooth in $x$ to every
order, elliptic regularity and Sobolev embedding make $f$ jointly smooth
on $U\times X$.  The metric $q=q^0e^f$ then satisfies the desired
Hermitian--Einstein equation and the imposed integral normalization.

If $q^0e^{f_1}$ and $q^0e^{f_2}$ are two normalized solutions, then
$\Delta(f_1-f_2)=0$.  Their difference is constant on $X$, and its
mean is zero, so it vanishes.  This proves uniqueness.  Finally, on a
real-analytic parameter chart the curvature source is analytic as a
Sobolev-valued map.  Composition with the same fixed bounded inverse
$\Delta^{-1}$ preserves analyticity, proving the last assertion.
\end{proof}

Apply this to $L=\det E$.  We obtain one smooth Hermitian--Einstein metric $q$ on the determinant line over all of $U\times X$.

Solving the scalar determinant equation is only the first step.  We next choose a smooth background metric on the full bundle with exactly that determinant, so that the remaining nonlinear variable is trace free.

Choose any smooth Hermitian metric $\kappa^0$ on $E|_{V\times X}$.  The quotient
\[
  \rho=\frac{q}{\det \kappa^0}
\]
is a positive smooth function.  Define
\begin{equation}\label{eq:background-prescribed-det}
  \kappa=\rho^{1/r}\kappa^0.
\end{equation}
Then $\det \kappa=q$.  At the distinguished point $u_0$, we may
arrange that $\kappa_{u_0}$ equals a chosen normalized harmonic metric
$h_0$ by first modifying $\kappa^0$ in the parameter neighbourhood.
Concretely, extend $h_0$ to a smooth metric $\widehat h$ in the chosen
local bundle trivialization, choose a parameter bump function $\chi$
with $\chi(u_0)=1$, and replace $\kappa^0$ by
\[
 (1-\chi)\kappa^0+\chi\widehat h.
\]
The cone of Hermitian metrics is convex, so this remains positive and
equals $h_0$ on the central slice.  Since $\det h_0=q_{u_0}$, the scalar
correction in \eqref{eq:background-prescribed-det} is one there.

We henceforth assume
\begin{equation}\label{eq:kappa-central-harmonic}
  \kappa_{u_0}=h_0,
  \qquad
  \det \kappa_u=q_u.
\end{equation}
For the real-analytic statement we use a different local choice, because the preceding bump-function adjustment is only smooth.  In a real-analytic local trivialization choose the determinant background so that $q_{u_0}=\det h_0$.  After the analytic Poisson construction of \cref{prop:smooth-det}, set
\begin{equation}\label{eq:analytic-background-prescribed-det}
  \kappa_u=\left(\frac{q_u}{\det h_0}\right)^{1/r}h_0.
\end{equation}
The positive ratio and its $r$th root depend real analytically on the parameter in each Sobolev algebra under consideration.  Thus \eqref{eq:analytic-background-prescribed-det} has determinant $q_u$, equals $h_0$ at $u_0$, and supplies the analytic background needed below without using analytic bump functions.

Every metric with determinant $q_u$ is uniquely of the form
\[
  h_{u,\widetilde s}=\kappa_ue^{\widetilde s},
  \qquad \widetilde s\in\Herm^0_{\kappa_u}(\End E_0).
\]
Let $b_u:(E_0,\kappa_u)\to(E_0,h_0)$ be the positive isometry obtained
from the square root of the relative metric endomorphism, chosen
smoothly with $b_{u_0}=\Id$.  Explicitly, if $r_u$ is defined by
$\kappa_u(v,w)=h_0(r_uv,w)$, then $b_u=r_u^{1/2}$; smooth functional
calculus on the open cone of positive endomorphisms gives smooth
dependence on $u$.  Conjugation by $b_u$ identifies the moving spaces
$\Herm^0_{\kappa_u}(\End E_0)$ with the fixed space
$\Herm^0_{h_0}(\End E_0)$.  Thus, for a variable
\[
  s\in\Herm^0_{h_0}(\End E_0),
\]
we set
\begin{equation}\label{eq:transported-s}
  \widetilde s_u=b_u^{-1}s\,b_u
  \in\Herm^0_{\kappa_u}(\End E_0),
  \qquad
  h_{u,s}:=\kappa_u e^{\widetilde s_u}.
\end{equation}
This makes the domain of the nonlinear problem a single fixed Sobolev space.  Notice that $\widetilde s_{u_0}=s$.

\subsection{The nonlinear moment-map equation and harmonicity}

The target Hermitian subspace also moves with the metric $h_{u,s}$.  Define the isometry
\[
  c_{u,s}=b_u e^{\widetilde s_u/2}
  =e^{s/2}b_u:
  (E_0,h_{u,s})\longrightarrow(E_0,h_0)
\]
and the corresponding identification
\[
  \mathcal J_{u,s}(A)=c_{u,s}Ac_{u,s}^{-1}.
\]
We then define
\begin{equation}\label{eq:stable-F}
  \cF(u,s)
  =\mathcal J_{u,s}\!\left(
  \mu_0(\cD''_u,h_{u,s})
  \right).
\end{equation}
This has the same zero set as the trace-free moment-map equation and takes values in a fixed Banach space of $h_0$-Hermitian trace-free endomorphisms.
At $(u_0,0)$,
\[
  \cF(u_0,0)=0.
\]
By \cref{lem:F-smooth}, $\cF$ is smooth between the Sobolev Banach spaces
\[
  V\times W^{k+2,2}(\Herm^0\End E_0)
  \longrightarrow
  W^{k,2}(\Herm^0\End E_0).
\]
Because the central moment map vanishes, differentiating the conjugating identification $\mathcal J_{u,s}$ contributes no extra term at $(u_0,0)$.  Hence the derivative in the $s$-direction is the Jacobi operator
\[
  D_s\cF(u_0,0)=L_{h_0}.
\]
By \cref{prop:stable-kernel}, this is an isomorphism.

\begin{theorem}[Local Sobolev family of normalized solutions]\label{thm:sobolev-family}
After shrinking $V$ about $u_0$, there is a unique smooth map
\[
  s:V\to W^{k+2,2}(\Herm^0\End E_0)
\]
with $s(u_0)=0$ and
\[
  \cF(u,s(u))=0.
\]
If the locally trivialized coefficients of the family and the determinant background are real analytic in $u$, then $s$ is real analytic.
\end{theorem}

\begin{proof}
The fixed-domain and fixed-target identifications preceding \eqref{eq:stable-F} put the nonlinear equation on fixed Banach spaces:
\[
  \cF:V\times W^{k+2,2}(\Herm^0\End E_0)
  \longrightarrow W^{k,2}(\Herm^0\End E_0),
\]
with $\cF(u_0,0)=0$.  The choice of $k$ guarantees that the Sobolev spaces used here are closed under the nonlinear products occurring in the Chern connection, the adjoint Higgs field, and the curvature expression.  The exponential parametrization of positive metrics is smooth on a neighbourhood of the origin, and the smoothly varying coefficients of the family therefore make $\cF$ a smooth Banach map.

The linearization in the metric variable has already been computed:
\[
  D_s\cF(u_0,0)=L_{h_0}.
\]
By \cref{prop:stable-kernel}, this operator is an isomorphism on the trace-free normalized slice.  The Banach implicit-function theorem therefore gives a smaller neighbourhood $V'\subset V$ of $u_0$, a neighbourhood $B$ of $0$ in $W^{k+2,2}$, and a unique smooth map
\[
  s:V'\to B
\]
such that $s(u_0)=0$ and $\cF(u,s(u))=0$.  This is the desired local Sobolev family.

If the locally trivialized coefficients and determinant background depend real analytically on $u$, then the same algebraic operations define a real-analytic Banach map $\cF$.  Since the vertical derivative is unchanged and remains invertible, the analytic Banach implicit-function theorem yields real-analytic dependence of $s$ in the stated Sobolev topology.
\end{proof}

The implicit-function theorem solves the trace-free moment-map equation.  We now explain why, under the NAH numerical conditions already imposed, this solution satisfies the full harmonic flatness equation.

The solution of \cref{thm:sobolev-family} satisfies
\[
  \mu_0(\cD''_u,h_u)=0,
  \qquad
  h_u=h_{u,s(u)}=\kappa_u e^{\widetilde s_u},
  \qquad
  \widetilde s_u=b_u^{-1}s(u)b_u.
\]
Since $\det h_u=q_u$ and $q_u$ solves the determinant Hermitian--Einstein equation, the trace of the moment map also vanishes.  Therefore
\begin{equation}\label{eq:full-mu-zero}
  \mu(\cD''_u,h_u)=0.
\end{equation}
Under the NAH numerical conditions \eqref{eq:NAH-condition}, the Hitchin--Simpson Chern--Weil identity upgrades \eqref{eq:full-mu-zero} to flatness of
\[
  \nabla_{h_u}=D_{h_u}+\theta_u+\theta_u^{\dagger_{h_u}}.
\]
Thus $h_u$ is harmonic in the sense of \cref{def:harmonic-metric}.

\begin{remark}[Where the numerical conditions enter]
The implicit-function argument itself uses only the normalized moment-map equation and invertibility of the trace-free Jacobi operator.  The condition $\nu_1(E)=0$ is used to solve the determinant equation with zero Einstein constant.  The condition $\nu_2(E)=0$ enters in the Chern--Weil identity that forces the $L^2$ norm of the remaining curvature terms to vanish.  Stronger hypotheses such as vanishing rational Chern classes are of course sufficient, but they are not required for the local analytic step.  This is the standard compact K\"ahler mechanism; see \cite{Simpson1988,Simpson1992}.
\end{remark}

\subsection{Regularity, uniqueness, and proof of the stable theorem}

The implicit-function theorem gives $s(u)$ in $W^{k+2,2}$.  We now show that the family is jointly smooth.

\begin{proposition}[Joint regularity]\label{prop:joint-regularity}
The metric
\[
  h(u,x)=\kappa(u,x)e^{\widetilde s_u(x)},
  \qquad
  \widetilde s_u=b_u^{-1}s(u)b_u,
\]
constructed above is $C^\infty$ on $V\times X$.
\end{proposition}

\begin{proof}
The equation $\cF(u,s(u))=0$ is a quasilinear elliptic equation in the $X$-variables.  Its linearization along the solution family is a smooth family of strongly elliptic operators
\[
  L_{u,s(u)}:W^{j+2,2}\to W^{j,2}.
\]
Differentiating in a parameter coordinate $u_i$ gives
\begin{equation}\label{eq:differentiate-family-equation}
  L_{u,s(u)}\,\partial_{u_i}s
  =-\partial_{u_i}^{\mathrm{exp}}\cF(u,s(u)),
\end{equation}
where the right side denotes the explicit parameter derivative with $s$ fixed.  The coefficients are smooth in $u$ and $x$.  Uniform elliptic estimates on a relatively compact sub-neighbourhood show that $\partial_{u_i}s$ gains two $X$-derivatives.  Repeating for arbitrary mixed parameter derivatives and increasing the Sobolev order proves smoothness into every $W^{j,2}$.  Sobolev embedding gives joint smoothness.
More explicitly, after applying a parameter multi-index $\alpha$, the
highest derivative appears only through
$L_{u,s(u)}\partial_u^\alpha s$; all remaining terms are finite sums of
products of coefficient derivatives and lower parameter derivatives of
$s$.  An induction on $|\alpha|$, followed by an induction on the
$X$-Sobolev order, therefore places every $\partial_u^\alpha s$ in
every $W^{j,2}$ locally uniformly in $u$.  Stability ensures that the
linearized operator remains invertible after shrinking, although the
regularity estimate itself only requires strong ellipticity and the
already controlled $L^2$ norm.  Taking $j$ successively larger than
$\dim_\R X/2+|\gamma|$ makes
$\partial_u^\alpha\partial_x^\gamma s$ continuous by Sobolev embedding.
The background factors $\kappa_u$ and $b_u$ are already jointly smooth,
and matrix exponentiation is smooth in each Sobolev algebra.  Hence the
displayed formula for $h$ is jointly $C^\infty$.
\end{proof}

Existence has been proved after choosing a local trivialization and a
background metric.  To make the result intrinsic and force local
solutions to agree on overlaps, we combine stability with determinant
normalization to prove uniqueness.

\begin{proposition}[Uniqueness]\label{prop:normalized-uniqueness}
Let $(E,\cD'')$ be stable.  If $h_1$ and $h_2$ are harmonic metrics with
\[
  \det h_1=\det h_2,
\]
then $h_1=h_2$.
\end{proposition}

\begin{proof}
Write
\[
 h_2(v,w)=h_1(e^s v,w)
\]
for the unique $h_1$-self-adjoint endomorphism $s$.  Join the two
metrics by the geodesic $h_t=h_1e^{ts}$.  The convexity formula for the
Higgs Donaldson functional gives, with our moment-map normalization,
\[
 \frac{d^2}{dt^2}\mathcal M(h_t)
 =c_\omega\,\|\cD''s\|^2_{L^2(h_t)}\ge0,
\]
where $c_\omega>0$ is the harmless convention constant.  Its first
variation is the $L^2$ pairing of $s$ with the moment map.  Both endpoint
metrics solve the moment-map equation, so the derivative of
$\mathcal M(h_t)$ vanishes at $t=0$ and $t=1$.  A convex function whose
derivative has the same value at both endpoints is affine; hence the
displayed nonnegative second derivative vanishes.  It follows that
$\cD''s=0$.

Thus $s$ is a Hermitian Higgs endomorphism.  Stability implies
simplicity, as in \cref{prop:stable-kernel}, so $s=c\Id_E$ for a real
constant $c$.  Taking determinants in $h_2=h_1e^s$ gives
\[
  \det h_2=e^{rc}\det h_1.
\]
The prescribed determinants are equal, so $e^{rc}=1$.  Since $c$ is
real, $c=0$, and therefore $h_2=h_1$.
\end{proof}

\begin{corollary}[Global gluing]\label{cor:metrics-glue}
Let $q$ be a smooth Hermitian--Einstein metric on $\det E$ over $U\times X$, and let $\{V_\alpha\}$ be a parameter cover on which the construction above produces $q$-normalized harmonic metrics.  Then the local metrics agree on overlaps and glue to a unique $q$-normalized harmonic metric on all of $E$.
\end{corollary}

\begin{proof}
Fix an overlap $V_\alpha\cap V_\beta$.  For each parameter value $u$ in the overlap, the two local constructions give harmonic metrics on the same stable Higgs fibre and both have determinant $q_u$.  By \cref{prop:normalized-uniqueness}, the two metrics coincide on $\{u\}\times X$.  Since this holds for every $u$, the two smooth sections of the bundle of positive Hermitian forms agree pointwise on the whole overlap.

Thus the local metrics satisfy the ordinary cocycle condition by equality, not merely up to gauge.  The sheaf property for smooth sections therefore produces a unique smooth metric on $E$ over $U\times X$.  Harmonicity is checked slicewise and is preserved under restriction, so the glued metric is harmonic on every fibre.
\end{proof}

By \cref{prop:smooth-det}, the required global determinant normalization is always available for the stable families under consideration.  Thus global gluing is unconditional here; it is not an additional hypothesis on the family.

An arbitrarily chosen collection of local harmonic metrics need not already have equal determinants on overlaps.  This is only a mismatch of representatives of the scalar torsor, not an obstruction attached to the family.  Given the global $q$, each local metric has a unique smooth positive scalar rescaling whose determinant is $q$; the rescaled metrics then agree by \cref{prop:normalized-uniqueness}.  Consequently there is no stable family which is locally harmonic but fails to lie in the prestack-level essential image merely because its first chosen local normalizations do not glue.

All ingredients are now in place, and the main stable theorem is obtained by assembling the determinant equation, the fixed-space implicit-function argument, elliptic regularity, and uniqueness.

\begin{proof}[Proof of \cref{thm:intro-smooth-metrics}]
Choose a smooth background metric on $\det E$.  By \cref{prop:smooth-det}, the determinant family admits a Hermitian--Einstein metric $q$ on all of $U\times X$.  For each $u_0\in U$, the ordinary Hitchin--Simpson correspondence gives a harmonic metric $h_0$ on the stable fibre $(E_{u_0},\cD''_{u_0})$, unique up to positive scalar; rescale it so that $\det h_0=q_{u_0}$.

To justify the last normalization, both $\det h_0$ and $q_{u_0}$
solve the degree-zero Hermitian--Einstein equation on the same
holomorphic determinant line.  The logarithm of their ratio has zero
Laplacian and is therefore constant on connected $X$.  A unique
positive scalar rescaling of $h_0$ makes the two determinant metrics
equal.

Choose a local smooth trivialization of the underlying family near
$u_0$ and a background metric $\kappa$ satisfying
\eqref{eq:kappa-central-harmonic}.  The trace-free moment-map equation
is \eqref{eq:stable-F}.  Its metric derivative at $(u_0,0)$ is
$L_{h_0}$, an isomorphism by stability and
\cref{prop:stable-kernel}.  The Banach implicit-function theorem gives
a unique smooth Sobolev solution $s(u)$ after the parameter
neighbourhood is shrunk.

The solution has determinant $q_u$ by construction.  The trace of the
Hitchin--Simpson moment map is the curvature contraction of the
determinant Chern connection, because traces of commutators vanish.
The Hermitian--Einstein equation for $q_u$ therefore restores the
missing central equation and gives $\mu(\cD''_u,h_u)=0$.  The standard
Hitchin--Simpson Chern--Weil identity expresses the remaining curvature
energy of $\nabla_{h_u}$ as the characteristic-number combination in
\eqref{eq:NAH-condition}, once the contracted moment map is zero.  That
combination vanishes by hypothesis, so every nonnegative curvature
term vanishes and $\nabla_{h_u}$ is flat.  Thus the Sobolev solution is
harmonic.  \Cref{prop:joint-regularity} gives joint smoothness.
Applying this construction near every point of $U$ produces a parameter
cover of local $q$-normalized metrics, and
\cref{cor:metrics-glue} glues them to one smooth harmonic metric on all
of $E$.

For the equivalent integral normalization, fix a global smooth background metric $\kappa$ on $E$ and let $h^q$ denote the global $q$-normalized solution just constructed.  The function $u\mapsto N_\kappa(h^q)_u$ is smooth, so
\[
  h^\kappa_u=\exp\!\bigl(-N_\kappa(h^q)_u\bigr)h^q_u
\]
is a global smooth harmonic family satisfying $N_\kappa(h^\kappa)_u=0$.  Any other integrally normalized harmonic family differs slicewise from $h^\kappa$ by a positive scalar, and the transformation law $N_\kappa(ch)=N_\kappa(h)+\log c$ forces that scalar to be one.

For real-analytic parameter dependence under the local hypotheses stated in \cref{thm:intro-smooth-metrics}, \cref{prop:smooth-det} and the analytic Banach implicit-function theorem give real-analytic dependence in a fixed Sobolev topology.  Elliptic bootstrapping promotes this to all Sobolev orders.  Uniqueness after normalization is \cref{prop:normalized-uniqueness}.
\end{proof}

Nothing in the proof assumes that $U$ is one-dimensional, contractible,
or a global trivializing chart.  The only local reduction is the smooth
identification of the underlying bundles over a parameter
neighbourhood, and mixed parameter derivatives are controlled by
repeated differentiation of \eqref{eq:differentiate-family-equation}.
Nontrivial monodromy may obstruct a global bundle trivialization, but
the global determinant construction and normalized overlap uniqueness
produce intrinsic metric sections on $U\times X$.

\subsection{Quantitative control and spectral degeneration}

The implicit-function proof also gives a local estimate, but the varying determinant background must be included among the coefficients unless the determinant is fixed.  Let $u_0=0$ in a coordinate ball and let $p_u$ denote the full coefficient datum entering the fixed-Banach-space equation.  After shrinking the ball,
\begin{equation}\label{eq:quantitative-s-estimate}
  \norm{s(u)}_{W^{k+2,2}}
  \le C\norm{p_u-p_0}_{\mathcal P}
\end{equation}
for a local Banach norm $\mathcal P$ strong enough to control the coefficients occurring in \eqref{eq:stable-F}.  In a fixed-determinant family one may take $p_u$ to be the Higgs operator alone.  The constant $C$ depends on the inverse spectral gap of $L_{h_0}$ and on a bounded neighbourhood of the coefficient data.

\begin{proposition}[Local Lipschitz estimate]\label{prop:local-Lipschitz}
After shrinking the parameter neighbourhood, there is $C>0$ such that
\[
  \norm{s(u)-s(v)}_{W^{k+2,2}}
  \le C\norm{p_u-p_v}_{\mathcal P^{k+1}}
\]
for all $u,v$ in the neighbourhood.  Here $p_u$ denotes the full local coefficient datum entering the fixed-Banach-space equation---the Higgs operator together with the determinant background produced by the Poisson construction---and $\mathcal P^{k+1}$ is a fixed local Banach norm strong enough to control the coefficient maps appearing in \eqref{eq:stable-F}.  In a fixed-determinant family one may take $p_u=\cD''_u$.
\end{proposition}

\begin{proof}
Subtract the equations
\[
  \cF(p_u,s(u))=0,
  \qquad
  \cF(p_v,s(v))=0.
\]
Split their difference as
\begin{align*}
0={}&\cF(p_u,s(u))-\cF(p_u,s(v))\\
   &+\cF(p_u,s(v))-\cF(p_v,s(v)).
\end{align*}
The Banach-space mean-value formula writes the first line as
\[
 A_{u,v}(s(u)-s(v)),
 \qquad
 A_{u,v}=\int_0^1D_s\cF
   \bigl(p_u,s(v)+t(s(u)-s(v))\bigr)\,dt.
\]
After shrinking the neighbourhood, every integrand is a small
operator-norm perturbation of $L_{h_0}$.  The Neumann-series argument
therefore makes $A_{u,v}$ invertible with
$\|A_{u,v}^{-1}\|\le C_0$ uniformly in $u,v$.  Applying the mean-value
formula to the coefficient term gives
\[
 \|\cF(p_u,s(v))-\cF(p_v,s(v))\|_{W^{k,2}}
 \le C_1\|p_u-p_v\|_{\mathcal P^{k+1}},
\]
because $D_p\cF$ is bounded on the same closed local neighbourhood.
Solving the preceding identity with $A_{u,v}^{-1}$ yields the claimed
estimate with $C=C_0C_1$.  The determinant background depends smoothly
on the determinant coefficients by \cref{prop:smooth-det}; including it
in $p$ is what makes the coefficient estimate complete.
\end{proof}

The estimate also indicates how the stable theory can degenerate near
the polystable boundary: as a stabilizer develops, the smallest positive
eigenvalue of the Jacobi operator may collapse, and the constants in the
implicit-function argument deteriorate with that loss of spectral gap.

Let
\[
  \lambda_1(u)>0
\]
be the smallest positive eigenvalue of $L_{h_u}$ on trace-free Hermitian endomorphisms.  On a compact subset of the stable parameter locus, $\lambda_1$ is bounded below.  If a family approaches a polystable object with non-scalar automorphisms, one expects $\lambda_1(u)$ to tend to zero along suitable degenerations.  The norm of the Green operator
\[
  G_{h_u}=L_{h_u}^{-1}
\]
then grows like $\lambda_1(u)^{-1}$.  The first-variation formula developed later makes this mechanism explicit.

\begin{remark}
This spectral-gap viewpoint is one reason that a global smoothness theorem on the entire polystable locus should not be expected without stratification or a stack of choices.  Even if harmonic metrics exist fibrewise, the choice of inverse to the Jacobi operator is unstable when the endomorphism algebra jumps.
\end{remark}

\section{Smoothness of the non-Abelian Hodge transform along plots}\label{sec:stack-smoothness}

\subsection{Diffeological plots and the harmonic metric stack}

We recall only the portion of the diffeological formalism needed here.  A diffeological space is a set equipped with a specified class of maps from open subsets of Euclidean spaces, called plots, satisfying the usual covering, locality, and smooth reparametrization axioms.  A diffeological groupoid presents a stack over the site of smooth manifolds.  In \cite{AzamRayan2026}, the groupoids of smooth families of Higgs bundles and flat bundles were shown to define diffeological stacks.

For the present paper, the operative principle is concrete: a map from a smooth manifold $U$ into $\MdiffDol(X)$ is represented locally by a smooth family of Higgs bundles over $U\times X$, and a smooth map $f:V\to U$ acts by pullback.  Therefore a transformation is smooth if it takes every such family to a smooth family, compatibly with reparametrization and descent.

Let
\[
  \MdiffDol^{\st,0}(X)\subset\MdiffDol(X)
\]
denote the full substack of stable families satisfying the vanishing conditions used above.  We similarly write
\[
  \MdiffdR^{\irr}(X)
\]
for irreducible flat families.

The plotwise theorem becomes stack-theoretic only after the normalization and its descent behavior are incorporated.  We therefore introduce the corresponding harmonic metric groupoid before defining the transform itself.

A normalization must be included if one wants a genuine section rather than a torsor under positive scalar functions.

\begin{definition}[Normalized harmonic metric stack]\label{def:normalized-metric-stack}
For a smooth manifold $U$, an object of
\[
  \MetHar^{\det}(X)(U)
\]
is a tuple
\[
  (E,\cD'',q,h)
\]
where $(E,\cD'')$ is a smooth Higgs family, $q$ is a smooth Hermitian--Einstein metric of degree zero on $\det E$, and $h$ is a smooth harmonic metric on $E$ with $\det h=q$.  Morphisms are isomorphisms of Higgs families preserving the metrics and determinant data.
\end{definition}

There are forgetful maps
\[
  \MetHar^{\det}(X)
  \longrightarrow
  \MdiffDol(X)
\]
and
\[
  \MetHar^{\det}(X)
  \longrightarrow
  \MdiffdR(X),
\]
the latter sending $(E,\cD'',q,h)$ to $(E,\nabla_h)$.

\begin{proposition}[Global lifts over the stable locus]\label{prop:metric-stack-section}
The forgetful map
\[
  \MetHar^{\det}(X)\to\MdiffDol^{\st,0}(X)
\]
is essentially surjective on objects over every smooth parameter manifold: a stable family over $U\times X$ admits a lift over the same $U$, without refining the parameter cover.  Once a Hermitian--Einstein determinant metric is fixed, the harmonic metric on a fixed underlying stable Higgs family is unique.  The choice of determinant metric is not canonical, so this objectwise lifting statement does not assert an unnormalized functorial section of prestacks.

No assertion of trivial isotropy is intended: a stable Higgs object can still have scalar unitary automorphisms preserving its harmonic metric.  Thus the correct statement is uniqueness of the normalized metric on the fixed Higgs object, not uniqueness of arrows in the ambient groupoid.
\end{proposition}

\begin{proof}
Let a plot of $\MdiffDol^{\st,0}(X)$ be represented by a stable Higgs family $(E,\cD'')$ over $U\times X$.  By \cref{prop:smooth-det}, $\det E$ admits a Hermitian--Einstein metric $q$ over all of $U\times X$, and \cref{thm:intro-smooth-metrics} supplies a global smooth harmonic metric $h$ with $\det h=q$.  Hence the original plot, without passing to a cover, lifts to the object $(E,\cD'',q,h)$ of $\MetHar^{\det}(X)$.

Suppose two lifts lie over the same Higgs family and determinant datum.  Fibrewise, \cref{prop:normalized-uniqueness} identifies their metrics.  Since both metrics vary smoothly, the fibrewise equality is equality of the smooth family data.  If the two underlying Higgs objects are merely isomorphic, pull one metric back along the isomorphism and apply the same argument.  Compatibility with pullback in the parameter follows from \cref{prop:pullback-solutions}.

This proves uniqueness of the normalized metric on a fixed Higgs object.  It does not eliminate automorphisms of that object: scalar unitary Higgs automorphisms still preserve the metric.  Thus the conclusion concerns global objectwise lifting, not triviality of isotropy or a canonical choice of normalization in the groupoid.
\end{proof}

\subsection{Construction and comparison of the stable transform}

Let $(E,\cD'')$ be a stable family over $U$.  Choose a global normalized smooth harmonic metric $h$ and define
\begin{equation}\label{eq:NAH-family-transform}
  \operatorname{NAH}_U(E,\cD'')=(E,\nabla_h),
  \qquad
  \nabla_h=\cD''+\cD'_h.
\end{equation}
There is an elementary point that removes a possible ambiguity from this definition.  If $h'=e^{c(u)}h$ where $c(u)$ is constant along the $X$-slice, then the relative Chern connection in the $X$-directions is unchanged and so is $\theta^{\dagger_h}$.  Hence
\begin{equation}\label{eq:scalar-independence-flat}
  \nabla_{h'}=\nabla_h.
\end{equation}
Thus determinant normalization is indispensable for producing a unique smooth metric section and for differentiating that section, but it is not needed to make the associated flat family canonical.  Different global normalizations therefore yield exactly the same flat connection, not merely centrally gauge-equivalent ones.

\begin{proposition}[Scalar independence]\label{prop:scalar-independence}
Let $h$ be harmonic on a stable Higgs family and let $c:U\to\R$ be smooth.  Then $e^{c\circ\operatorname{pr}_U}h$ is harmonic and determines the same relative flat connection on $U\times X$.
\end{proposition}

\begin{proof}
Set $h'=e^{c\circ\operatorname{pr}_U}h$.  The relative differential acts only in the $X$-directions, so
\[
  \partial_{X/U}(c\circ\operatorname{pr}_U)=0,
  \qquad
  \ddbar_{X/U}(c\circ\operatorname{pr}_U)=0.
\]
In a local holomorphic frame, the $(1,0)$ part of the Chern connection is $h^{-1}\partial_{X/U}h$.  Hence
\[
  (h')^{-1}\partial_{X/U}h'
  =e^{-c}h^{-1}\partial_{X/U}(e^ch)
  =h^{-1}\partial_{X/U}h.
\]
Thus the relative Chern connection is unchanged.  The Higgs adjoint is also unchanged, because scalar rescaling multiplies both sides of its defining Hermitian identity by the same factor.  Therefore
\[
  \cD'_{h'}=\cD'_h,
  \qquad
  \nabla_{h'}=\cD''+\cD'_{h'}=\nabla_h.
\]
The latter connection is flat because $h$ is harmonic, so $h'$ is harmonic as well and determines exactly the same relative flat connection.
\end{proof}

\begin{theorem}[Smooth stable transform]\label{thm:smooth-stable-transform}
The construction \eqref{eq:NAH-family-transform} defines a smooth morphism of diffeological stacks
\begin{equation}\label{eq:smooth-stack-map}
  \operatorname{NAH}^{\st}:
  \MdiffDol^{\st,0}(X)
  \longrightarrow
  \MdiffdR^{\irr}(X).
\end{equation}
On point fibres, this is the classical non-Abelian Hodge correspondence from stable Higgs bundles to irreducible flat bundles.  On every plot it is represented by a jointly smooth family of flat connections.
\end{theorem}

\begin{proof}
Let $U\to\MdiffDol^{\st,0}(X)$ be a plot, represented by $(E,\cD'')$ over $U\times X$.  By \cref{thm:intro-smooth-metrics}, it admits a global smooth normalized harmonic metric.  Formula \eqref{eq:NAH-family-transform} then gives a smooth family of connection one-forms because the relative Chern operator and the Higgs adjoint depend smoothly on the metric and family data.  The resulting connection is flat slicewise by the NAH numerical conditions and the Chern--Weil identity discussed in \cref{sec:numerical-conditions}.

Two normalized choices agree when the determinant normalization agrees; without matching normalizations, they differ by an $X$-constant positive scalar and hence give the same connection by \cref{prop:scalar-independence}.  Therefore the flat family is canonical even though the auxiliary metric normalization is not.

If $f:V\to U$ is smooth, \cref{prop:pullback-solutions} implies that the normalized metric for the pulled-back family is $f_X^*h$ after pulling back the determinant data.  Hence
\[
  \operatorname{NAH}_V(f_X^*E,f_X^*\cD'')
  =f_X^*\operatorname{NAH}_U(E,\cD'').
\]
The construction respects isomorphisms and descent, and stable harmonic bundles correspond fibrewise to irreducible local systems.  It therefore defines the stated smooth morphism of stacks.
\end{proof}

The stable transform agrees with the extension-generated equivalence of
our earlier paper, while retaining additional analytic information in
the smoothly varying metric.  That equivalence was constructed by
passing from the harmonic dg-prestack to finite iterated extension
completions and then stackifying.  On the stable locus neither operation
is needed to obtain the object: by \cref{thm:intro-smooth-metrics}, a
stable family over $U$ is already the Dolbeault image of a harmonic
family over the same $U$.

\begin{corollary}[Agreement of the two constructions]\label{cor:agreement}
The restriction of the equivalence
\[
  \MHDol(X)\simeq\MHdR(X)
\]
from \cite{AzamRayan2026} to $\MdiffDol^{\st,0}(X)$ is naturally isomorphic to the smooth transform \eqref{eq:smooth-stack-map}.
\end{corollary}

\begin{proof}
The equivalence of \cite{AzamRayan2026} is obtained from the harmonic mediator by applying the Dolbeault and de Rham forgetful functors, then closing under finite extensions and descent.  On the stable locus considered here, \cref{thm:intro-smooth-metrics} shows that a family is already globally represented by a smooth harmonic object over its original parameter manifold; neither extension completion nor stackification is needed to produce that object.

Choose such a harmonic representative $(E,\cD'',h)$ over $U$.  The present transform sends it to $(E,\nabla_h)$.  The earlier construction applies the de Rham forgetful functor to the same harmonic object and therefore gives the same flat family before stackification.  Even if different scalar normalizations are used, \cref{prop:scalar-independence} shows that the induced flat connections agree.  These identifications commute with pullback and descent, so they assemble into the claimed natural isomorphism.
\end{proof}

\subsection{The inverse transform, analytic plots, and naturality}

The inverse direction is most naturally stated for irreducible flat families.  The distinction matters: a flat object can have only scalar endomorphisms without being semisimple, whereas the harmonic-metric correspondence is formulated for reductive representations.  Irreducibility gives reductivity and corresponds to stability on the Higgs side.

\begin{theorem}[Smooth inverse on the irreducible locus]\label{thm:smooth-inverse}
Let $(E,\nabla_u)$ be a smooth family of irreducible flat bundles over $U\times X$.  Then $E$ admits a global smooth family of harmonic reductions.  After fixing a global harmonic metric on the determinant flat line, the reduction is unique and determines a smooth family of stable Higgs bundles.  The resulting construction is inverse to \eqref{eq:smooth-stack-map}.
\end{theorem}

\begin{proof}
We spell out the analytic reduction because this inverse statement is not a formal consequence of choosing pointwise inverses.  First choose a global smooth background metric on the determinant flat line.  The rank-one harmonic-reduction equation is a scalar Poisson equation on $X$: rescaling the background by $e^f$ changes the coclosedness defect of the self-adjoint part of the flat connection by a fixed scalar Laplacian applied to $f$.  Solving on mean-zero functions with the fixed Green operator, exactly as in \cref{prop:smooth-det}, gives a global smooth harmonic determinant metric $q$.

Now fix $u_0\in U$, let $\nabla_0$ be the irreducible flat connection there, and choose its Corlette harmonic metric $h_0$ normalized by $\det h_0=q_{u_0}$.  Relative to a metric $h$, write
\[
  \nabla_u=D_{h,u}+\Psi_{h,u},
\]
where $D_{h,u}$ is unitary and $\Psi_{h,u}$ is self-adjoint.  The harmonic-reduction equation is
\[
  D_{h,u}^{*}\Psi_{h,u}=0.
\]
After locally identifying the underlying smooth bundles, parametrizing metrics by trace-free Hermitian logarithms, and conjugating the moving Hermitian target to a fixed one exactly as in \cref{sec:analytic-setup}, this becomes a smooth nonlinear map between Sobolev spaces.

Its vertical derivative at $(\nabla_0,h_0)$ is the standard symmetric-space Jacobi operator.  With the usual normalization its quadratic form is
\[
  \ip{J_{h_0}s}{s}_{L^2}
  =\norm{D_{h_0}s}_{L^2}^2
   +\norm{[\Psi_{h_0},s]}_{L^2}^2.
\]
Hence the kernel consists precisely of self-adjoint endomorphisms parallel for $\nabla_0$.  Irreducibility and Schur's lemma reduce these to real scalars, and determinant normalization removes the scalar direction.  The trace-free operator is therefore self-adjoint elliptic and invertible.  The Banach implicit-function theorem gives a smooth Sobolev family of $q$-normalized harmonic reductions near $u_0$; the parameter-dependent elliptic bootstrap gives joint smoothness on the product.  Repeating around every point of $U$ and using normalized uniqueness makes the local reductions agree on overlaps, so they glue to a global smooth reduction on $E$.

Decomposing $\nabla_u$ with respect to the varying harmonic metric and the fixed complex structure of $X$ gives
\[
  \ddbar_{E,u}=D_{h,u}^{0,1},
  \qquad
  \theta_u=\Psi_{h,u}^{1,0}.
\]
The harmonic-bundle identities imply $\ddbar_{E,u}^2=0$,
$\ddbar_{E,u}\theta_u=0$, and $\theta_u\wedge\theta_u=0$.
Irreducibility corresponds to stability on this locus.  Starting with a
stable Higgs family, choose the determinant normalization compatibly in
the two constructions.  Normalized uniqueness then returns the same
harmonic metric, so decomposition of the flat connection recovers the
original $\ddbar_E$ and $\theta$.  If different determinant
normalizations are chosen, the two harmonic metrics differ by an
$X$-constant positive scalar, and the Chern operator and Higgs adjoint
are still unchanged by \cref{prop:scalar-independence}.  The reverse
composition similarly recovers the original flat connection.  Hence
the two stack morphisms are inverse independently of the auxiliary
normalization.
\end{proof}

\begin{remark}
The analytic construction is local in the parameter and globalizes while every fibre remains irreducible.  The theorem does not assert that a family crossing from irreducible to reducible monodromy admits a single smooth normalized harmonic reduction; the Jacobi kernel can enlarge exactly at such a stabilizer jump.
\end{remark}

Diffeological smoothness is the natural global language for arbitrary
smooth families, but the same fixed-space implicit-function argument
gives the following stronger local statement on analytic plots.

\begin{corollary}[Analytic plots]\label{cor:analytic-plots}
Suppose a plot $U\to\MdiffDol^{\st,0}(X)$ is represented in local coordinates by Higgs data depending real analytically on $u$.  Choose the determinant metric by the analytic Poisson construction of \cref{prop:smooth-det} and the determinant-preserving background by \eqref{eq:analytic-background-prescribed-det}.  Then the normalized harmonic metric and the resulting flat connection depend real analytically on $u$ as maps into each Sobolev completion.
\end{corollary}

\begin{proof}
Use the fixed-domain and fixed-target identifications of \cref{sec:smooth-metrics}.  With the analytic determinant metric supplied by \cref{prop:smooth-det} and the determinant-preserving analytic background in \eqref{eq:analytic-background-prescribed-det}, every coefficient entering the nonlinear moment map depends real analytically on the parameter.  The operations on the logarithmic metric variable---exponentiation, inversion, multiplication, taking the Chern operator, and trace-free projection---are analytic between the chosen Sobolev Banach spaces.  Hence the map
\[
  (u,s)\longmapsto\cF(u,s)
\]
is real analytic near the base solution.

Its vertical derivative is the stable Jacobi operator, which is invertible on the normalized trace-free slice by \cref{prop:stable-kernel}.  The analytic Banach implicit-function theorem gives a real-analytic solution map $u\mapsto s(u)$ into $W^{k+2,2}$.  To improve the Sobolev order, differentiate the elliptic equation with respect to the parameter.  Each derivative satisfies a linear elliptic equation whose right-hand side is analytic in lower derivatives already controlled.  The uniform estimates used in \cref{prop:joint-regularity} therefore give analytic dependence into $W^{j,2}$ for every $j$.

Finally, the relative Chern connection and Higgs adjoint are obtained from the analytic family data and the analytic metric by algebraic operations and one $X$-derivative.  Thus the associated flat connection is real analytic in the parameter as a map into every Sobolev completion in which it is defined.
\end{proof}

This is compatible with the familiar real-analytic nature of the classical correspondence on smooth finite-dimensional moduli loci.  The point here is that no finite-dimensional moduli chart is assumed in advance: the analytic dependence is proved directly for the given family.

Finally, uniqueness makes the construction compatible with restriction, products of parameter spaces, and smooth pullback.  These functorialities are what allow the local analytic theorem to define a morphism of family stacks.

Suppose $U=U_1\times U_2$ and the family is pulled back from $U_1$.  Uniqueness implies that the harmonic metric is also pulled back from $U_1$.  More generally, restriction to an embedded submanifold, a curve in parameter space, or a coordinate slice commutes with the transform.  Thus all directional derivatives computed in the next section are intrinsic to the plot and are independent of the auxiliary extension of a tangent vector to a parameter field.

\begin{proposition}[Curve detection]\label{prop:curve-detection}
Let $u_0\in U$ and $v\in T_{u_0}U$.  If $\gamma_1,\gamma_2:(-\eps,\eps)\to U$ are smooth curves with
\[
  \gamma_i(0)=u_0,
  \qquad
  \dot\gamma_i(0)=v,
\]
then the first derivatives at $0$ of the normalized harmonic metrics pulled back along $\gamma_1$ and $\gamma_2$ coincide after the canonical identification of the central fibre.
\end{proposition}

\begin{proof}
After the local fixed-bundle identification used in the parametric theorem, the normalized harmonic metric is a smooth Banach-valued map
\[
  H:U\longrightarrow W^{k+2,2}(\Herm^+\End E_0)
\]
near $u_0$.  Pullback along $\gamma_i$ gives $H\circ\gamma_i$.  The Banach-space chain rule yields
\[
  \left.\frac{d}{dt}\right|_0H(\gamma_i(t))
  =dH_{u_0}(\dot\gamma_i(0)).
\]
Both curves have tangent vector $v$, so the two derivatives are equal.  A change of local trivialization transports both derivatives through the same identification of the central fibre; consequently the equality is intrinsic after that canonical identification.

Indeed, if the second trivialization differs by a smooth gauge
$g:U\to\Ggauge^{\C}$, then its metric-valued map is obtained from $H$
by the smooth gauge action.  Differentiating that action adds terms
depending only on $g(u_0)$ and $dg_{u_0}(v)$, not on the second jet of
the chosen curve.  The same terms therefore occur for $\gamma_1$ and
$\gamma_2$, so equality of the derivatives is preserved under the
change of trivialization.
\end{proof}

Consequently the differential of the non-Abelian Hodge transform along a diffeological plot is well defined once the gauge-local representative is interpreted as above.

\section{First variation of the harmonic metric}\label{sec:first-variation}

\subsection{Differentiating the harmonic equation and the Green operator}

Let
\[
  t\longmapsto (\cD''_t,h_t)
\]
be a smooth one-parameter family through a stable harmonic bundle $(\cD'',h)$ at $t=0$.  Choose a smooth identification of the underlying bundles and set
\begin{equation}\label{eq:eta-s-def}
  \eta=\dot{\cD}''_0,
  \qquad
  \dot h_0(v,w)=h(s_{\mathrm{tot}}v,w),
\end{equation}
where $s_{\mathrm{tot}}=s_{\mathrm{tot}}^{\dagger_h}$.  Unless the determinant data are fixed, $s_{\mathrm{tot}}$ need not be trace-free.  We therefore decompose
\begin{equation}\label{eq:metric-trace-split}
  s_{\mathrm{tot}}=s_0+\frac{\tau}{r}\Id_E,
  \qquad
  \tr(s_0)=0,
  \qquad
  \tau=\tr(s_{\mathrm{tot}}).
\end{equation}
If $q_t=\det h_t$, then
\begin{equation}\label{eq:tau-det}
  \tau=\left.\frac{d}{dt}\right|_{0}
  \log\frac{q_t}{q_0}.
\end{equation}
Thus $\tau$ is determined by the differentiated rank-one determinant equation of \cref{prop:smooth-det}.

Differentiating the trace-free moment-map equation and using \cref{prop:full-linearization} gives
\begin{equation}\label{eq:first-variation-PDE}
  L_hs_0=-\cS_h(\eta).
\end{equation}
The scalar determinant variation does not enter this equation because multiplication of the metric by an arbitrary positive scalar function on $X$ changes the Chern curvature only by a central term and leaves the Higgs adjoint unchanged; the trace-free projection removes that central variation.

\begin{theorem}[First variation: fixed determinant]\label{thm:first-variation-metric}
At a stable harmonic bundle, in a fixed-determinant direction the logarithmic metric variation is trace-free and is given by
\begin{equation}\label{eq:first-variation-Green}
  s_0=-G_h\cS_h(\eta),
  \qquad
  G_h=L_h^{-1}:
  W^{k,2}(\Herm^0\End E)\to W^{k+2,2}(\Herm^0\End E).
\end{equation}
For a tangent vector arising from a smooth plot, the right side depends only on that tangent vector and not on the chosen representing curve.
\end{theorem}

\begin{proof}
Let $t\mapsto(\cD''_t,h_t)$ be a one-parameter family representing the chosen plot direction, with $h_0=h$, and let $s_0$ be the logarithmic derivative of $h_t$.  Fixed determinant gives $\tr(s_0)=0$.  Differentiating the trace-free moment-map equation at $t=0$ separates the variation into the metric part and the variation of the Higgs data.  By \cref{prop:full-linearization}, this is exactly
\[
  L_hs_0+\cS_h(\eta)=0,
\]
which is \eqref{eq:first-variation-PDE}.

On trace-free Hermitian endomorphisms, $L_h$ is invertible by \cref{prop:stable-kernel}.  Applying its inverse therefore gives
\[
  s_0=-L_h^{-1}\cS_h(\eta)=-G_h\cS_h(\eta).
\]
The source lies in the trace-free Hermitian target because it is the derivative of the trace-free moment map.

For a tangent vector arising from a smooth plot, the normalized metric section is a smooth map from the parameter manifold into the Sobolev space of metrics.  Its derivative depends only on the tangent vector, by \cref{prop:curve-detection}.  Since the Green-operator expression is equal to that derivative, it is independent of the particular tangent curve used to represent the vector.
\end{proof}

\begin{proposition}[First variation with varying determinant]\label{prop:first-variation-general-GL}
For a general $\mathrm{GL}_r$ family, the total logarithmic metric variation is
\begin{equation}\label{eq:first-variation-general-GL}
  s_{\mathrm{tot}}
  =-G_h\cS_h(\eta)+\frac{\tau}{r}\Id_E,
\end{equation}
where $\tau$ is the unique normalized solution obtained by differentiating the scalar determinant Poisson equation \eqref{eq:det-poisson}.  Equivalently, if $q_t=q^0_t e^{f_t}$ is the normalized Hermitian--Einstein determinant metric, then $\tau=\dot f_0$ plus the explicitly known logarithmic variation of $q^0_t$.
\end{proposition}

\begin{proof}
Decompose the logarithmic metric variation into its trace-free and central parts:
\[
  s_{\mathrm{tot}}=s_0+\frac{\tau}{r}\Id_E,
  \qquad \tr(s_0)=0.
\]
After trace-free projection, the central variation drops out of the moment-map equation.  Hence \cref{thm:first-variation-metric} gives
\[
  s_0=-G_h\cS_h(\eta).
\]

It remains to determine $\tau$.  Taking determinants of the family relation defining $s_{\mathrm{tot}}$ gives \eqref{eq:tau-det}.  Write the normalized determinant metric as $q_t=q_t^0e^{f_t}$.  The scalar function $f_t$ solves \eqref{eq:det-poisson} together with the chosen normalization.  Differentiating that equation gives a linear scalar Poisson equation for $\dot f_0$.  On the normalized mean-zero subspace the scalar Laplacian is invertible, so $\dot f_0$ is uniquely determined; the remaining contribution to $\tau$ is the explicitly known logarithmic variation of $q_t^0$.

Combining the uniquely determined trace-free and central components yields
\[
  s_{\mathrm{tot}}
  =-G_h\cS_h(\eta)+\frac{\tau}{r}\Id_E,
\]
as claimed.
\end{proof}

On a stable family, $L_{h_u}$ is invertible on the normalized
trace-free subspace, so the linearized metric correction can be written
with its Green operator.  Parameter-dependent elliptic theory shows
that this inverse itself varies smoothly.

\begin{proposition}[Smooth Green operator]\label{prop:smooth-green}
Let $(E,\cD'',h)$ be a smooth family of normalized stable harmonic bundles over $U$.  Then
\[
  u\longmapsto G_{h_u}
\]
is smooth as a map into
\[
  \mathcal L\bigl(W^{k,2},W^{k+2,2}\bigr)
\]
locally on $U$.  It is real analytic for real-analytic family data.
\end{proposition}

\begin{proof}
After the local fixed-bundle identification, $L_{h_u}$ is a family of bounded operators
\[
  L_{h_u}:W^{k+2,2}(\Herm^0\End E_0)
  \longrightarrow W^{k,2}(\Herm^0\End E_0).
\]
The coefficients of $L_{h_u}$ are smooth functions of the Higgs data and of $h_u$ together with finitely many $X$-derivatives.  Hence $u\mapsto L_{h_u}$ is smooth in operator norm by the parameter-dependent elliptic results of \cref{app:elliptic}.  At every stable point the operator is invertible by \cref{prop:stable-kernel}; after shrinking $U$, invertibility persists because the invertible operators form an open subset of the relevant Banach operator space.

The inversion map on that open subset is smooth.  Therefore
\[
  G_{h_u}=L_{h_u}^{-1}
\]
depends smoothly on $u$ as an element of $\mathcal L(W^{k,2},W^{k+2,2})$.  Differentiating
\[
  L_{h_u}G_{h_u}=\Id
\]
in a parameter direction gives
\[
  dG=-G\,(dL)\,G,
\]
which also makes the smooth dependence explicit.  If the coefficient family is real analytic, the operator-valued map $u\mapsto L_{h_u}$ is analytic and analytic inversion gives the final assertion.
\end{proof}

Thus the first variation formula is itself smooth in the base point of a stable plot.

\subsection{Variation of the flat connection and closedness}

The flat connection is
\[
  \nabla_h=\cD''+\cD'_h.
\]
At fixed metric the variation of $\cD'_h$ is $\eta^{\star_h}$ from \eqref{eq:star-companion}.  Metric variation contributes $\cD'_h s_{\mathrm{tot}}$ by \cref{lem:variation-Dprime}.  Therefore
\begin{equation}\label{eq:flat-variation}
  \dot\nabla
  =\eta+\eta^{\star_h}+\cD'_h s_{\mathrm{tot}}.
\end{equation}
In a fixed-determinant direction this becomes
\begin{equation}\label{eq:dNAH-explicit}
  \dot\nabla
  =\eta+\eta^{\star_h}
  -\cD'_hG_h\cS_h(\eta).
\end{equation}

\begin{definition}[Fixed-determinant infinitesimal NAH operator]\label{def:inf-NAH}
At a normalized stable harmonic bundle define
\begin{equation}\label{eq:inf-NAH-operator}
  \mathfrak N_h(\eta)
  :=\eta+\eta^{\star_h}
  -\cD'_hG_h\cS_h(\eta).
\end{equation}
\end{definition}

The operator just defined is a gauge-dependent representative-level formula.  Along an actual plot it becomes the derivative of the transformed family, while its de Rham cohomology class is independent of the chosen local trivialization.

\begin{theorem}[Differential along a plot]\label{thm:differential-plot}
Let $U\to\MdiffDol^{\st,0}(X)$ be a smooth plot and $v\in T_uU$.  Choose a local smooth trivialization of the underlying bundle near $u$ and a curve tangent to $v$.  Let $\eta_v$ be the resulting infinitesimal Higgs deformation in this gauge.  Then a connection-valued representative of the derivative of the transformed plot is
\begin{equation}\label{eq:differential-plot-formula}
  \dot\nabla_v
  =\eta_v+\eta_v^{\star_{h_u}}
   +\cD'_{h_u}s_{\mathrm{tot},v}.
\end{equation}
If the determinant data are fixed along $v$, this reduces to
\[
  \dot\nabla_v=\mathfrak N_{h_u}(\eta_v).
\]
For the chosen trivialization the representative depends only on $v$, not on the tangent curve.  Under a change of smooth gauge it changes by a $d_{\nabla_{h_u}}$-exact term.  Hence its de Rham cohomology class, and therefore the differential on the stack tangent class represented by the plot, is intrinsic.
\end{theorem}

\begin{proof}
Choose a curve $\gamma(t)$ through $u$ with $\dot\gamma(0)=v$ and use the chosen local trivialization to identify the underlying bundles.  Let $\eta_v$ be the derivative of $\cD''_{\gamma(t)}$ and $s_{\mathrm{tot},v}$ the logarithmic derivative of the corresponding harmonic metrics.  Differentiating
\[
  \nabla_t=\cD''_t+\cD'_{h_t}
\]
produces three terms.  The direct variation of $\cD''_t$ is $\eta_v$.  The variation of $\cD'_h$ caused by changing the Higgs/Dolbeault data while holding $h$ fixed is $\eta_v^{\star_{h_u}}$.  Finally, the metric variation contributes $\cD'_{h_u}s_{\mathrm{tot},v}$ by \cref{lem:variation-Dprime}.  Hence
\[
  \dot\nabla_v
  =\eta_v+\eta_v^{\star_{h_u}}
   +\cD'_{h_u}s_{\mathrm{tot},v},
\]
which proves \eqref{eq:differential-plot-formula}.  If the determinant is fixed along $v$, substitute \cref{thm:first-variation-metric} to obtain $\dot\nabla_v=\mathfrak N_{h_u}(\eta_v)$.

In the fixed trivialization, the transformed plot is a smooth map into an affine Sobolev space of connections.  Its derivative is therefore determined by $v$ and is independent of the representing curve.  Under a change of smooth gauge, differentiating the gauge action on the family of flat connections changes the representative by a $d_{\nabla_{h_u}}$-exact term; this is made explicit in \cref{prop:gauge-compatibility}.  Thus the cohomology class of $\dot\nabla_v$ is independent of the chosen gauge and gives the intrinsic differential on the stack tangent class represented by the plot.
\end{proof}

An infinitesimal connection obtained by differentiating an actual
transformed plot is de Rham closed.  Indeed, because $\nabla_t$ is flat
for every member of the family,
\[
  F_{\nabla_t}=0.
\]
Differentiating at $t=0$ gives
\begin{equation}\label{eq:flat-linearized}
  d_{\nabla_h}\dot\nabla=0.
\end{equation}
Differentiation gives the required closedness statement without assuming that an arbitrary closed deformation integrates.

\begin{proposition}[Plotwise closedness]\label{prop:complex-level-closedness}
Let $\eta$ be the derivative of an actual smooth one-parameter stable Higgs family, and let $s_{\mathrm{tot}}$ be the derivative of its harmonic metric.  Then
\[
  d_{\nabla_h}
  \bigl(\eta+\eta^{\star_h}+\cD'_hs_{\mathrm{tot}}\bigr)=0.
\]
In fixed-determinant directions this says $d_{\nabla_h}\mathfrak N_h(\eta)=0$.
\end{proposition}

\begin{proof}
Let $\nabla_t$ be the flat family obtained from the one-parameter stable Higgs family.  For every $t$,
\[
  F_{\nabla_t}=0.
\]
Write $A=\dot\nabla_0$.  The standard first-variation formula for curvature gives
\[
  \left.\frac{d}{dt}\right|_0F_{\nabla_t}
  =d_{\nabla_h}A.
\]
Indeed, in a fixed trivialization one differentiates $F_{\nabla_t}=dA_t+A_t\wedge A_t$, obtaining $d\dot A+[A_0,\dot A]$.  Since the curvature vanishes identically in $t$, this derivative is zero, and therefore $d_{\nabla_h}A=0$.

Formula \eqref{eq:flat-variation} identifies
\[
  A=\eta+\eta^{\star_h}+\cD'_hs_{\mathrm{tot}}.
\]
This gives the first assertion.  In a fixed-determinant direction, \cref{thm:first-variation-metric} turns the same expression into $\mathfrak N_h(\eta)$.  The proof uses an actual family of flat connections and therefore makes no integrability claim for an arbitrary closed Higgs deformation.
\end{proof}

\begin{warning}[Closed does not mean integrable]
The linearized Higgs equation $\cD''\eta=0$ is only the first-order Maurer--Cartan condition.  At an obstructed point, a closed class need not be tangent to an actual smooth plot.  Therefore plotwise closedness cannot be promoted merely by saying ``extend $\eta$ to an integrable family.''  This distinction is essential when the discussion moves beyond smooth moduli points.
\end{warning}

\subsection{Gauge compatibility and deformation cohomology}

Let $g_t$ be a smooth complex gauge path with $g_0=\Id$ and infinitesimal generator $\xi$.  Applying the stable transform to a family and to its $g_t$-transform gives isomorphic flat families.  Differentiating that isomorphism yields:

\begin{proposition}[Gauge compatibility for plot directions]\label{prop:gauge-compatibility}
For a tangent direction represented by an actual smooth family, changing the representative by the infinitesimal gauge direction $\cD''\xi$ changes the infinitesimal flat deformation by an exact de Rham gauge term:
\begin{equation}\label{eq:gauge-compatibility}
  \dot\nabla_{\mathrm{new}}-\dot\nabla
  =d_{\nabla_h}\Xi_h(\xi)
\end{equation}
for the infinitesimal generator $\Xi_h(\xi)$ of the induced flat gauge path.  Consequently the derivative of the stack morphism is well defined on tangent classes represented by plots.
\end{proposition}

\begin{proof}
Let $g_t$ be a smooth complex gauge path relating the two Higgs-family representatives, with $g_0=\Id$.  Naturality of the non-Abelian Hodge transform produces a smooth path of flat gauge transformations $\widehat g_t$ intertwining the transformed families:
\[
  \nabla_t^{\mathrm{new}}=\widehat g_t\cdot\nabla_t.
\]
Let $\Xi_h(\xi)$ be the infinitesimal generator of $\widehat g_t$ in the gauge-action convention used here.  Differentiating the usual action of a gauge transformation on a connection yields
\[
  \dot\nabla_{\mathrm{new}}-\dot\nabla
  =d_{\nabla_h}\Xi_h(\xi).
\]
This is \eqref{eq:gauge-compatibility}.

The normalized harmonic metric may itself move along the gauge path, so the explicit expression for $\Xi_h(\xi)$ depends on the slice used to select the metric.  Exactness of the difference does not depend on that choice.  Therefore the two infinitesimal connection forms determine the same de Rham cohomology class, which proves that the differential descends from gauge-dependent representatives to tangent classes represented by plots.
\end{proof}

Gauge compatibility identifies the plotwise formula with an intrinsic
tangent class.  It should nevertheless be distinguished from the
classical, purely cohomological comparison: at a harmonic bundle, the
harmonic-bundle K\"ahler identities and the principle of two types
identify the first hypercohomology of the Higgs deformation complex,
viewed as a real vector space, with the first cohomology of the de Rham
deformation complex.  Denote this canonical real-linear isomorphism by
\begin{equation}\label{eq:classical-H1-comparison}
  \mathbf H^1\bigl(\cC^\bullet_{\Dol}(E)\bigr)_{\R}
  \xrightarrow{\;\sim\;}
  H^1\bigl(\cC^\bullet_{\dR}(E,\nabla_h)\bigr).
\end{equation}

\begin{theorem}[Compatibility with the classical tangent comparison]\label{thm:harmonic-representatives}
Suppose a Dolbeault tangent class is represented by an actual smooth stable plot through $(E,\cD'')$.  Then the de Rham cohomology class of the representative \eqref{eq:differential-plot-formula} is the image of that tangent class under the classical comparison \eqref{eq:classical-H1-comparison}.  In particular, at a smooth or unobstructed moduli point where every tangent class is represented by a local curve, the differential of the stable non-Abelian Hodge transform realizes the full classical real-linear cohomology isomorphism.
\end{theorem}

\begin{proof}
Choose the given stable plot and, locally, its smooth family of normalized harmonic metrics.  Differentiating this family gives a degree-one infinitesimal class in the harmonic deformation complex.  The Dolbeault and de Rham forgetful maps send that same class to two representatives: on the Dolbeault side one obtains the Higgs deformation $\eta$, while on the de Rham side one obtains
\[
  \eta+\eta^{\star_h}+\cD'_hs_{\mathrm{tot}},
\]
the derivative of the associated flat connection.  The latter is $d_{\nabla_h}$-closed by \cref{prop:complex-level-closedness}.

The harmonic-bundle K\"ahler identities and the principle of two types identify the degree-one cohomologies of the two forgetful complexes.  Under that identification, two classes agree precisely when they arise from the same harmonic infinitesimal class.  The two representatives above do so by construction; hence the de Rham class of \eqref{eq:differential-plot-formula} is the image of the Dolbeault tangent class under \eqref{eq:classical-H1-comparison}.

A change of local trivialization changes the flat representative by an exact term by \cref{prop:gauge-compatibility}, so the resulting class is intrinsic.  At a smooth or unobstructed moduli point, every tangent class is represented by a local curve; applying the preceding argument to each such class shows that the differential realizes the full classical real-linear isomorphism.
\end{proof}

\begin{remark}
In fixed-determinant directions the explicit representative is $\mathfrak N_h(\eta)$.  The Green correction is the nonlocal term required to keep the varying harmonic metric in the normalized slice.  The theorem does not claim that the formula $\mathfrak N_h(\eta)$, by itself, defines a map on every abstract closed degree-one element at an obstructed point; the universal cohomological isomorphism exists independently of such an integration claim.
\end{remark}

\subsection{The differential as a pseudodifferential expression}

The operator $G_h$ has order $-2$, while $\cS_h$ has order one in the deformation variable and $\cD'_h$ has order one.  Therefore
\[
  \cD'_hG_h\cS_h
\]
is an order-zero pseudodifferential operator on degree-one deformation data.  Formula \eqref{eq:inf-NAH-operator} exhibits the differential of the transform as
\[
  \text{identity plus adjoint minus an order-zero gauge correction}.
\]
This order count yields the Sobolev estimates below.

\begin{proposition}[Sobolev boundedness]\label{prop:dnah-bounded}
For every $j\ge0$, the infinitesimal operator extends to a bounded map
\[
  \mathfrak N_h:W^{j,2}A^1_{\Dol}(\End E)\to W^{j,2}A^1(\End E).
\]
On a compact subset of a stable family with a uniform spectral gap for $L_h$, the operator norms are uniformly bounded.
\end{proposition}

\begin{proof}
For $j\ge1$, the source $\cS_h$ is first order and maps $W^{j,2}$ to $W^{j-1,2}$.  The Green operator gains two derivatives and $\cD'_h$ loses one, so the net correction has order zero.  For $j=0$, interpret the first step as the continuous extension
\[
  \cS_h:L^2\longrightarrow W^{-1,2}.
\]
Elliptic duality gives $G_h:W^{-1,2}\to W^{1,2}$ on the normalized trace-free slice, after which $\cD'_h:W^{1,2}\to L^2$.  Thus the same order-zero estimate holds at the endpoint $j=0$.  Uniformity follows from smooth coefficient bounds and a uniform lower bound on the positive spectrum of $L_h$.
\end{proof}

\section{Polystable families and constant decomposition type}\label{sec:poly}

\subsection{From stability to locally split constant type}

The proof of \cref{thm:intro-smooth-metrics} depends on the inverse of $L_h$ on trace-free Hermitian endomorphisms.  At a polystable point this operator has kernel.  One can still attempt a Lyapunov--Schmidt reduction, but the kernel is geometrically meaningful and can vary in dimension.  It is therefore preferable to separate two questions:
\begin{enumerate}[label=(\alph*)]
\item does the family admit a smooth local decomposition into stable factors of constant multiplicity;
\item given such a decomposition, can the harmonic metrics be assembled smoothly?
\end{enumerate}
The second question is answered below under the stated splitting hypothesis.  The first is a stratification problem and need not hold across stabilizer jumps.

A workable replacement for stability is obtained by restricting first to families whose stable decomposition is locally split and of constant type.  On such a stratum the kernel has constant rank and can be separated from its orthogonal complement.

\begin{definition}[Locally split polystable family]\label{def:locally-split-poly}
A smooth polystable Higgs family $(E,\cD'')$ over $U$ is \emph{locally split of constant type} if every $u_0\in U$ has a neighbourhood $V$ and an isomorphism of smooth Higgs families
\begin{equation}\label{eq:constant-type-splitting}
  (E,\cD'')|_{V\times X}
  \cong
  \bigoplus_{j=1}^m(E_j,\cD''_j)\otimes W_j,
\end{equation}
where:
\begin{enumerate}[label=(\roman*)]
\item each $(E_j,\cD''_j)$ is a smooth family of stable Higgs bundles satisfying the NAH numerical conditions;
\item for every $u\in V$, the stable Higgs bundles $(E_{j,u},\cD''_{j,u})$ are pairwise non-isomorphic;
\item the multiplicity spaces $W_j$ are fixed complex vector spaces.
\end{enumerate}
\end{definition}

The decomposition is not required to be globally ordered on $U$.  A finite monodromy can permute factors of the same numerical type.  Locally one may choose an ordering.

\subsection{Harmonic metrics and centralizer geometry}

Apply \cref{thm:intro-smooth-metrics} to each stable family $(E_j,\cD''_j)$.  Choose normalized smooth harmonic metrics $h_j$.  Let $q_j$ be any positive Hermitian form on $W_j$, constant in $x$ and allowed to vary smoothly in $u$ if desired.  Define
\begin{equation}\label{eq:poly-product-metric}
  h=\bigoplus_{j=1}^m h_j\otimes q_j.
\end{equation}

\begin{theorem}[Smooth harmonic metrics on a constant-type stratum]\label{thm:constant-type}
Every locally split polystable family of constant type admits a smooth family of harmonic metrics locally on the parameter manifold.  More precisely, on a fixed splitting chart and after normalizing the harmonic metrics on the stable factors, all harmonic metrics on a fibre with the direct-sum decomposition \eqref{eq:constant-type-splitting} are exactly those of the form \eqref{eq:poly-product-metric}, with $q_j$ a positive Hermitian form on $W_j$.  In particular, the stable summands belonging to distinct indices are orthogonal for every harmonic metric.
\end{theorem}

\begin{proof}
Each $h_j$ is smooth by \cref{thm:intro-smooth-metrics}.  The total connection associated to $h_j\otimes q_j$ is the tensor product of the flat harmonic connection on $E_j$ with the trivial unitary connection on the constant multiplicity space.  Hence it is flat.  Direct sums of flat connections are flat, so \eqref{eq:poly-product-metric} is harmonic.

For the converse, fix reference forms $q_j^0$ and the product harmonic metric
\[
  h^0=\bigoplus_j h_j\otimes q_j^0.
\]
Let $h$ be any other harmonic metric and write
\[
  h(v,w)=h^0(fv,w)
\]
for the positive $h^0$-self-adjoint endomorphism $f$.  The standard uniqueness argument for harmonic metrics on a reductive flat bundle \cite{Corlette1988,Simpson1992} implies that $f$ is parallel for the harmonic flat connection.  Equivalently, it lies in the positive self-adjoint part of the flat endomorphism algebra.  By stability and pairwise non-isomorphism of the factors,
\[
  \End_{\mathrm{flat}}(E)
  \cong\bigoplus_j \Id_{E_j}\otimes\End(W_j).
\]
Thus $f=\bigoplus_j\Id_{E_j}\otimes f_j$ with each $f_j$ positive and self-adjoint relative to $q_j^0$; in particular, $h$ is automatically orthogonal for the decomposition.  Setting
\[
  q_j(v,w)=q_j^0(f_jv,w)
\]
gives $h=\bigoplus_j h_j\otimes q_j$.  Conversely every such choice was already shown to be harmonic.
\end{proof}

The resulting harmonic metric is not unique.  Its nonuniqueness is governed by positive Hermitian forms on the multiplicity spaces, with the complex centralizer acting transitively and compact unitary groups appearing as isotropy.

Fix reference Hermitian forms $q_j^0$ on $W_j$.  The positive cone
\[
  \operatorname{Herm}^+(W_j)
\]
is the symmetric space
\[
  \operatorname{GL}(W_j,\C)/U(W_j,q_j^0).
\]
Thus the space of all harmonic metrics, after factor normalization, is
\begin{equation}\label{eq:poly-metric-fibre}
  \prod_{j=1}^m\operatorname{Herm}^+(W_j).
\end{equation}
The residual unitary symmetry is
\begin{equation}\label{eq:compact-centralizer}
  K_Z=\prod_{j=1}^mU(W_j,q_j).
\end{equation}
This is a maximal compact subgroup of the complex centralizer
\[
  Z^\C=\prod_{j=1}^m\operatorname{GL}(W_j,\C)
\]
of the polystable Higgs object.

\begin{proposition}[Harmonic metric groupoid on the stratum]\label{prop:metric-groupoid-poly}
Over a local constant-type chart, the groupoid whose objects are harmonic metrics and whose arrows are Higgs automorphisms carrying one metric to another is equivalent to the action groupoid
\begin{equation}\label{eq:poly-action-groupoid}
  \left[
  \prod_j\operatorname{Herm}^+(W_j)
  \bigg/
  \prod_j\operatorname{GL}(W_j,\C)
  \right],
\end{equation}
where the complex centralizer acts on Hermitian forms by pullback.  The isotropy group at a metric $q=(q_j)$ is
\[
  \prod_jU(W_j,q_j),
\]
the compact centralizer of that metric.  This description varies smoothly with the parameter.
\end{proposition}

\begin{proof}
By \cref{thm:constant-type}, after the stable factor metrics have been normalized, every harmonic metric is uniquely specified by a tuple of positive Hermitian forms
\[
  q=(q_1,\ldots,q_m)\in\prod_j\Herm^+(W_j).
\]
Thus the objects of the metric groupoid are exactly the points of the product cone in \eqref{eq:poly-metric-fibre}.

Because the stable factors are pairwise non-isomorphic, every Higgs automorphism preserving the decomposition is block diagonal and has the form
\[
  g=\bigoplus_j\Id_{E_j}\otimes g_j,
  \qquad g_j\in\operatorname{GL}(W_j,\C).
\]
Such an automorphism carries the metric determined by $q$ to the metric determined by the pullback tuple $g^*q=(g_j^*q_j)_j$.  Conversely, every tuple $(g_j)$ gives a Higgs automorphism with exactly this action.  Hence arrows between metric objects are precisely the arrows of the indicated action groupoid.

The isotropy at $q$ consists of those tuples satisfying $g_j^*q_j=q_j$ for every $j$, namely
\[
  \prod_jU(W_j,q_j).
\]
Finally, on a constant-type chart the stable factor metrics and the identifications of the multiplicity spaces vary smoothly.  The product cone, the centralizer action, and the isotropy groups therefore assemble smoothly in the parameter, proving the last assertion.
\end{proof}

\begin{remark}
The quotient in \eqref{eq:poly-action-groupoid} should not be collapsed prematurely to a set.  The centralizer is part of the geometry.  At a stabilizer jump, the change in this groupoid is exactly what a coarse-space description suppresses.
\end{remark}

\subsection{Slices, first variation, and factor monodromy}

The decomposition also gives a canonical complement to the kernel of $L_h$.  Let
\[
  \mathfrak z_h
  =\ker L_h\cap\Herm_h(\End E)
\]
and let
\[
  \mathfrak z_h^\perp
\]
be its $L^2$-orthogonal complement.  Since $L_h$ is self-adjoint,
\begin{equation}\label{eq:L-on-complement}
  L_h:\mathfrak z_h^\perp\cap W^{k+2,2}
  \xrightarrow{\cong}
  \mathfrak z_h^\perp\cap W^{k,2}.
\end{equation}
On a constant-type stratum, the projection to $\mathfrak z_h$ varies smoothly.

\begin{proposition}[Smooth kernel projection]\label{prop:smooth-kernel-projection}
For a locally split constant-type family with chosen smooth factor metrics, the $L^2$-orthogonal projection
\[
  \Pi_u:W^{k,2}(\Herm\End E_u)\to\ker L_{h_u}
\]
varies smoothly in $u$.
\end{proposition}

\begin{proof}
The kernel is explicitly the finite-dimensional bundle
\[
  \bigoplus_j\Herm(W_j)
\]
embedded as $\Id_{E_j}\otimes A_j$.  Choose locally a fixed real basis
$A_1,\ldots,A_N$ of the Hermitian multiplicity endomorphisms at a base
parameter.  If $b_j(u):(W_j,q_j(u))\to(W_j,q_j(u_0))$ is the positive
isometry of multiplicity metrics, conjugate the basis by $b_j(u)$.
The resulting $A_a(u)$ form a smooth basis of
$\Herm_{q_j(u)}(\End W_j)$.  Let $e_a(u)$ be the corresponding embedded
sections $\Id_{E_j}\otimes A_a(u)$.  Their
$L^2$ Gram matrix
\[
  M_{ab}(u)=\ip{e_a(u)}{e_b(u)}_{L^2(h_u)}
\]
is smooth and positive definite, so $M(u)^{-1}$ is smooth.  The
orthogonal projection is given explicitly by
\[
 \Pi_uv=\sum_{a,b}e_a(u)M^{ab}(u)
                  \ip{v}{e_b(u)}_{L^2(h_u)}.
\]
Every coefficient in this finite-rank formula depends smoothly on $u$,
and the $L^2$ pairing is continuous on $W^{k,2}$.  Hence $u\mapsto
\Pi_u$ is smooth in the bounded-operator topology, as asserted.
\end{proof}

Accordingly there is a Green operator on the complement,
\begin{equation}\label{eq:reduced-green}
  G_h^\perp:
  (1-\Pi)W^{k,2}\to(1-\Pi)W^{k+2,2},
\end{equation}
which varies smoothly on the stratum.

Once the kernel directions are separated, the stable first-variation argument survives on the complementary slice.  The resulting formula describes variation tangent to a fixed decomposition-type stratum.

Let a one-parameter family remain in a constant-type stratum.  Write $s$ for the metric variation after choosing a smooth family of multiplicity metrics $q_j(t)$.  Differentiation gives
\[
  L_hs=-\cS_h(\eta).
\]
Solvability requires the right side to be orthogonal to $\ker L_h$.  This holds automatically for an actual family of harmonic solutions.  The solution is determined only modulo $\ker L_h$.

\begin{proposition}[Reduced first-variation formula]\label{prop:reduced-first-variation}
Choose the gauge condition
\[
  \Pi s=0.
\]
Then the transverse metric variation is uniquely
\begin{equation}\label{eq:poly-first-variation}
  s^\perp=-G_h^\perp(1-\Pi)\cS_h(\eta).
\end{equation}
The kernel component $\Pi s$ records the independent variation of the multiplicity-space Hermitian forms.
\end{proposition}

\begin{proof}
Differentiate the harmonic equation along the given one-parameter family.  As in the stable case one obtains
\[
  L_hs=-\cS_h(\eta).
\]
Decompose both sides with respect to the orthogonal splitting determined by $\Pi$.  Since $L_h$ is self-adjoint and annihilates $\ker L_h$, it preserves the orthogonal complement and satisfies
\[
  L_hs^\perp=-(1-\Pi)\cS_h(\eta).
\]
For an actual family of harmonic solutions, the kernel projection of the differentiated equation vanishes automatically; this is the infinitesimal solvability condition.

On $(1-\Pi)W^{k+2,2}$, the operator is invertible with inverse $G_h^\perp$.  Imposing the slice condition $\Pi s=0$ therefore gives the unique transverse solution
\[
  s^\perp=-G_h^\perp(1-\Pi)\cS_h(\eta).
\]
If the slice condition is dropped, one may add any element of $\ker L_h$.  Under the constant-type description of \cref{thm:constant-type}, these kernel directions are precisely the infinitesimal changes of the positive Hermitian forms on the multiplicity spaces.  Thus $\Pi s$ records that independent motion in the harmonic-metric fibre.
\end{proof}

Formula \eqref{eq:poly-first-variation} is the polystable analogue of \eqref{eq:first-variation-Green}.  It makes explicit that a differential exists after a choice of slice and a choice of motion in the harmonic-metric fibre.

Local splittings need not globalize without permutation monodromy among isomorphic stable factors.  We therefore state the descent condition explicitly and keep track of the induced cocycle data.

A family can be locally of constant type while stable factors are permuted around loops in $U$.  Suppose a finite group $\Gamma$ acts by deck transformations on a finite cover $\widetilde U\to U$, and that the pulled-back splitting is equipped with the induced Higgs descent data.  After ordering the stable factors upstairs, the construction above can be made compatible with this finite action at the groupoid level.

\begin{proposition}[Descent under finite factor monodromy]\label{prop:factor-monodromy}
Let a polystable family become locally split of constant type after a finite covering $\widetilde U\to U$.  Assume that the deck transformations act through the actual descent data of the pulled-back Higgs family, permuting isomorphic stable factor types and acting linearly on the corresponding multiplicity spaces.  Then the groupoid of harmonic metrics on the pulled-back family carries compatible descent data and descends to the original parameter space.
\end{proposition}

\begin{proof}
Choose the determinant normalizations of the stable factors equivariantly.  Because the deck group is finite, an arbitrary smooth background determinant metric may first be averaged over the group and the scalar Poisson construction of \cref{prop:smooth-det} is then equivariant by uniqueness of the normalized solution.  The stable theorem therefore gives factor metrics whose pullbacks under deck transformations agree with the metrics on the permuted factors after the prescribed Higgs identifications.

For a multiplicity block $E_j\otimes W_j$, a harmonic metric is
$h_j\otimes q_j$ with $q_j\in\Herm^+(W_j)$.  Start with arbitrary
smooth positive forms on all multiplicity blocks upstairs.  Transport
them by the given deck action and average over the finite group.  More
explicitly, after regarding the direct sum of multiplicity spaces as a
$\Gamma$-equivariant bundle, set
\[
 q^{\mathrm{av}}=
 \frac1{|\Gamma|}\sum_{\gamma\in\Gamma}\gamma^*q^0,
\]
where each pullback includes the prescribed linear identification of
the permuted blocks.  Positivity is preserved by the finite average,
and $q^{\mathrm{av}}$ is $\Gamma$-equivariant.  The product of these
forms with the equivariant stable factor metrics is therefore an
equivariant harmonic metric on the pulled-back family.  The original
Higgs descent maps are now isometries, and their existing cocycle
condition is exactly the cocycle condition in the harmonic-metric
groupoid.  Stack descent produces the asserted object downstairs.
\end{proof}

The descent statement is naturally formulated for the metric groupoid; a globally selected metric is neither required nor generally invariant under the factor monodromy.

\subsection{Stabilizer jumps and the stack of harmonic metrics}

Consider a family in which two stable factors become isomorphic at a special parameter.  Away from the special point the centralizer may be
\[
  \C^\times\times\C^\times,
\]
while at the collision it enlarges to
\[
  \operatorname{GL}_2(\C).
\]
Correspondingly, the Hermitian kernel of $L_h$ jumps from dimension two to dimension four.  No smooth family of inverses to $L_h$ on a fixed trace-free subspace can exist across the jump.

\begin{example}[A model kernel jump]\label{ex:kernel-jump}
Let $(F_t,\cD''_t)$ and $(G_t,\cD''_t{}')$ be stable rank-$r$ families such that $F_t\not\cong G_t$ for $t\ne0$ but $F_0\cong G_0$.  Set
\[
  E_t=F_t\oplus G_t.
\]
For $t\ne0$,
\[
  \End_{\mathrm{Higgs}}(E_t)\cong\C\oplus\C,
\]
while at $t=0$,
\[
  \End_{\mathrm{Higgs}}(E_0)\cong M_2(\C)
\]
under an identification $E_0\cong F_0\otimes\C^2$.  Hence $\dim\ker L_{h_t}$ jumps.  The family can still possess smooth harmonic metrics, but the stable implicit-function chart does not extend across $0$.
\end{example}

Thus a kernel jump alone does not decide existence.  In \cref{prop:coalescing-lines-filtered,prop:square-root-polystable-failure} we give explicit real-analytic kernel-jumping families for which no harmonic metric can be chosen even continuously; the second also has no local relative harmonic filtration.

These centralizer and descent phenomena are naturally groupoidal rather than set-valued.  They lead to a stack of polystable harmonic metrics whose fibres remember both the cone of choices and their isotropy.

Motivated by the preceding analysis, define a fibred groupoid
\[
  \MetHar^{\poly}(X)\to\MdiffDol^{\poly,0}(X)
\]
whose objects over $U$ are polystable Higgs families equipped with smooth harmonic metrics and whose arrows are metric-preserving Higgs isomorphisms.  The central question from \cite{AzamRayan2026} can be reformulated as the local essential surjectivity of this map.  The next theorem proves it on constant-type strata; \cref{cor:full-polystable-failure} shows that it fails on the full polystable locus.

\begin{theorem}[Local essential surjectivity on constant-type strata]\label{thm:essential-surjectivity-constant}
The map
\[
  \MetHar^{\poly}(X)\to\MdiffDol^{\poly,0}(X)
\]
is locally essentially surjective over the locally split constant-type locus.
\end{theorem}

\begin{proof}
Let a plot of $\MdiffDol^{\poly,0}(X)$ take values in the locally split constant-type locus.  Around each parameter value, choose a splitting as in \eqref{eq:constant-type-splitting}.  The stable factor families satisfy the numerical hypotheses, so \cref{thm:intro-smooth-metrics} provides smooth normalized harmonic metrics on them.  Choosing smooth positive Hermitian forms on the fixed multiplicity spaces and applying \cref{thm:constant-type} gives a smooth harmonic metric on the entire polystable family.  Hence the plot locally lifts to an object of $\MetHar^{\poly}(X)$.

If the local constant-type presentation is changed by an isomorphism, transport the constructed harmonic metric along that isomorphism.  Pullback in the parameter commutes with the factorwise stable construction by \cref{prop:pullback-solutions}, and the direct-sum tensor construction of \eqref{eq:poly-product-metric} is functorial.  Therefore the local lifts are legitimate objects of the metric groupoid and are compatible with the stack structure.  This proves local essential surjectivity on the stated locus.
\end{proof}

\subsection{The full polystable locus and Lyapunov--Schmidt reduction}

The general affirmative statement is false by \cref{cor:full-polystable-failure}.  The role of a stratified analysis is therefore not to produce harmonic metrics along every polystable plot, but to characterize liftable strata and the compatibility required when a plot approaches their boundaries.  Three difficulties organize that remaining problem:
\begin{enumerate}[label=(\roman*)]
\item a smooth family of polystable objects need not come equipped with a smooth decomposition into stable factors;
\item the decomposition type and centralizer can jump;
\item even with constant isomorphism type fibrewise, monodromy can act nontrivially on the factor category.
\end{enumerate}

The appropriate stratumwise statement may involve a stratification by stabilizer type and a smooth stack of harmonic reductions rather than a smooth selected metric.  This leads to the following question.

\begin{question}[Stratified relative harmonic metrics]\label{ques:stratified-poly}
Does the polystable diffeological stack admit a locally finite stratification by harmonic automorphism type such that, after restricting both objects and plots to each stratum, the stack of harmonic metrics is a smooth locally trivial groupoid fibration with fibre controlled by the compact centralizer, with the locally split constant-type locus forming an open part of the corresponding strata?
\end{question}

We do not use Question~\ref{ques:stratified-poly} as an input anywhere in the paper.  The question isolates the additional geometry available within liftable strata; \cref{prop:square-root-polystable-failure} rules out using such a stratification to manufacture a continuous lift along every plot crossing their boundaries.

The obstruction to extending this picture across arbitrary stabilizer jumps can be expressed analytically through Lyapunov--Schmidt reduction.  We use that perspective to formulate what remains to be controlled on the full polystable locus.

The same kernel decomposition leads to a Lyapunov--Schmidt reduction near a fixed polystable harmonic point.  Near a polystable harmonic point, decompose
\[
  \Herm^0(\End E)=\mathfrak z_h^0\oplus(\mathfrak z_h^0)^\perp.
\]
Write $s=z+y$ according to this splitting.  The harmonic equation
\[
  \cF(u,z+y)=0
\]
can be projected to the complement and kernel.  The complement equation has invertible $y$-derivative and can be solved as
\[
  y=Y(u,z).
\]
Substituting into the kernel equation yields a finite-dimensional reduced equation
\begin{equation}\label{eq:LS-reduced}
  \Pi\cF(u,z+Y(u,z))=0.
\end{equation}
At a constant-type split family, the variable $z$ parametrizes the harmonic multiplicity metrics and the reduced equation vanishes along the expected symmetric-space family.  At a stabilizer jump, the dimension of $z$ changes.  Thus the Lyapunov--Schmidt reduction does not by itself remove the need for stratification; rather, it exhibits it.

\begin{remark}
A further treatment could combine \eqref{eq:LS-reduced} with local models for reductive group actions and a Kuranishi description of polystable Higgs deformations.  Diffeological plots would then retain smooth parameter families through the resulting singular quotient models.
\end{remark}

\section{Relative harmonic filtrations and the extension-generated stack}\label{sec:filtrations}

\subsection{From extension completion to geometric filtrations}

The stable theorem gives an analytic description of a large family locus.  The semistable stack $\MHDol(X)$ constructed in \cite{AzamRayan2026} is larger: it is generated locally under finite iterated extensions by smooth harmonic families.  We now recast that construction in geometric terms.

Let
\[
  0\longrightarrow(E',\cD''_{E'})
  \xto{i}
  (E,\cD''_E)
  \xto{p}
  (E'',\cD''_{E''})
  \longrightarrow0
\]
be a short exact sequence of smooth relative Higgs bundles over $U\times X$.  Exactness means exactness as smooth vector bundles and compatibility with the Higgs operators.  Choose a smooth Hermitian metric on the middle bundle.  The orthogonal complement of the subbundle $E'$ maps isomorphically onto $E''$, giving a smooth vector-bundle splitting.  In that splitting the middle term can be represented as
\begin{equation}\label{eq:extension-matrix}
  \cD''_E
  =\begin{pmatrix}
    \cD''_{E'} & \beta\\
    0 & \cD''_{E''}
  \end{pmatrix},
\end{equation}
where
\[
  \beta\in A^1_{\Dol,X/U}(\Hom(E'',E'))
\]
satisfies the relative cocycle equation
\begin{equation}\label{eq:extension-cocycle}
  \cD''_{\Hom}\beta=0.
\end{equation}
Changing the smooth splitting changes $\beta$ by a relative coboundary.

\begin{definition}[Relative harmonic filtration]\label{def:relative-harmonic-filtration}
Let $(E,\cD'')$ be a smooth Higgs family over $U$.  A relative harmonic filtration is a finite filtration
\begin{equation}\label{eq:relative-harmonic-filtration}
  0=E_0\subset E_1\subset\cdots\subset E_\ell=E
\end{equation}
by smooth $\cD''$-invariant complex subbundles such that every quotient family
\[
  Q_i=E_i/E_{i-1}
\]
is in the essential image of the smooth harmonic-bundle prestack on the Dolbeault side.
\end{definition}

If the quotients are stable, \cref{thm:intro-smooth-metrics} provides the required smooth harmonic metrics automatically.  More generally a quotient may be a smooth harmonic family without being stable.

The filtration language is not merely suggestive: it is equivalent, locally on the parameter space, to the finite iterated extension completion used in the earlier categorical construction.

\begin{proposition}[Filtration--extension equivalence]\label{prop:filtration-extension}
For a smooth relative Higgs bundle $(E,\cD'')$ over $U$, the following are equivalent:
\begin{enumerate}[label=(\roman*)]
\item $(E,\cD'')$ is a finite iterated extension of smooth harmonic Higgs families;
\item $(E,\cD'')$ admits a relative harmonic filtration.
\end{enumerate}
\end{proposition}

\begin{proof}
Suppose first that $E$ is obtained by a finite sequence of extensions.  Induct on the number of extensions.  For one extension
\[
  0\to E'\to E\to E''\to0
\]
with $E'$ and $E''$ harmonic families, the subbundle $E'$ gives the filtration
\[
  0\subset E'\subset E.
\]
For an iterated extension, append the filtration of the previous stage to the inverse images of the filtration on the quotient.  Exactness in the category of smooth vector bundles ensures that all terms are smooth subbundles.

Conversely, a filtration \eqref{eq:relative-harmonic-filtration} gives short exact sequences
\[
  0\to E_{i-1}\to E_i\to Q_i\to0.
\]
Starting with $E_1=Q_1$ and iterating these extensions expresses $E$ as a finite iterated extension of the harmonic quotient families $Q_i$.
\end{proof}

\subsection{The stack criterion and relative Jordan--H\"older filtrations}

Let $\MHDol(X)$ be the substack from \cite{AzamRayan2026}.  By construction, before stackification its objects are locally generated by finite iterated extensions of harmonic families.  Combining this with \cref{prop:filtration-extension} gives the promised geometric characterization.

\begin{theorem}[Relative harmonic filtration criterion]\label{thm:filtration-criterion}
A smooth Higgs family $(E,\cD'')$ over $U$ defines an object of $\MHDol(X)$ if and only if there exists an open cover $\{U_\alpha\}$ of $U$ such that
\[
  (E,\cD'')|_{U_\alpha\times X}
\]
admits a relative harmonic filtration for every $\alpha$.
\end{theorem}

\begin{proof}
Suppose first that $(E,\cD'')$ defines an object of $\MHDol(X)$.  By construction in \cite{AzamRayan2026}, the stack $\MHDol(X)$ is obtained by stackifying the prestack generated from smooth harmonic families by finite iterated extensions.  Membership is therefore local on the parameter manifold: after passing to an open cover $\{U_\alpha\}$, the restricted family is represented by a finite iterated extension of harmonic families.  By \cref{prop:filtration-extension}, such an iterated extension determines a filtration by smooth invariant subbundles whose successive quotients are precisely those harmonic families.  Hence each restriction admits a relative harmonic filtration.

Conversely, suppose that on $U_\alpha\times X$ there is a relative harmonic filtration
\[
  0=E_0\subset E_1\subset\cdots\subset E_\ell=E.
\]
For each $i$ there is a short exact sequence
\[
  0\to E_{i-1}\to E_i\to Q_i\to0
\]
with $Q_i$ a smooth harmonic family.  Starting from $E_1\cong Q_1$ and iterating these exact sequences expresses $E|_{U_\alpha\times X}$ as a finite iterated extension of harmonic families.  Thus each restriction is an object of the extension-generated prestack.  The original family already comes with descent data on overlaps, and stackification identifies these local objects with a global object of $\MHDol(X)$.  This proves the equivalence.
\end{proof}

\begin{corollary}[Stable families]\label{cor:stable-filtration}
Every smooth stable family satisfying the NAH numerical conditions belongs to $\MHDol(X)$.
\end{corollary}

\begin{proof}
By \cref{thm:intro-smooth-metrics}, the stable family carries a global
smooth harmonic metric after choosing a determinant normalization.  It
is therefore itself an object in the essential image of the harmonic
prestack.  Equivalently, the one-step filtration
\[
  0\subset E
\]
has the harmonic family $E$ as its sole quotient.  The criterion
\cref{thm:filtration-criterion} then places the family in
$\MHDol(X)$.
\end{proof}

\begin{corollary}[Constant-type polystable families]\label{cor:poly-filtration}
Every locally split constant-type polystable family satisfying the NAH numerical conditions belongs to $\MHDol(X)$.
\end{corollary}

\begin{proof}
On every constant-type splitting chart, \cref{thm:constant-type}
constructs a smooth harmonic metric on the restricted family.  Thus the
one-step filtration $0\subset E$ is a relative harmonic filtration on
that chart.  The local criterion \cref{thm:filtration-criterion} then
gives membership in $\MHDol(X)$.  No globally ordered splitting or
globally selected harmonic metric is required for this stack-level
conclusion.
On overlaps, the two local harmonic objects are related by the given
Higgs-family transition isomorphisms together with the induced metric
arrows in the harmonic groupoid.  The original cocycle supplies the
descent datum, so stackification produces the asserted global object.
\end{proof}

The stack criterion becomes especially concrete for semistable families when a Jordan--H\"older filtration can be chosen smoothly.  The numerical conditions must be imposed on the stable quotients at the point where harmonic metrics on those quotients are invoked.

A semistable fibre admits a Jordan--H\"older filtration, but such filtrations are neither unique nor automatically smooth in a parameter.  The correct relative notion is stronger.

\begin{definition}[Relative Jordan--H\"older filtration]\label{def:relative-JH}
A relative Jordan--H\"older filtration of a semistable family $(E,\cD'')$ is a filtration by smooth invariant subbundles
\[
  0=E_0\subset E_1\subset\cdots\subset E_\ell=E
\]
such that every quotient $E_i/E_{i-1}$ is a smooth stable family of the same slope.
\end{definition}

If every stable quotient family in a relative Jordan--H\"older filtration satisfies the NAH numerical conditions, \cref{thm:intro-smooth-metrics} makes that filtration locally relative harmonic.

\begin{corollary}[Relative Jordan--H\"older criterion]\label{cor:relative-JH-criterion}
Let $(E,\cD'')$ be a smooth semistable family.  Suppose it locally admits a relative Jordan--H\"older filtration and that every stable quotient family satisfies the NAH numerical conditions.  Then $(E,\cD'')$ belongs to $\MHDol(X)$.
\end{corollary}

\begin{proof}
Work on an open set on which the relative Jordan--H\"older filtration exists.  Its successive quotients are smooth stable families by definition.  By hypothesis, each quotient also satisfies the NAH numerical conditions.  Applying \cref{thm:intro-smooth-metrics} to each quotient, and shrinking the parameter open set if necessary, gives smooth harmonic metrics on all graded pieces simultaneously.

The original filtration is therefore a relative harmonic filtration in the sense of \cref{def:relative-harmonic-filtration}: the subbundles are smooth and invariant by assumption, and the graded quotients are now equipped with smooth harmonic metrics.  The criterion of \cref{thm:filtration-criterion} then places the family in $\MHDol(X)$ locally.  Since membership in the stack is local on the parameter manifold, the global family defines an object of $\MHDol(X)$.
\end{proof}

The converse need not force the harmonic quotients to be stable; one can refine a harmonic polystable quotient after choosing a relative stable splitting, which may itself have monodromy or stabilizer issues.

\subsection{Upper-triangular forms and the de Rham side}

Choose a smooth splitting of the filtration \eqref{eq:relative-harmonic-filtration} as vector bundles,
\[
  E\cong\bigoplus_{i=1}^\ell Q_i.
\]
Then the Higgs operator has upper-triangular form
\begin{equation}\label{eq:upper-triangular-D}
  \cD''_E=
  \begin{pmatrix}
    \cD''_{Q_1} & \beta_{12} & \beta_{13} & \cdots & \beta_{1\ell}\\
    0 & \cD''_{Q_2} & \beta_{23} & \cdots & \beta_{2\ell}\\
    \vdots & \ddots & \ddots & \ddots & \vdots\\
    0 & \cdots & 0 & \cD''_{Q_{\ell-1}} & \beta_{\ell-1,\ell}\\
    0 & \cdots & \cdots & 0 & \cD''_{Q_\ell}
  \end{pmatrix}.
\end{equation}
The integrability equation is the Maurer--Cartan system
\begin{equation}\label{eq:upper-MC}
  \cD''_{ij}\beta_{ij}
  +\sum_{i<k<j}\beta_{ik}\wedge\beta_{kj}=0.
\end{equation}
Thus a semistable family in $\MHDol(X)$ can be viewed locally as smoothly varying harmonic diagonal blocks coupled by relative extension data satisfying \eqref{eq:upper-MC}.

This matrix is a concrete model for what the extension-generated stack remembers and a coarse S-equivalence quotient forgets.

The same upper-triangular description has a de Rham counterpart.  Passing through the harmonic graded pieces gives a parallel filtration picture for flat families.

The same arguments apply to smooth flat families.  A relative flat filtration is a filtration by smooth $\nabla$-invariant subbundles whose quotients are smooth semisimple harmonic families.  The de Rham analogue of \cref{thm:filtration-criterion} says that an object belongs to $\MHdR(X)$ if and only if it locally admits such a filtration.

Under the equivalence
\[
  \MHDol(X)\simeq\MHdR(X),
\]
upper-triangular Higgs extension data are transformed to upper-triangular flat extension data through the harmonic quotient blocks.  This is the family-level version of Simpson's categorical extension mechanism.

\subsection{Filtration length and relative extension classes}

Define
\[
  \mathscr M_{\Dol}^{\cH,(\le\ell)}(X)
\]
to be the full substack of objects locally admitting a relative harmonic filtration of length at most $\ell$.  Then
\[
  \mathscr M_{\Dol}^{\cH,(\le1)}(X)
  \subset
  \mathscr M_{\Dol}^{\cH,(\le2)}(X)
  \subset\cdots\subset
  \MHDol(X)
\]
and
\[
  \MHDol(X)=\bigcup_{\ell\ge1}\mathscr M_{\Dol}^{\cH,(\le\ell)}(X)
\]
in the local sense.

\begin{proposition}[Filtration-length preservation]\label{prop:length-preservation}
The diffeological non-Abelian Hodge equivalence restricts to equivalences
\[
  \mathscr M_{\Dol}^{\cH,(\le\ell)}(X)
  \simeq
  \mathscr M_{\dR}^{\cH,(\le\ell)}(X)
\]
for every $\ell$.
\end{proposition}

\begin{proof}
Take an object of $\mathscr M_{\Dol}^{\cH,(\le\ell)}(X)$.  Locally it admits a relative harmonic filtration
\[
  0=E_0\subset E_1\subset\cdots\subset E_m=E,
  \qquad m\le\ell.
\]
The diffeological non-Abelian Hodge equivalence is induced by the common harmonic mediator and its finite extension completion.  It sends each harmonic quotient $Q_i=E_i/E_{i-1}$ to the corresponding flat harmonic quotient and sends the extension class defining
\[
  0\to E_{i-1}\to E_i\to Q_i\to0
\]
to the corresponding extension class on the de Rham side.  Iterating this procedure produces a flat filtration with the same number $m$ of steps.  Hence the equivalence maps the Dolbeault substack of length at most $\ell$ into the analogous de Rham substack.

Apply the inverse equivalence to a de Rham object of length at most $\ell$.  The same argument shows that its image has Dolbeault length at most $\ell$.  Therefore the two restricted functors are quasi-inverse equivalences.  In particular the minimal local filtration bound is preserved in both directions.
\end{proof}

The filtration records extension length and isolates the first nontrivial semistable case at length two.

At filtration length two the construction can be written explicitly in terms of a relative extension class.  This case isolates the first mechanism that can fail to assemble smoothly and prepares the obstruction theory of the next section.

Suppose
\[
  0\to H_1\to E\to H_2\to0
\]
with $H_1,H_2$ smooth harmonic families.  In a smooth splitting, $E$ is determined by
\[
  \beta\in A^1_{\Dol,X/U}(\Hom(H_2,H_1)),
  \qquad
  \cD''_{\Hom}\beta=0.
\]
A change of splitting by
\[
  \xi\in A^0(\Hom(H_2,H_1))
\]
changes
\[
  \beta\mapsto\beta+\cD''_{\Hom}\xi.
\]
Thus the family is encoded by a smooth relative cocycle, not merely by a pointwise collection of cohomology classes.

\begin{definition}[Smooth relative extension datum]\label{def:smooth-relative-ext}
A smooth relative extension datum from $H_2$ to $H_1$ is a smooth degree-one relative cocycle $\beta$ modulo smooth relative coboundaries.
\end{definition}

A smooth relative cocycle contains more information than a fibrewise class
\[
  u\mapsto[\beta_u]\in\Ext^1(H_{2,u},H_{1,u}).
\]
The dimensions of the Ext groups may jump, and a pointwise assignment need not admit a smooth cocycle representative.

\section{Obstructions to assembling fibrewise filtrations}\label{sec:obstructions}

\subsection{Invariant subbundles and the projection equation}

Let $(E,\cD'')$ be a smooth semistable family.  Suppose that for every $u\in U$ one has chosen a $\cD''_u$-invariant subbundle
\[
  F_u\subset E_u
\]
of fixed rank.  Even if each $F_u$ is the first term of a Jordan--H\"older filtration, the union
\[
  \bigcup_{u\in U}F_u
\]
need not be a smooth subbundle of $E\to U\times X$.  There are two logically separate issues:
\begin{enumerate}[label=(\roman*)]
\item smooth variation of the underlying subspaces;
\item invariance of the resulting smooth subbundle under $\cD''$.
\end{enumerate}
We develop an obstruction theory for the second issue once a smooth candidate subbundle has been chosen, and then explain how the first issue is related to maps into a relative Grassmannian.

A candidate subbundle is most efficiently encoded by a smooth projection.  Invariance then becomes a nonlinear equation for the off-diagonal second fundamental form, a formulation well suited to linearization.

Fix a smooth Hermitian background metric $\kappa$ on $E$.  A smooth rank-$r'$ subbundle $F\subset E$ is equivalent to a smooth $\kappa$-orthogonal projection
\[
  p\in A^0(\End E),
  \qquad
  p^2=p,
  \quad p^{\dagger_\kappa}=p.
\]
The subbundle is $\cD''$-invariant precisely when
\begin{equation}\label{eq:invariant-projection}
  (1-p)\cD''p=0.
\end{equation}
Indeed, for a section $s$ of $F$ with $ps=s$,
\[
  (1-p)\cD''s=(1-p)(\cD''p)s.
\]

\begin{definition}[Relative second fundamental form]\label{def:relative-second-fundamental}
For a smooth projection $p$, define
\begin{equation}\label{eq:beta-p}
  \beta_p=(1-p)\cD''p
  \in A^1_{\Dol,X/U}(\Hom(F,F^\perp)).
\end{equation}
\end{definition}
Then $F$ is invariant if and only if $\beta_p=0$.

\subsection{Linearization and the first obstruction class}

Let $p_0$ be an invariant projection for a fixed fibre or central parameter.  A tangent vector to the Grassmannian of projections is off-diagonal:
\[
  \dot p=
  \begin{pmatrix}
    0 & \xi^{\dagger}\\
    \xi & 0
  \end{pmatrix}
\]
relative to $E=F\oplus F^\perp$, with
\[
  \xi\in A^0(\Hom(F,F^\perp)).
\]
Suppose the Higgs operator varies by $\eta$.  Differentiate \eqref{eq:beta-p}.

\begin{proposition}[Linearized invariance equation]\label{prop:linearized-invariance}
At an invariant projection $p_0$, the off-diagonal component of the linearized equation is
\begin{equation}\label{eq:linearized-invariance}
  \cD''_{\Hom}\xi
  =-\gamma(\eta,p_0),
\end{equation}
where
\begin{equation}\label{eq:gamma-source}
  \gamma(\eta,p_0)
  =(1-p_0)\eta p_0
\end{equation}
is the component of the Higgs deformation carrying $F$ into the quotient $E/F$.
\end{proposition}

\begin{proof}
Let $p_t$ be a curve of projections with $p_0$ invariant and let
\[
  \cD''_t=\cD''+t\eta+O(t^2).
\]
The invariance equation is
\[
  (1-p_t)\cD''_tp_t=0.
\]
Differentiate at $t=0$.  Using $\dot p=\left.\frac d{dt}\right|_0p_t$ gives
\[
  -\dot p\,\cD''p_0
  +(1-p_0)\eta p_0
  +(1-p_0)\cD''\dot p=0.
\]
Because $p_0$ is invariant, $(1-p_0)\cD''p_0=0$.  Relative to the splitting $E=F\oplus F^\perp$, the tangent vector to the Grassmannian is off-diagonal; write its component $F\to F^\perp$ as $\xi$.  Projecting the differentiated equation from $F$ to $F^\perp$ combines the two terms containing $\dot p$ into the induced Hom differential $\cD''_{\Hom}\xi$.  The remaining source term is precisely
\[
  (1-p_0)\eta p_0=\gamma(\eta,p_0).
\]
Thus
\[
  \cD''_{\Hom}\xi=-\gamma(\eta,p_0),
\]
which is \eqref{eq:linearized-invariance}.  The formula is independent of the chosen orthogonal complement at the level of the resulting cohomology class, as explained immediately below the proposition.
\end{proof}

The linearized equation defines the first obstruction intrinsically in the appropriate Hom deformation complex.  Its formulation does not depend on the auxiliary complement used to write a local block matrix.

Assume $\eta$ satisfies the linearized integrability equation $\cD''\eta=0$.  Let
\[
  \iota:F\hookrightarrow E,
  \qquad
  q:E\twoheadrightarrow E/F
\]
be the canonical inclusion and quotient.  Since $F$ is invariant, the diagonal blocks of $\cD''$ induce differentials on $F$ and $E/F$.  Applying $q(\,\cdot\,)\iota$ to $\cD''\eta=0$ gives
\[
  \cD''_{\Hom}(q\eta\iota)=0.
\]
Under a choice of complement, $q\eta\iota$ is represented by $(1-p_0)\eta p_0$.  The resulting cohomology class is therefore independent of the auxiliary projection $p_0$.

\begin{definition}[First filtration obstruction]\label{def:first-obstruction}
The first obstruction to extending $F$ in the deformation direction $\eta$ is the intrinsic class
\begin{equation}\label{eq:first-obstruction}
  \operatorname{ob}_1(F,\eta)
  =[q\eta\iota]
  =[(1-p_0)\eta p_0]
  \in
  \mathbf H^1\bigl(\cC^\bullet_{\Dol}(\Hom(F,E/F))\bigr),
\end{equation}
where the second representative uses any chosen complement of $F$.
\end{definition}

The class measures exactly whether the linearized invariance equation can be solved.  Its vanishing criterion and the residual freedom in a solution are summarized next.

\begin{theorem}[Infinitesimal lifting criterion]\label{thm:infinitesimal-lifting}
The invariant subbundle $F$ extends to first order along the Higgs deformation $\eta$ if and only if
\[
  \operatorname{ob}_1(F,\eta)=0.
\]
When the obstruction vanishes, the set of first-order lifts is a torsor under
\[
  \mathbf H^0\bigl(\cC^\bullet_{\Dol}(\Hom(F,E/F))\bigr).
\]
\end{theorem}

\begin{proof}
A first-order deformation of the subbundle is represented, in a chosen complement, by an off-diagonal zero-form
\[
  \xi\in A^0(\Hom(F,E/F)).
\]
By \cref{prop:linearized-invariance}, the deformed subbundle is invariant to first order if and only if
\[
  \cD''_{\Hom}\xi=-q\eta\iota.
\]
The source $q\eta\iota$ is $\cD''_{\Hom}$-closed because $\eta$ satisfies the linearized integrability equation and $F$ is invariant.  Therefore the equation has a solution exactly when the cohomology class
\[
  [q\eta\iota]=\operatorname{ob}_1(F,\eta)
\]
vanishes in the first hypercohomology of the Hom deformation complex.  This proves the existence criterion.

Assume now that a solution $\xi_0$ has been chosen.  Any other solution $\xi_1$ satisfies
\[
  \cD''_{\Hom}(\xi_1-\xi_0)=0.
\]
Conversely, adding a closed zero-form to $\xi_0$ gives another solution.  Thus the set of solutions is an affine space under the space of closed degree-zero elements.  Since there are no degree $-1$ gauge parameters in this complex, that space is exactly
\[
  \mathbf H^0\bigl(\cC^\bullet_{\Dol}(\Hom(F,E/F))\bigr).
\]
Hence the set of first-order lifts is a torsor under the stated group.
\end{proof}

\subsection{Higher-order obstruction theory}

Let $t$ be a formal parameter.  Write
\[
  \cD''_t=\cD''+t\eta_1+t^2\eta_2+\cdots
\]
and seek a formal projection
\[
  p_t=p_0+tp_1+t^2p_2+\cdots
\]
satisfying
\[
  p_t^2=p_t,
  \qquad
  (1-p_t)\cD''_tp_t=0.
\]
The idempotent equation determines the diagonal part of $p_j$ from lower-order data, while the invariance equation gives a linear equation for the off-diagonal part.

\begin{proposition}[Inductive obstruction classes]\label{prop:higher-obstructions}
Assume an invariant projection has been constructed modulo $t^N$.  Relative to this chosen lift through order $N$, the obstruction to extending it modulo $t^{N+1}$ is a class
\begin{equation}\label{eq:obN}
  \operatorname{ob}_{N+1}(p_{\le N})
  \in
  \mathbf H^1\bigl(\cC^\bullet_{\Dol}(\Hom(F,E/F))\bigr)
\end{equation}
computed from $\eta_1,\dots,\eta_{N+1}$ and the chosen coefficients $p_1,\dots,p_N$.  Its vanishing is equivalent to the existence of a coefficient $p_{N+1}$ extending that chosen lower-order lift.
\end{proposition}

\begin{proof}
Insert the expansions of $\cD''_t$ and $p_t$ into
\[
  (1-p_t)\cD''_tp_t=0
\]
and collect the coefficient of $t^{N+1}$.  The idempotent equation $p_t^2=p_t$ has already fixed the diagonal part of $p_{N+1}$ in terms of lower-order coefficients.  The only new free variable is therefore the off-diagonal component, say $\xi_{N+1}$.  It appears linearly through the central Hom differential, while all other terms are known from the chosen lift through order $N$.  The order-$N+1$ equation has the form
\[
  \cD''_{\Hom}\xi_{N+1}=-R_{N+1}.
\]

The source $R_{N+1}$ is a universal polynomial expression in $\eta_1,\ldots,\eta_{N+1}$ and $p_1,\ldots,p_N$.  To see that it is closed, apply the induced Hom differential to the order-$N+1$ invariance equation.  The relation $(\cD''_t)^2=0$, together with the invariance equations already satisfied through order $N$, cancels every lower-order term, leaving
\[
  \cD''_{\Hom}R_{N+1}=0.
\]
Hence $R_{N+1}$ defines a class in the indicated first hypercohomology group.

For the fixed lower-order lift and the chosen identifications with the central Hom complex, a coefficient $p_{N+1}$ exists if and only if the equation for $\xi_{N+1}$ is solvable, which is equivalent to $[R_{N+1}]=0$.  This proves the stagewise obstruction statement.  Because $R_{N+1}$ depends on the previously chosen coefficients and on the identifications used to compare moving quotient complexes, the proposition deliberately makes no stronger choice-independence claim.
\end{proof}

These classes form a stagewise obstruction sequence for a chosen formal lift.  They measure whether that lift can be continued, rather than defining choice-free invariants of the central subobject in all orders.

\subsection{A finite-dimensional Kuranishi model}

The preceding obstruction classes are formal.  To obtain a smooth local statement one must use both the ellipticity of the full Hom complex and the nonlinear compatibility identity forced by integrability of the ambient Higgs operator.  It is not correct to argue that vanishing of $\mathbf H^1$ makes the single map
\[
  \cD''_{\Hom}:A^0\longrightarrow A^1
\]
surjective onto all degree-one forms.  Its image has infinite codimension in the unrestricted degree-one space.  What is true is that the nonlinear invariance defect satisfies a Bianchi-type identity, and the elliptic complex reduces the remaining obstruction to the finite-dimensional harmonic space in degree one.

Let
\[
  \cC^\bullet_F
  :=\cC^\bullet_{\Dol}
  \bigl(\Hom(F_0,E_{u_0}/F_0)\bigr),
\]
and write
\[
  d_0:\cC^0_F\to\cC^1_F,
  \qquad
  d_1:\cC^1_F\to\cC^2_F
\]
for its first two differentials.  Choose background Hermitian metrics and let $\cH^i_F$ be the finite-dimensional harmonic spaces of the resulting elliptic complex.  At the Sobolev level the Hodge decompositions are
\begin{align}
  \cC^0_F&=\cH^0_F\oplus\overline{\im d_0^*},\label{eq:Hodge-C0}\\
  \cC^1_F&=\overline{\im d_0}\oplus\cH^1_F
  \oplus\overline{\im d_1^*}.\label{eq:Hodge-C1}
\end{align}
For an elliptic complex on compact $X$ the relevant images are closed in the Sobolev completions, so the overlines may be omitted after the completions have been fixed.

We now record the nonlinear compatibility explicitly.  After shrinking the parameter neighbourhood and choosing a smooth splitting of the underlying bundle,
\[
  E\cong F_0\oplus Q_0,
  \qquad Q_0\cong E_{u_0}/F_0,
\]
write the varying Higgs differential in block form
\begin{equation}\label{eq:block-Du-Kuranishi}
  \cD''_u=
  \begin{pmatrix}
    A_u & B_u\\
    C_u & D_u
  \end{pmatrix}.
\end{equation}
At $u=u_0$ the invariance of $F_0$ gives $C_{u_0}=0$.  A nearby subbundle transverse to $Q_0$ is the graph of a section
\[
  \xi\in A^0(\Hom(F_0,Q_0)).
\]
Its invariance defect is the Riccati-type expression
\begin{equation}\label{eq:Psi-graph}
  \Psi(u,\xi)
  =C_u+D_u\xi-\xi A_u-\xi B_u\xi,
\end{equation}
where $D_u\xi-\xi A_u$ denotes the induced first-order Hom operator.  Thus the graph is invariant exactly when $\Psi(u,\xi)=0$, and
\[
  D_\xi\Psi(u_0,0)=d_0.
\]

The identity $(\cD''_u)^2=0$ implies a nonlinear Bianchi identity
\begin{equation}\label{eq:Psi-Bianchi}
  d_{1,u,\xi}\Psi(u,\xi)=0,
\end{equation}
where $d_{1,u,\xi}$ is the degree-one differential induced on
$\Hom(\operatorname{graph}\xi,E/\operatorname{graph}\xi)$ after using the fixed splitting to identify the bundles.  In particular, $d_{1,u,\xi}$ is a smooth first-order perturbation of $d_1$.  Formula \eqref{eq:Psi-Bianchi} can also be checked directly by multiplying the block matrix \eqref{eq:block-Du-Kuranishi}: it is the lower-left block of $(\cD''_u)^2$ after the graph change of variables.  This compatibility is the ingredient that removes the infinite-dimensional coexact part of the target.

\begin{theorem}[Kuranishi model for an invariant subbundle]\label{thm:kuranishi-subobject}
Let $(E,\cD''_u)$ be a smooth Higgs family near $u_0$, and let $F_0\subset E_{u_0}$ be a $\cD''_{u_0}$-invariant subbundle.  Fix Sobolev indices high enough that graph composition and multiplication are smooth.  After choosing the splitting above and shrinking the parameter neighbourhood, there exist neighbourhoods
\[
  B\ni u_0,
  \qquad
  Z\ni0\text{ in }\cH^0_F,
\]
and a smooth finite-dimensional map
\begin{equation}\label{eq:kuranishi-subobject-map}
  \kappa_F:B\times Z\longrightarrow\cH^1_F
\end{equation}
such that nearby invariant subbundles transverse to $Q_0$, in the chosen graph slice, are in bijection with the zero set $\kappa_F^{-1}(0)$.

For a parameter direction represented by $\eta=\dot{\cD}''_0$, the derivative of the Kuranishi map is
\begin{equation}\label{eq:kuranishi-derivative-ob1}
  D_u\kappa_F(u_0,0)[\eta]
  =\mathbb H_1\bigl((1-p_0)\eta p_0\bigr),
\end{equation}
where $\mathbb H_1$ is harmonic projection in degree one.  Hence the cohomology class of this derivative is the first obstruction $\operatorname{ob}_1(F_0,\eta)$.
\end{theorem}

\begin{proof}
Let $P_{\mathrm{ex}}$ and $P_{\mathrm{har}}$ denote the projections in \eqref{eq:Hodge-C1} onto $\im d_0$ and $\cH^1_F$, respectively.  Decompose the graph variable as
\[
  \xi=z+y,
  \qquad
  z\in\cH^0_F,
  \quad
  y\in(\ker d_0)^\perp.
\]
The derivative at the origin of
\[
  (u,z,y)\longmapsto P_{\mathrm{ex}}\Psi(u,z+y)
\]
in the $y$-direction is
\[
  d_0:(\ker d_0)^\perp\cap W^{k+1,2}
  \xrightarrow{\;\cong\;}
  \im d_0\cap W^{k,2}.
\]
The inverse is bounded by the Green operator of the degree-zero Laplacian $d_0^*d_0$.  The Banach implicit-function theorem therefore gives a unique smooth map
\[
  y=Y(u,z)
\]
for which
\begin{equation}\label{eq:exact-projection-solved}
  P_{\mathrm{ex}}\Psi(u,z+Y(u,z))=0.
\end{equation}
Define
\begin{equation}\label{eq:kappa-def}
  \kappa_F(u,z)
  :=P_{\mathrm{har}}\Psi(u,z+Y(u,z)).
\end{equation}
This is finite-dimensional and smooth.

It remains to justify that $\kappa_F=0$ is equivalent to the full equation, rather than only to the vanishing of two projections.  Under \eqref{eq:exact-projection-solved} and $\kappa_F=0$, the Hodge decomposition places the residual defect in $\im d_1^*$:
\[
  \Psi=d_1^*\beta.
\]
The coexact summand is
\[
  \overline{\im d_1^*}=(\ker d_1)^\perp.
\]
On this subspace the elliptic estimate for the complex gives
\begin{equation}\label{eq:coexact-estimate}
  \norm{\Psi}_{W^{k,2}}
  \le C\norm{d_1\Psi}_{W^{k-1,2}}.
\end{equation}
The nonlinear Bianchi identity \eqref{eq:Psi-Bianchi} may be written
\[
  d_1\Psi=-(d_{1,u,\xi}-d_1)\Psi.
\]
Because $d_{1,u,\xi}-d_1$ has coefficients tending to zero with $(u-u_0,\xi)$, shrinking $B\times Z$ makes its operator norm in \eqref{eq:coexact-estimate} smaller than $C^{-1}$.  Hence $\Psi=0$.  The converse is immediate.  Therefore the zero set of \eqref{eq:kappa-def} parametrizes exactly the nearby invariant graphs in the chosen slice.

Finally, differentiate at $(u_0,0)$.  The derivative of $\Psi$ in the parameter direction is the lower-left block $(1-p_0)\eta p_0$.  The derivative contributed by $Y$ lies in $\im d_0$ and is killed by $P_{\mathrm{har}}$.  This gives \eqref{eq:kuranishi-derivative-ob1}.
\end{proof}

\begin{corollary}[Smooth lifting when the obstruction space vanishes]\label{thm:smooth-lifting-subobject}
Under the hypotheses of \cref{thm:kuranishi-subobject}, suppose
\begin{equation}\label{eq:H1-vanishing}
  \mathbf H^1\bigl(\cC^\bullet_F\bigr)=0.
\end{equation}
Then, after shrinking the parameter neighbourhood, $F_0$ extends to a smooth family of $\cD''_u$-invariant subbundles of the same rank.  The family need not be unique when $\mathbf H^0(\cC^\bullet_F)\ne0$.
\end{corollary}

\begin{proof}
Condition \eqref{eq:H1-vanishing} is equivalent to $\cH^1_F=0$.  The Kuranishi map therefore has zero target.  Taking $z=0$ gives a solution
\[
  \xi(u)=Y(u,0)
\]
for every $u$ sufficiently close to $u_0$.  To obtain joint smoothness, repeat the implicit-function construction at each higher Sobolev level.  The coefficients of \eqref{eq:Psi-graph} are smooth, so the same linearized isomorphism persists after increasing the Sobolev index.  The higher-regularity solution agrees with the original one by uniqueness in the lower-level graph slice.  Hence $u\mapsto\xi(u)$ is smooth into every Sobolev completion, and Sobolev embedding yields a jointly smooth graph subbundle over the parameter neighbourhood times $X$.  Nonzero $\cH^0_F$ records the expected local nonuniqueness.
\end{proof}

\begin{remark}
The role of the Bianchi identity is essential.  Without \eqref{eq:Psi-Bianchi}, the coexact summand $\im d_1^*$ in \eqref{eq:Hodge-C1} would remain an infinite-dimensional equation, and the claimed finite-dimensional reduction would be false.  This is why the full elliptic Hom complex, rather than only the first differential, is the correct deformation object.
\end{remark}

Apply \cref{thm:smooth-lifting-subobject} successively to the terms of a Jordan--H\"older filtration.

\begin{corollary}[Acyclic criterion for relative Jordan--H\"older filtrations]\label{cor:acyclic-JH}
Let
\[
  0=F_0\subset F_1\subset\cdots\subset F_\ell=E_{u_0}
\]
be a Jordan--H\"older filtration of the central fibre.  Suppose that at each stage
\[
  \mathbf H^1\bigl(
  \cC^\bullet_{\Dol}(\Hom(F_i/F_{i-1},E_{u_0}/F_i))
  \bigr)=0
\]
for the corresponding quotient lifting problem.  Then, after shrinking the parameter neighbourhood, the filtration extends to a relative filtration by smooth invariant subbundles.  If the resulting quotient families remain stable and satisfy the NAH numerical conditions, this is a relative harmonic filtration and the family belongs to $\MHDol(X)$ locally.
\end{corollary}

\begin{proof}
We argue inductively.  The first step applies \cref{thm:smooth-lifting-subobject} to $F_1\subset E_{u_0}$, with quotient $E_{u_0}/F_1$, and produces a smooth invariant family $\mathcal F_1\subset E$ after shrinking the base.  Suppose smooth invariant families
\[
  0=\mathcal F_0\subset\mathcal F_1\subset\cdots\subset\mathcal F_{i-1}
\]
have already been constructed.  Form the smooth relative quotient $E/\mathcal F_{i-1}$.  The central fibre contains the invariant subbundle $F_i/F_{i-1}$.  By hypothesis, the first hypercohomology of the Hom complex governing this quotient lifting problem vanishes.  Applying \cref{thm:smooth-lifting-subobject} in the quotient yields a smooth invariant subbundle near $u_0$; its inverse image in $E$ is $\mathcal F_i$.  Since only finitely many stages occur, the parameter neighbourhood may be shrunk once so that all steps are simultaneously defined.

The successive quotients $\mathcal F_i/\mathcal F_{i-1}$ are smooth Higgs families.  Under the final stability and numerical hypotheses, \cref{thm:intro-smooth-metrics} supplies local smooth harmonic metrics on them.  The resulting filtration is therefore relative harmonic in the sense of \cref{def:relative-harmonic-filtration}, and \cref{thm:filtration-criterion} gives local membership in $\MHDol(X)$.
\end{proof}

\subsection{Lifting filtrations and the geometry of constant type}

The same lifting problem can be viewed geometrically as a section problem for the relative Grassmannian.  This viewpoint clarifies why the projection equation is coordinate independent and how local Kuranishi charts fit together.

A rank-$r'$ smooth subbundle of $E\to U\times X$ is a smooth section of the relative Grassmannian bundle
\[
  \operatorname{Gr}_{r'}(E)\to U\times X.
\]
Thus a fibrewise collection $F_u$ assembles smoothly exactly when the induced map
\[
  U\times X\to\operatorname{Gr}_{r'}(E)
\]
is smooth.  In many moduli problems one knows only that $F_u$ exists abstractly, with no canonical choice.  The set of possible $F_u$ can itself jump or carry monodromy.  This is the zeroth-order obstruction preceding the cohomological obstruction \eqref{eq:first-obstruction}.

\begin{definition}[Filtration-choice space]\label{def:filtration-choice-space}
For a family $(E,\cD'')$, let $\mathscr F_{r'}(E,\cD'')$ be the fibred space over $U$ whose fibre at $u$ is the set, or more naturally groupoid, of rank-$r'$ invariant Higgs subbundles of $E_u$ with prescribed slope and numerical type.
\end{definition}

A relative subbundle is a smooth section of this fibred object.  The geometry of $\mathscr F_{r'}$ can be singular.  Our obstruction classes describe the infinitesimal failure of local sections once a central choice is fixed.

When the Jordan--H\"older type is constant, the preceding machinery suggests a tractable regime in which successive subbundles may be lifted inductively.  We state precisely what the present argument proves and what additional surjectivity hypotheses are still required.

It is tempting to conjecture that constant Jordan--H\"older type is sufficient for a relative filtration.  This requires care: constancy of the isomorphism classes of graded factors does not by itself provide smooth subbundles or rule out monodromy.

\begin{assumption}[Strongly constant Jordan--H\"older type]\label{ass:strong-JH}
A semistable family has strongly constant Jordan--H\"older type on $U$ if:
\begin{enumerate}[label=(\roman*)]
\item the stable graded factors form smooth stable families after a finite cover of $U$;
\item their multiplicities are constant;
\item the successive invariant-subobject choice spaces admit compatible
local smooth sections that form a filtration;
\item the extension classes admit smooth relative cocycle representatives.
\end{enumerate}
\end{assumption}

These hypotheses are stronger than mere fibrewise constancy, but they isolate a regime in which the geometric filtration criterion can be applied without further obstruction theory.

\begin{proposition}[Strong constant type implies membership]\label{prop:strong-constant-membership}
Let a semistable family satisfy \cref{ass:strong-JH}.  Assume in addition that every stable graded-factor family satisfies the NAH numerical conditions.  Then the family belongs locally to $\MHDol(X)$.
\end{proposition}

\begin{proof}
Work on a neighbourhood on which the local sections in
\cref{ass:strong-JH}(iii) have been chosen.  They give smooth invariant
subbundles of the prescribed ranks.  The smooth relative cocycle data
in (iv) identify the successive quotients and extension maps, so these
subbundles form a relative filtration whose graded pieces are the
stable factor families supplied by (i)--(ii).  By the additional
numerical hypothesis, \cref{thm:intro-smooth-metrics} gives a smooth
harmonic metric on every graded factor; since there are finitely many
factors, the neighbourhood may be shrunk once for all of them.  The
filtration is therefore relative harmonic, and
\cref{thm:filtration-criterion} gives local membership in
$\MHDol(X)$.
\end{proof}

The proposition isolates the geometric input required by this route to membership in $\MHDol(X)$: smooth relative subobjects, smooth extension cocycles, and harmonic graded factors.

These results lead to a sharper form of the family-level question left open in our earlier paper: one must determine whether constant type alone forces the needed local lifting conditions, or whether further obstruction strata occur.

Although the universal form of Question 5.3.2 of \cite{AzamRayan2026} will be disproved by \cref{prop:square-root-polystable-failure}, the filtration criterion decomposes the characterization of its positive locus into concrete subproblems.

\begin{problem}[Semistable family problem]\label{prob:semistable-family}
Determine conditions under which a smooth family of semistable Higgs bundles admits local relative harmonic filtrations.  In particular:
\begin{enumerate}[label=(\alph*)]
\item characterize the local smoothness of the invariant-subobject choice spaces;
\item compute the obstruction classes \eqref{eq:first-obstruction} and \eqref{eq:obN};
\item determine whether a suitable stratification by Jordan--H\"older and stabilizer type forces the obstructions to vanish;
\item compare these conditions with fibrewise quasi-fullness properties of the harmonic forgetful dg-functor.
\end{enumerate}
\end{problem}

This formulation shows that the remaining semistable issue is not simply the existence of harmonic metrics.  It is a relative subobject and extension problem.

\section{Finite regularity and singular parameter spaces}\label{sec:finite-regularity}

Question~5.3.4 of \cite{AzamRayan2026} asked what remains of the family-level correspondence when the parameter dependence is weakened from smooth to $C^d$, and it suggested that one should eventually allow families over $C^\infty$-schemes.  The analytic constructions above give a more definite answer than the earlier paper could provide.  The decisive point is that the moment-map equation was placed on the ambient affine Sobolev space \eqref{eq:universal-moment-map}, rather than on the integrable Higgs locus itself.  The latter can be singular, but the implicit-function theorem does not need it to be a Banach manifold.

We first formulate the mixed regularity used below.  The definition is local and is arranged so that parameter regularity can be read as regularity of maps into the Fr\'echet space of smooth coefficients on $X$.

\subsection{Mixed \texorpdfstring{$C^d$}{Cd} parameter regularity and stable harmonic metrics}

\begin{definition}[Mixed $C^d_UC^\infty_X$ family]\label{def:mixed-Cd-family}
Let $d\in\{0,1,2,\ldots,\infty\}$ and let $U$ be a $C^d$ manifold.  A vector-bundle family $E\to U\times X$ is of class $C^d_UC^\infty_X$ if, for every $u_0\in U$, there is a neighbourhood $V$ and an isomorphism of complex bundles
\[
  E|_{V\times X}\cong \pi_X^*E_{u_0}
\]
whose local matrix coefficients, and those of its inverse, are smooth in the $X$-variables and $C^d$ in $u$ with values in $C^\infty(X)$.  Relative tensor fields are required to have the same mixed regularity in such a family trivialization.  For $d=0$ this means continuity into $C^\infty(X)$ with its usual Fr\'echet topology.  A Higgs family $(E,\cD'')$ is a $C^d_UC^\infty_X$ family if $E$ has this local family-triviality and the coefficients of $\cD''$ have the same regularity, with every slice integrable.
\end{definition}

Equivalently, in a local family trivialization every coefficient map is $C^d$ into each Sobolev completion.  The local product-triviality in the definition is the bundle-theoretic hypothesis needed to place the nonlinear equation on fixed Banach spaces.  In the smooth case it is obtained by the parameter-direction connection used in \cref{sec:analytic-setup}.  For finite $d$ it should be viewed as part of the chosen model of a mixed-regularity family; extending the theorem to a broader sheaf-theoretic category without this explicit triviality is a separate bundle-regularity question.  Apart from that point, the definition matches the picture proposed in Question~5.3.4 of \cite{AzamRayan2026}: all $X$-derivatives vary with the prescribed parameter regularity, and the relative operators do not differentiate along $U$.

\begin{theorem}[Finite-regularity stable harmonic metrics]\label{thm:Cd-stable-metrics}
Let $d\in\{0,1,2,\ldots,\infty\}$, let $U$ be a $C^d$ manifold, and let
\[
  (E,\cD'')\longrightarrow U\times X
\]
be a $C^d_UC^\infty_X$ family of stable Higgs bundles satisfying the NAH numerical conditions fibrewise.  Then the following hold globally over $U$.
\begin{enumerate}[label=(\roman*)]
\item The determinant family admits a Hermitian--Einstein metric $q$ of class $C^d_UC^\infty_X$ on all of $\det E$.
\item There is a harmonic metric $h$ on all of $E$ of class $C^d_UC^\infty_X$ with $\det h=q$.
\item Once $q$ is prescribed, $h$ is unique.
\end{enumerate}
For $d=0$, the conclusion means that
\[
  u\longmapsto h_u
\]
is continuous as a map into the Fr\'echet space of smooth Hermitian metrics on the fixed fibre bundle after local trivialization.
\end{theorem}

\begin{proof}
We separate the determinant equation from the trace-free equation, as in the smooth proof.  A $C^d$ partition of unity (a continuous partition when $d=0$) applied to local metrics gives a global positive background metric on $\det E$ of class $C^d_UC^\infty_X$.  The scalar source in the determinant Poisson equation \eqref{eq:det-poisson} is globally defined and, in every family chart, is $C^d$ as a map into $C^\infty(X)$.  Composition with the fixed inverse Laplacian on mean-zero functions gives a global determinant metric $q$ of the same mixed regularity.

Now fix $u_0\in U$ and choose a mixed-regularity local trivialization of the family over a coordinate neighbourhood $V$ of $u_0$.  By \cref{def:mixed-Cd-family}, the coefficient maps
\[
  u\longmapsto \cD''_u
\]
are $C^d$ into every Sobolev affine space used in \cref{sec:analytic-setup}.

We next solve the trace-free equation.  The proof of \cref{thm:sobolev-family} can be written with the full local coefficient datum
\[
  p=(\cD'',q)
\]
as the parameter.  More precisely, after fixing the base harmonic metric $h_0$, the constructions \eqref{eq:transported-s}--\eqref{eq:stable-F} define, on a neighbourhood of the base datum $p_0$, a smooth Banach map
\begin{equation}\label{eq:Cd-full-moment-map}
  \widehat\Phi_k:\cP_k\times \cS^0_{k+2}\longrightarrow \cT^0_k,
\end{equation}
where $\cP_k$ is a local Banach space in which the Higgs coefficients
have $W^{k+1,2}$ regularity and the determinant backgrounds have
$W^{k+2,2}$ regularity.  The extra derivative on the Higgs coefficient
is required by the curvature term, as explained after
\eqref{eq:ambient-Higgs-affine-space}; the determinant background enters
through the metric chart and is taken at the metric regularity.  The
integrability equation is not imposed in the definition of $\cP_k$.
At $(p_0,0)$,
\[
  D_s\widehat\Phi_k(p_0,0)=L_{h_0},
\]
which is an isomorphism by \cref{prop:stable-kernel}.  The Banach implicit-function theorem therefore produces a smooth local solution operator
\begin{equation}\label{eq:Cd-solution-operator}
  \Sigma_k:\cO_k\subset\cP_k\longrightarrow \cS^0_{k+2}
\end{equation}
with
\[
  \widehat\Phi_k(p,\Sigma_k(p))=0.
\]
This is the same implicit-function argument as before, but with the coefficient datum itself, rather than a preassigned smooth manifold map $u\mapsto p_u$, used as the external variable.

Choose one Sobolev index $k_0$ satisfying \eqref{eq:k-assumption}.  After shrinking $V$ so that the coefficient datum stays in the implicit-function neighbourhood $\cO_{k_0}$, the mixed family gives a $C^d$ map
\[
  p:V\longrightarrow\cO_{k_0}\subset\cP_{k_0},
\]
and hence
\[
  s(u)=\Sigma_{k_0}(p_u)
\]
is $C^d$ as a map to $W^{k_0+2,2}$.  At this stage we do not intersect a countable family of implicit-function neighbourhoods.  Instead we bootstrap on one fixed, slightly smaller neighbourhood.

The solution satisfies a quasilinear elliptic equation in the $X$-variables whose coefficients are $C^d$ in $u$ with values in $C^\infty(X)$.  Ordinary elliptic regularity first makes every $s(u)$ smooth in $x$.  To control the parameter dependence at higher $X$-orders, use the uniform elliptic estimate \eqref{eq:uniform-elliptic-estimate} for the linearization
\[
  L_{u,s(u)}=D_s\widehat\Phi_{k_0}(p_u,s(u)).
\]
For $d\ge1$, differentiating the equation once in a parameter direction gives
\[
  L_{u,s(u)}(\partial_{u_i}s)
  =-D_p\widehat\Phi_{k_0}(p_u,s(u))\,\partial_{u_i}p.
\]
Interpreting the same differential expression in higher Sobolev scales, the right-hand side has arbitrarily high $X$-regularity because the coefficient datum does.  Uniform elliptic estimates therefore improve the $X$-regularity of $\partial_{u_i}s$.  Repeated differentiation gives the same conclusion for every mixed parameter derivative of order at most $d$; at each step the right-hand side involves only coefficient derivatives and lower parameter derivatives of $s$ already controlled.  This is the finite-regularity version of the bootstrap in \cref{prop:parametric-bootstrap,prop:joint-regularity}.

For $d=0$, no differentiation in $u$ is available or needed.  Subtract the equations at $u$ and $v$ and use the Banach-space mean-value formula in the metric variable to write
\[
  \overline L_{u,v}\bigl(s(u)-s(v)\bigr)=R_{u,v},
\]
where $\overline L_{u,v}$ is the average of the metric linearizations along the segment joining $s(v)$ to $s(u)$ and $R_{u,v}$ records the change of the coefficient datum.  After shrinking the neighbourhood, $\overline L_{u,v}$ remains uniformly strongly elliptic, while $R_{u,v}\to0$ in each Sobolev norm for which the already established lower-order terms are controlled.  The uniform elliptic estimate applied inductively therefore gives continuity of $s$ into every $W^{j,2}$.

Consequently, on one fixed neighbourhood of $u_0$, the map $u\mapsto s(u)$ is $C^d$ into every Sobolev space.  Sobolev embedding at arbitrarily high order gives the mixed regularity asserted in the theorem.  In particular, for every $\ell$ the map is $C^d$ into $C^\ell(X)$.  Hence $s$, and therefore
\[
  h_u=\kappa_u\exp(b_u^{-1}s(u)b_u),
\]
is of class $C^d_UC^\infty_X$.

The trace-free moment-map equation together with the determinant equation gives the full Hermitian--Yang--Mills--Higgs equation.  The fibrewise NAH numerical conditions then imply flatness of the associated harmonic connection exactly as in the proof of \cref{thm:intro-smooth-metrics}.  Repeating the local construction near every point of $U$ gives $q$-normalized mixed-regularity metrics on a cover.  Uniqueness for prescribed determinant is fibrewise \cref{prop:normalized-uniqueness}; it forces these local solutions to agree pointwise on overlaps.  They therefore glue to the asserted harmonic metric on all of $E$.
\end{proof}

\begin{remark}[Why finite parameter regularity is not lost]\label{rem:no-loss-Cd}
The equation differentiates only in the $X$-direction.  Parameter derivatives enter solely through the coefficient map $u\mapsto p_u$ and composition with the smooth solution operator \eqref{eq:Cd-solution-operator}.  Thus the argument does not consume one parameter derivative at each elliptic step.  This is the analytic counterpart of the observation in Question~5.3.4 of \cite{AzamRayan2026} that the partial Dolbeault and exterior differentials do not differentiate along $U$.
\end{remark}

The theorem gives a direct finite-regularity analogue of the stable plotwise transform.  To state it without rebuilding all of the abstract differential formalism, let $\mathscr M_{\Dol,d}^{\st,0}(X)$ and $\mathscr M_{\dR,d}^{\irr}(X)$ denote the stackifications, on the open-cover site of $C^d$ manifolds, of the evident prestacks of mixed $C^d_UC^\infty_X$ stable Higgs families satisfying the numerical conditions and mixed $C^d_UC^\infty_X$ irreducible flat families, respectively.  These are the finite-regularity versions anticipated in \cite[Question~5.3.4]{AzamRayan2026}.

\begin{corollary}[Finite-regularity stable transform]\label{cor:Cd-stable-transform}
For every $d\in\{0,1,2,\ldots,\infty\}$, the harmonic construction defines a natural transformation
\[
  \operatorname{NAH}^{\st}_d:
  \mathscr M_{\Dol,d}^{\st,0}(X)
  \longrightarrow
  \mathscr M_{\dR,d}^{\irr}(X).
\]
For $d=\infty$ it agrees with the smooth stable transform of \cref{thm:smooth-stable-transform}.
\end{corollary}

\begin{proof}
By \cref{thm:Cd-stable-metrics}, a stable $C^d_UC^\infty_X$ family has a global harmonic metric of the same mixed regularity.  The operations forming the relative Chern connection and $\theta^{\dagger_h}$ involve $X$-differentiation and smooth algebraic operations on the coefficients of $h$ and $h^{-1}$.  They therefore preserve the mixed regularity.  Hence
\[
  \nabla_h=\cD''+\cD'_h
\]
is a $C^d_UC^\infty_X$ flat family.  Pullback compatibility follows from uniqueness of the normalized metric, exactly as in \cref{prop:pullback-solutions}; scalar independence removes the dependence on the chosen global normalization.  This gives the stated natural transformation.
More explicitly, a morphism of Higgs families transports a harmonic
metric to another harmonic metric on the target.  After determinant
normalization the transported and selected metrics agree; without a
matching normalization they differ by an $X$-constant positive scalar
and still yield the same relative flat connection.  Thus the assignment
respects arrows as well as objects.  The same pullback calculation gives
compatibility with restriction and descent on the $C^d$ site.  For
$d=\infty$ the construction and its uniqueness characterization are
identical to those of \cref{thm:smooth-stable-transform}, so the two
natural transformations agree.
\end{proof}

\subsection{Finite-regularity filtrations and the semistable locus}

The stable theorem answers only part of Question~5.3.4.  For a semistable family the relevant object is not a selected harmonic metric on the total bundle but a relative filtration whose graded pieces are harmonic families.  The filtration criterion of \cref{sec:filtrations} is insensitive to replacing smooth parameter dependence by $C^d$ dependence because its proof uses only invariant subbundles, quotients, finite extensions, and descent.

\begin{definition}[$C^d$ relative harmonic filtration]\label{def:Cd-harmonic-filtration}
Let $(E,\cD'')$ be a $C^d_UC^\infty_X$ Higgs family.  A $C^d$ relative harmonic filtration is a finite filtration
\[
  0=F_0\subset F_1\subset\cdots\subset F_\ell=E
\]
by $C^d_UC^\infty_X$ subbundles invariant under $\cD''$ such that every quotient $F_i/F_{i-1}$ admits a harmonic metric of class $C^d_UC^\infty_X$.
\end{definition}

Let $\mathscr M_{\Dol,d}^{\cH}(X)$ denote the stackification, on the site of $C^d$ manifolds, of the prestack generated from $C^d_UC^\infty_X$ harmonic Higgs families by finite iterated extensions.  This is the finite-regularity extension-generated stack proposed in Question~5.3.4.  We use this definition only for the following local geometric characterization; no new geometricity statement for the ambient $C^d$ moduli stacks is needed.

\begin{theorem}[Finite-regularity filtration criterion]\label{thm:Cd-filtration-criterion}
A $C^d_UC^\infty_X$ Higgs family $(E,\cD'')$ defines an object of $\mathscr M_{\Dol,d}^{\cH}(X)$ if and only if, locally on $U$, it admits a $C^d$ relative harmonic filtration.  The analogous statement holds for $C^d_UC^\infty_X$ flat families.
\end{theorem}

\begin{proof}
Suppose first that the family is locally a finite iterated extension of $C^d_UC^\infty_X$ harmonic families.  Choose a local presentation by successive short exact sequences
\[
  0\to E_{i-1}\to E_i\to Q_i\to0,
\]
with each $Q_i$ harmonic.  In the category of mixed-regularity vector bundles, the injections identify $E_{i-1}$ with $C^d_UC^\infty_X$ subbundles of $E_i$; local complements can be chosen by taking orthogonal complements with respect to a background mixed-regularity Hermitian metric.  The connection or Higgs compatibility in the short exact sequence makes each subbundle invariant.  Pulling the filtration through the iterated extensions gives a $C^d$ relative harmonic filtration of the total family.

Conversely, let
\[
  0=F_0\subset F_1\subset\cdots\subset F_\ell=E
\]
be a $C^d$ relative harmonic filtration.  For every $i$ there is a short exact sequence of mixed-regularity Higgs families
\[
  0\longrightarrow F_{i-1}
  \longrightarrow F_i
  \longrightarrow F_i/F_{i-1}
  \longrightarrow0.
\]
Starting from $F_1\cong F_1/F_0$, which is harmonic, these sequences express $F_i$ inductively as a finite extension of the harmonic quotients.  Hence the total family belongs locally to the prestack generated under finite extensions by harmonic families.  Stackification and descent give membership in $\mathscr M_{\Dol,d}^{\cH}(X)$.

The proof on the de Rham side is identical after replacing Higgs-invariant subbundles by flat-invariant subbundles.  No step differentiates in the parameter direction, so there is no loss of $C^d$ regularity.
\end{proof}

\begin{corollary}[A semistable finite-regularity criterion]\label{cor:Cd-semistable-criterion}
Let $(E,\cD'')$ be a semistable $C^d_UC^\infty_X$ Higgs family.  Suppose that locally on $U$ it admits a filtration by $C^d_UC^\infty_X$ invariant subbundles whose graded quotients are stable and satisfy the NAH numerical conditions.  Then
\[
  (E,\cD'')\in\mathscr M_{\Dol,d}^{\cH}(X)
\]
locally on $U$.
\end{corollary}

\begin{proof}
On a neighbourhood carrying the stated invariant filtration, every
graded quotient is a mixed $C^d_UC^\infty_X$ stable family: quotient
charts and quotient Higgs coefficients inherit the mixed regularity
from the subbundle charts.  Apply \cref{thm:Cd-stable-metrics} to the
finitely many quotients, shrinking the neighbourhood once so that all
their normalized harmonic metrics are defined simultaneously.  Each
metric has the same mixed regularity.  The original invariant
filtration is consequently a $C^d$ relative harmonic filtration in the
sense of \cref{def:Cd-harmonic-filtration}, and
\cref{thm:Cd-filtration-criterion} gives the conclusion.
\end{proof}

This corollary is a genuine partial answer to the semistable part of Question~5.3.4.  It leaves intact the essential assembly problem: a fibrewise Jordan--H\"older filtration need not vary as a $C^d$ relative filtration.  The obstruction classes of \cref{sec:obstructions} remain relevant, now in finite regularity.  In particular, the first obstruction is still represented by the lower-left component of the parameter deformation in the appropriate Hom complex.  What changes is only the regularity class of the parameter map, not the fibrewise elliptic complex on $X$.

\subsection{Reduced singular parameter spaces and parameter-singular metrics}

Finite regularity is one meaning of singular parameter dependence.  A different issue arises when the parameter space itself is singular.  Here the ambient formulation \eqref{eq:universal-moment-map} has a concrete advantage: the solution operator is constructed before imposing the Maurer--Cartan equation.  Thus it can be pulled back along maps from singular subsets without requiring the integrable locus to carry a manifold structure.

We isolate a class for which this argument is immediate and intrinsic.

\begin{definition}[Ambiently smooth reduced parameter space]\label{def:ambient-reduced-space}
A reduced parameter space $S$ is called \emph{locally ambiently smooth} if every $s_0\in S$ has a neighbourhood presented by a subset $S_0\subset V$ of a smooth manifold, equipped with the subset diffeology and the reduced sheaf of smooth functions induced from $V$.  A map
\[
  p:S\longrightarrow\cP
\]
to a Banach or Fr\'echet coefficient space is \emph{ambiently smooth} if, locally after such a presentation, it is the restriction of a smooth map
\[
  \widetilde p:V\longrightarrow\cP.
\]
\end{definition}

The definition covers, for example, reduced zero sets in Euclidean space and many singular differential subspaces.  The ambient extension is not required to satisfy the Higgs integrability equation away from $S$.

\begin{theorem}[Stable harmonic metrics over reduced singular parameter spaces]\label{thm:ambient-singular-stable}
Let $S$ be a locally ambiently smooth reduced parameter space and let $(E,\cD'')$ be a stable Higgs family on $S\times X$ satisfying the NAH numerical conditions fibrewise.  Suppose that, in local trivializations, the full coefficient datum
\[
  p_s=(\cD''_s,q_s)
\]
consisting of the Higgs operator and a determinant normalization is ambiently smooth in the sense of \cref{def:ambient-reduced-space}.  Then, locally on $S$, there is a unique normalized harmonic metric
\[
  s\longmapsto h_s
\]
which is ambiently smooth as a map into $C^\infty$ Hermitian metrics on $X$.

The conclusion is independent of the chosen ambient manifold and of the chosen ambient extension of the coefficient datum.
\end{theorem}

\begin{proof}
Fix $s_0\in S$ and choose an ambient presentation $S_0\subset V$ together with a smooth extension
\[
  \widetilde p:V\to\cP_k
\]
of the local coefficient datum.  At the base stable harmonic point, the vertical derivative of the full moment-map map \eqref{eq:Cd-full-moment-map} is the invertible Jacobi operator.  Hence the implicit-function theorem gives the smooth coefficient-space solution operator \eqref{eq:Cd-solution-operator} on an open neighbourhood $\cO_k$ of $p_{s_0}$.  After shrinking $V$,
\[
  \widetilde p(V)\subset\cO_k,
\]
and therefore
\[
  \widetilde s(v)=\Sigma_k(\widetilde p(v))
\]
is a smooth Sobolev family on $V$.  Restricting to $S_0$ gives a solution of the trace-free moment-map equation for the original family.

The ambient extension need not be integrable at points of $V\setminus S_0$.  This causes no difficulty because the nonlinear moment-map map and its implicit solution operator are defined on the full coefficient space.  Integrability and the NAH numerical conditions are used only after restriction to $S_0$, where they promote the normalized moment-map solution to a harmonic metric.

Starting from this one Sobolev solution family on $V$, the parameter-dependent elliptic bootstrap of \cref{prop:parametric-bootstrap,prop:joint-regularity} gives joint smoothness in $(v,x)$ after shrinking $V$ once.  Thus restriction to $S_0$ is ambiently smooth as a map into $C^\infty$ metrics; no intersection of infinitely many Sobolev neighbourhoods is involved.  If two ambient presentations or two extensions are used, their restrictions to $S_0$ produce normalized harmonic metrics on the same stable fibres with the same determinant.  Fibrewise uniqueness, \cref{prop:normalized-uniqueness}, forces equality.  Thus the restricted metric is independent of all auxiliary ambient choices.
\end{proof}

\begin{corollary}[Singular reduced bases with an ambient extension]\label{cor:singular-subset-stable}
Let $S\subset\R^N$ be a reduced singular subset with the induced smooth structure.  If a stable Higgs family on $S\times X$ is the restriction of smooth ambient coefficient data near every point of $S$, including an ambiently smooth determinant normalization, and satisfies the NAH numerical conditions on $S$, then it admits a locally ambiently smooth normalized harmonic metric.
\end{corollary}

The same mechanism gives a semistable statement when the extension data themselves vary ambiently.

\begin{corollary}[Filtered semistable families over reduced singular bases]\label{cor:singular-filtered-semistable}
Let $S$ be a locally ambiently smooth reduced parameter space and let $(E,\cD'')$ be a semistable Higgs family on $S\times X$.  Suppose that locally on $S$ there is a filtration by invariant subbundles
\[
  0=F_0\subset F_1\subset\cdots\subset F_\ell=E
\]
whose coefficient data and subbundle projections are ambiently smooth, and whose graded quotients are stable, satisfy the NAH numerical conditions, and carry ambiently smooth determinant normalizations.  Then the graded quotients admit locally ambiently smooth harmonic metrics.  Consequently the family is, locally on $S$, a finite iterated extension of ambiently smooth harmonic families.
\end{corollary}

\begin{proof}
Apply \cref{thm:ambient-singular-stable} to each of the finitely many stable quotient families, shrinking the ambient presentation once so that all harmonic metrics are defined simultaneously.  The invariant filtration already supplies the successive short exact sequences
\[
  0\to F_{i-1}\to F_i\to F_i/F_{i-1}\to0.
\]
The assumed ambiently smooth subbundle projections ensure that every
quotient bundle and its induced Higgs coefficients extend smoothly to
the same ambient presentation.  Hence the quotient metrics supplied by
\cref{thm:ambient-singular-stable} and the short exact sequences all
live in the ambiently smooth category.  Iterating the finitely many
sequences therefore gives the claimed extension there, not merely a
fibrewise collection of extensions.  The conclusion is deliberately
local: the theorem does not assert that an arbitrary fibrewise
Jordan--H\"older filtration assembles ambiently over a singular base.
\end{proof}

The theorem should not be confused with a result for arbitrary nonreduced $C^\infty$-schemes.  In a nilpotent thickening one must decide what a positive Hermitian metric means in the nilpotent directions and how the nonlinear adjoint operation is represented sheaf-theoretically.  The ordinary Banach implicit-function theorem does suggest a formal lifting mechanism once an appropriate functor of points is fixed, but we do not claim such a theorem here.

\begin{problem}[Nonreduced $C^\infty$ parameter schemes]\label{prob:nonreduced-Cinfty}
Develop a notion of harmonic metric for families over nonreduced $C^\infty$-schemes, determine whether the stable coefficient-space solution operator extends functorially to nilpotent thickenings, and compare the resulting deformation theory with the ordinary Higgs deformation complex.
\end{problem}

There is still another meaning of singularity, closer to the boundary phenomenon implicit in Question~5.3.4.  Suppose a family is stable away from a limiting parameter but has a semistable nonpolystable central fibre.  Since an ordinary positive harmonic metric would force that central fibre to be polystable, degeneration of the metric family is not merely possible; in the natural mixed topology it is unavoidable.

\begin{proposition}[Degeneration at a nonpolystable limit]\label{prop:nonpoly-limit-degeneration}
Let $U$ be a $C^0$ manifold with a marked point $0$, and let $(E,\cD'')$ be a $C^0_UC^\infty_X$ Higgs family satisfying the NAH numerical conditions fibrewise.  Assume that $(E_u,\cD''_u)$ is polystable for $u\neq0$, while the central fibre $(E_0,\cD''_0)$ is semistable but not polystable.  Then no choice of harmonic metrics $h_u$ for $u\neq0$ can extend to a positive-definite family
\[
  u\longmapsto h_u
\]
that is continuous at $0$ with values in the Fr\'echet space of smooth Hermitian metrics on $X$.
\end{proposition}

\begin{proof}
Suppose such an extension existed and denote its value at the central parameter by $h_0$.  The mixed continuity of the Higgs coefficients and of $h_u$ implies convergence in every $C^j(X)$ norm.  The Chern connection, the adjoint Higgs field, and the Hitchin--Simpson moment map depend continuously on these data once two $X$-derivatives of the metric are controlled.  Passing to the limit in the harmonic equation therefore gives
\[
  \mu(\cD''_0,h_0)=0.
\]
Together with the numerical conditions, the usual Hitchin--Simpson correspondence implies that the central Higgs bundle is polystable.  This contradicts the hypothesis.  Hence the harmonic metrics cannot converge in the $C^\infty$ metric topology to a positive-definite limit.  They may diverge, lose positive definiteness, oscillate, or otherwise fail to extend continuously.
\end{proof}

Thus a singular metric in the parameter direction is a natural boundary object for a degeneration toward a nonpolystable semistable fibre.  The proposition does not construct such a boundary metric.  It shows that an ordinary continuous positive harmonic reduction cannot survive the degeneration.  On the extension-generated side, the limiting information is instead retained by harmonic metrics on the graded pieces together with the extension class.

Even when the limiting fibre is polystable, a smooth or finite-regularity Higgs family may have harmonic metrics on the stable locus whose parameter derivatives become unbounded as a stabilizer develops.  The first-variation formula makes a precise spectral criterion available.

\begin{proposition}[Spectral criterion for loss of parameter regularity]\label{prop:spectral-parameter-singularity}
Let $u\mapsto(E_u,\cD''_u)$ be a one-parameter family which is stable for $u\neq0$, and let $h_u$ be normalized harmonic metrics there.  Write
\[
  f_u=\cS_{h_u}(\eta_u),
  \qquad
  s_u=-L_{h_u}^{-1}f_u
\]
for the trace-free first metric variation.  Suppose $e_u$ is an $L^2$-normalized eigenvector of $L_{h_u}$ with eigenvalue $\lambda_u>0$.  Then
\begin{equation}\label{eq:spectral-lower-bound}
  \abs{\ip{s_u}{e_u}_{L^2}}
  =\frac{\abs{\ip{f_u}{e_u}_{L^2}}}{\lambda_u},
  \qquad
  \norm{s_u}_{L^2}
  \ge
  \frac{\abs{\ip{f_u}{e_u}_{L^2}}}{\lambda_u}.
\end{equation}
In particular, if $\lambda_u\to0$ and
\[
  \abs{\ip{f_u}{e_u}_{L^2}}
\]
fails to be $O(\lambda_u)$, then the first derivative of the normalized harmonic metric cannot remain bounded in $L^2$ along that sequence.  If the numerator is bounded below by a positive constant, the derivative grows at least like $\lambda_u^{-1}$.
\end{proposition}

\begin{proof}
Since $L_{h_u}$ is self-adjoint on the normalized Hermitian slice and
\[
  L_{h_u}s_u=-f_u,
\]
taking the $L^2$ inner product with $e_u$ gives
\[
  \lambda_u\ip{s_u}{e_u}_{L^2}
  =-\ip{f_u}{e_u}_{L^2}.
\]
Self-adjointness moves $L_{h_u}$ from $s_u$ to $e_u$, and the eigenvalue
equation gives the left side of the last display.  Dividing by
$\lambda_u>0$ yields the equality in
\eqref{eq:spectral-lower-bound}.  Cauchy--Schwarz and
$\|e_u\|_{L^2}=1$ give
\[
  |\ip{s_u}{e_u}_{L^2}|\le\|s_u\|_{L^2},
\]
which proves the lower bound.  If the numerator is not
$O(\lambda_u)$, the quotient is unbounded along a subsequence; if it is
bounded below by $c>0$, the quotient is at least
$c\lambda_u^{-1}$.  These are exactly the final two assertions.
\end{proof}

\begin{remark}[Three parameter regimes behind Question~5.3.4]\label{rem:three-parameter-regimes}
The results of this section separate three phenomena that should not be conflated.
\begin{enumerate}[label=(\roman*)]
\item A family may have only finite parameter regularity.  On the stable locus, \cref{thm:Cd-stable-metrics} shows that the normalized harmonic metric has the same $C^d$ regularity.
\item A family may approach a semistable nonpolystable fibre.  Then Proposition~\ref{prop:nonpoly-limit-degeneration} shows that a positive harmonic metric family cannot extend continuously in the $C^0_UC^\infty_X$ topology; genuine degeneration is necessary.
\item A family may approach a polystable stabilizer jump.  Even when the Higgs coefficients are smooth away from the limiting parameter, \cref{prop:spectral-parameter-singularity} shows that the harmonic metric can lose parameter regularity because the inverse Jacobi operator becomes singular in a direction detected by the source.
\end{enumerate}
The third mechanism does not by itself prove existence of a singular harmonic metric across the limiting parameter.  It identifies the precise compatibility condition required for bounded first variation: the source must vanish against the emerging small-eigenvalue directions at a compensating rate.  Taken together, the three regimes explain why finite regularity, nonpolystable degeneration, and stabilizer jumps require different analytic formulations.
\end{remark}

The stabilizer-jump question has a sharp negative example beyond this spectral criterion.  The real-analytic polystable family in \cref{prop:square-root-polystable-failure} has no continuous harmonic metric on any neighbourhood of its central parameter and no continuous relative harmonic filtration there.  It is therefore outside $\mathscr M_{\Dol,d}^{\cH}(X)$ for every $d=0,1,2,\ldots,\infty$.  By contrast, \cref{prop:coalescing-lines-filtered} has no continuous harmonic metric on the total family but does have a smooth relative harmonic filtration.  Thus failure of metric selection and failure of extension-generated membership are genuinely different phenomena.

\section{Heat flow, extension data, and the onset of non-geometricity}\label{sec:heat-flow-geometricity}

The preceding finite-regularity results show that singular behaviour in parameter directions is not an accident of notation.  We now record a complementary interpretation in terms of the Yang--Mills--Higgs heat flow.  This section is deliberately formal: it does not construct a new flow.  Its purpose is to isolate what any such correction would have to remember if it were to extend non-Abelian Hodge theory across the strictly semistable locus without passing to an extension completion.

On stable and polystable objects, the heat-flow picture is geometric.  The complex gauge orbit is closed modulo unitary gauge, a zero of the Hitchin--Simpson moment map exists, and the associated harmonic connection gives the flat object.  On a strictly semistable object, the complex gauge orbit is not closed.  The natural closed orbit in its closure is represented by a polystable associated graded object.  This is the analytic counterpart of S-equivalence, and it is the point at which an ordinary heat-flow limit ceases to see extension data.

\subsection{Closed orbits, polystable shadows, and what the heat flow forgets}

We first put this observation in a form compatible with the stack language used throughout the paper.  Fix a smooth complex bundle $E_0\to X$, and let $\cA^{\ssm}$ denote the space of semistable Higgs operators on $E_0$ satisfying the NAH numerical conditions.  The complex gauge group $\Ggauge^{\C}$ acts on $\cA^{\ssm}$.  For $A\in\cA^{\ssm}$ write $(E_A,\theta_A)$ for the corresponding Higgs bundle.

\begin{definition}[Polystable shadow]\label{def:polystable-shadow}
A \emph{polystable shadow} of a semistable Higgs bundle $(E,\theta)$ is the isomorphism class
\[
  \operatorname{sh}(E,\theta):=\operatorname{gr}_{\mathrm{JH}}(E,\theta)
\]
of a polystable associated graded object for a Jordan--H\"older filtration.  Equivalently, after choosing a representative in an analytic gauge-theoretic model, it is the closed complex-gauge orbit contained in the closure of the orbit of $(E,\theta)$.
\end{definition}

The definition is independent of the chosen Jordan--H\"older filtration up to isomorphism.  It is therefore a geometric object on the coarse semistable quotient, but it is not an object of the original extension problem: the passage
\[
  (E,\theta)\longmapsto \operatorname{gr}_{\mathrm{JH}}(E,\theta)
\]
discards the extension classes that distinguish non-isomorphic semistable objects in the same S-equivalence class.

In the case of Higgs bundles over a compact Riemann surface, Wilkin's convergence theorem for the Yang--Mills--Higgs flow identifies the limiting critical point, up to isomorphism, with the graded object associated to the Harder--Narasimhan--Seshadri filtration \cite{Wilkin2008}.  For a semistable Higgs bundle of slope zero, this is precisely the polystable Jordan--H\"older shadow.  We use this theorem as analytic motivation rather than as an input into the proofs above; the structural point below is independent of the detailed convergence theory.

\begin{proposition}[Ordinary harmonic limits force polystability]\label{prop:ordinary-heat-polystable}
Let $(E,\theta)$ be a semistable Higgs bundle satisfying the NAH numerical conditions.  Suppose there is a positive-definite Hermitian metric $h$ on $E$ such that
\[
  \mu(\cD'',h)=0.
\]
Then $(E,\theta)$ is polystable.  Consequently, no heat-flow correction whose limiting object is an ordinary positive solution of the Hitchin--Simpson equation on the original Higgs bundle can include a strictly semistable nonpolystable object in its domain.
\end{proposition}

\begin{proof}
The first statement is the polystability direction of the Hitchin--Simpson correspondence, or equivalently the Chern--Weil equality case in the Hermitian--Yang--Mills--Higgs theorem.  Under the numerical assumptions, a solution of the slope-zero Hitchin--Simpson equation gives a harmonic metric, and the resulting harmonic bundle decomposes according to its parallel endomorphism algebra into stable harmonic summands of the same slope.  Thus the Higgs bundle is polystable.

For the consequence, assume that a corrected heat-flow procedure assigned to $(E,\theta)$ a path of positive metrics whose limit $h_\infty$ remained a positive metric on $E$ and solved the same equation.  The first part would imply that $(E,\theta)$ is polystable.  This contradicts strict semistability when the object is not polystable.  The obstruction is therefore not the absence of a better analytic estimate; it is the incompatibility between a positive harmonic limit on the original object and nonclosed complex-gauge orbit geometry.
\end{proof}

This proposition should be read together with Proposition~\ref{prop:nonpoly-limit-degeneration}.  There the parameter family approaches a nonpolystable central fibre; here a single strictly semistable object is considered.  In both formulations, ordinary positive harmonic metrics cannot be the missing boundary objects.

The next statement formalizes the loss of extension data.  Let $\mathfrak H$ denote the groupoid of polystable harmonic bundles, and let $\mathfrak H_{\mathrm{coarse}}$ be its set of isomorphism classes.  A heat-flow construction whose output is only the limiting harmonic bundle determines a map on isomorphism classes
\[
  \Theta:\{\text{semistable Higgs bundles with NAH numerical conditions}\}/\cong
  \longrightarrow \mathfrak H_{\mathrm{coarse}}
\]
whenever the limiting isomorphism class exists.

\begin{proposition}[Factorization through the polystable shadow]\label{prop:heat-flow-factors-shadow}
Assume that $\Theta(E,\theta)$ is represented by the harmonic metric on the polystable closed orbit in the complex-gauge orbit closure of $(E,\theta)$.  Then $\Theta$ factors as
\[
\begin{tikzcd}
\{\text{semistable Higgs bundles}\}/\cong
  \arrow[r,"\operatorname{sh}"]
  \arrow[dr,"\Theta"']
&
\{\text{polystable Higgs bundles}\}/\cong
  \arrow[d,"\operatorname{NAH}_{\poly}"]
\\
&
\mathfrak H_{\mathrm{coarse}} .
\end{tikzcd}
\]
In particular, if two semistable Higgs bundles have the same Jordan--H\"older graded object but different extension classes, the construction cannot distinguish them.
\end{proposition}

\begin{proof}
By hypothesis the value of $\Theta$ is obtained by first replacing
$(E,\theta)$ by the closed orbit in its complex-gauge orbit closure.
The Jordan--H\"older theorem identifies that closed orbit, up to
isomorphism, with $\operatorname{gr}_{\mathrm{JH}}(E,\theta)$ and also
shows that its isomorphism class is independent of the chosen
Jordan--H\"older filtration.  The ordinary polystable correspondence
then assigns to this graded object its harmonic-bundle isomorphism
class.  Thus
\[
 \Theta(E,\theta)
 =\operatorname{NAH}_{\poly}
   \bigl(\operatorname{sh}(E,\theta)\bigr),
\]
which is the displayed factorization.  Two semistable objects are
S-equivalent precisely when these graded objects are isomorphic.
Consequently extension classes within one S-equivalence class have the
same image, proving the final assertion.
\end{proof}

Thus the heat flow, when observed only through its unrenormalized limiting harmonic object, behaves like the coarse semistable quotient.  It detects the closed orbit and not the path by which a nonclosed orbit approaches it.  This explains why the elliptic example in \cref{sec:examples} can be simultaneously harmless for the extension-generated stack and discontinuous for the coarse correspondence.

\subsection{Filtered and renormalized targets}

The preceding obstruction does not rule out a more refined analytic theory.  It says only that the target cannot be the ordinary groupoid of positive harmonic metrics on the original semistable Higgs bundle.  To preserve extension data, one must enlarge what is remembered at the limit.

Let
\[
  0=E_0\subset E_1\subset\cdots\subset E_\ell=E
\]
be a relative harmonic filtration, and set $Q_i=E_i/E_{i-1}$.  Choose local splittings so that
\begin{equation}\label{eq:upper-triangular-Ddoubleprime-heat}
  \cD''_E=
  \begin{pmatrix}
  \cD''_{Q_1}&\beta_{12}&\beta_{13}&\cdots&\beta_{1\ell}\\
  0&\cD''_{Q_2}&\beta_{23}&\cdots&\beta_{2\ell}\\
  \vdots&\ddots&\ddots&\ddots&\vdots\\
  0&\cdots&0&\cD''_{Q_{\ell-1}}&\beta_{\ell-1,\ell}\\
  0&\cdots&\cdots&0&\cD''_{Q_\ell}
  \end{pmatrix}.
\end{equation}
The diagonal blocks carry harmonic metrics $h_i$.  The off-diagonal
matrix $\boldsymbol\beta=(\beta_{ij})$ represents the relative
extension datum modulo filtered changes of splitting and satisfies the
Maurer--Cartan system \eqref{eq:upper-MC}.  The extension-generated
object is therefore not just the collection of harmonic quotients; it
is the full upper-triangular object formed from those quotients and this
compatible Maurer--Cartan datum.

\begin{definition}[Filtered harmonic shadow]\label{def:filtered-harmonic-shadow}
A \emph{filtered harmonic shadow} of a semistable Higgs family with a chosen relative harmonic filtration is the data
\[
  \operatorname{FHar}(E,\cD'',\mathcal F)
  =\left(\{(Q_i,\cD''_{Q_i},h_i)\}_{i=1}^\ell,
  [\boldsymbol\beta]_{\mathrm{fil}}\right),
\]
where the diagonal terms are harmonic quotient families and
$\boldsymbol\beta=(\beta_{ij})_{i<j}$ is the strictly upper-triangular
degree-one datum satisfying the Maurer--Cartan system
\eqref{eq:upper-MC}.  The bracket $[\boldsymbol\beta]_{\mathrm{fil}}$
means equivalence under smooth filtered changes of splitting.  Only the
adjacent entries are ordinary extension cocycles in isolation; for
$j-i>1$, their differentials contain the quadratic lower-extension
terms in \eqref{eq:upper-MC}.  Isomorphisms are upper-triangular
isomorphisms preserving the induced associated-graded data and this
filtered Maurer--Cartan class.
\end{definition}

This definition is simply the geometric face of the extension completion.  It keeps exactly the information that the polystable shadow forgets.  The ordinary heat-flow limit corresponds to dropping the second component:
\[
  \operatorname{FHar}(E,\cD'',\mathcal F)
  \longmapsto
  \{(Q_i,\cD''_{Q_i},h_i)\}_{i=1}^\ell.
\]
The resulting collection of diagonal harmonic bundles is the polystable shadow.  The extension-generated stack retains the off-diagonal part.

A possible analytic correction of the heat flow would therefore have to produce, not a positive harmonic metric on $E$, but an asymptotic filtered object.  Formally, if $R_t$ is a block-diagonal rescaling
\[
  R_t=\bigoplus_i e^{\alpha_i t}\Id_{Q_i},
\]
then the off-diagonal block from $Q_j$ to $Q_i$ in $R_t^{-1}\cD''_E R_t$ is scaled by
\[
  e^{(\alpha_j-\alpha_i)t}\beta_{ij}.
\]
The unrenormalized limit may kill this term, whereas a suitable choice of weights can keep track of its first nonzero asymptotic order.  This is the analytic analogue of remembering extension classes after passing to the associated graded.

\begin{definition}[Renormalized heat-flow limit, formal version]\label{def:renormalized-heat-limit}
A \emph{renormalized heat-flow limit} for a semistable Higgs bundle is a tuple
\[
  \left(\mathcal F,\{h_i\}_{i=1}^\ell,\{\alpha_i\}_{i=1}^\ell,
  [\boldsymbol\beta^{\mathrm{ren}}]_{\mathrm{fil}}\right)
\]
consisting of a filtration, harmonic metrics on the graded quotients,
weights $\alpha_i$, and retained renormalized limits of the
off-diagonal Higgs-extension components in a weighted block gauge,
subject to the corresponding filtered Maurer--Cartan compatibility.
Two such tuples are identified when they differ by filtered gauge
transformations and the induced change of Maurer--Cartan representative.
\end{definition}

We do not assert here that the Yang--Mills--Higgs flow canonically produces such a tuple.  The definition records what must be added if a heat-flow picture is to recover the same semistable family information as the extension-generated stack.  The weights should be thought of as measuring rates of collapse toward the closed orbit; the cocycles measure the transverse directions that are invisible after taking the ordinary limit.

\begin{proposition}[Extension completion as a receptacle for filtered limits]\label{prop:extension-completion-receptacle}
Any heat-flow-based assignment which, on strictly semistable objects,
remembers the limiting positive harmonic metrics on stable graded
factors together with a compatible filtered Maurer--Cartan extension
datum determines an object of the finite extension completion of the
harmonic-bundle category.  Conversely, every object produced from a
relative harmonic filtration gives such filtered harmonic data.
\end{proposition}

\begin{proof}
The first direction is the construction of the finite extension
completion in upper-triangular form.  The harmonic metrics on the stable
graded factors give objects of the harmonic-bundle category.  The
adjacent components of $\boldsymbol\beta$ specify the successive short
exact sequences, while the remaining components and the
Maurer--Cartan equations encode the compatibility of those extensions
with the chosen total differential.  Iterating the short exact
sequences produces an object of the extension completion.

Conversely, an object of the extension completion is, locally by
\cref{thm:filtration-criterion}, represented by a finite relative
harmonic filtration.  Choosing local splittings writes it in
upper-triangular form as in
\eqref{eq:upper-triangular-Ddoubleprime-heat}.  Integrability of the
total Higgs operator is exactly the Maurer--Cartan equation for its
strictly upper-triangular part.  Changing the splitting acts by a
filtered gauge transformation, so it changes the representative
$\boldsymbol\beta$ but not its filtered class or the resulting extension
object.
\end{proof}

The proposition explains why extension completion should not be regarded as a purely categorical substitute for a missing analytic construction.  It is the natural target for any analytic construction that tries to remember what the unrenormalized heat flow loses.

\subsection{Where geometricity fails and what remains possible}

The preceding discussion gives a precise sense in which non-Abelian Hodge theory begins to become non-geometric at the semistable boundary.  On individual stable and polystable objects, the correspondence is represented by honest harmonic metrics.  At a strictly semistable object, the same equation on the same underlying Higgs bundle would force a splitting.  The analytic flow therefore sees the polystable closed orbit.  The stacky extension construction records the nonclosed orbit data that the flow contracts away.

For parameter families there is an intermediate possibility.  A slicewise harmonic reduction may become singular while the adjoint Higgs operator and the resulting flat connection remain continuous because the singularity is compensated by vanishing of the Higgs coefficients.  The square-root family of Proposition~\ref{prop:square-root-polystable-failure} exhibits this phenomenon, and \cref{subsec:weak-C0-mediator} packages it in a weak $C^0$ operator-level mediator.  This does not restore a regular metric or retain non-split semistable extension data, but it shows that failure of metric regularity is not identical to failure of every relative operator-level correspondence.

This yields the following informal but mathematically constrained principle.

\begin{quote}
The semistable boundary is not non-geometric because harmonic metrics are technically hard to construct there.  It is non-geometric because ordinary harmonic metrics geometrize closed complex-gauge orbits, whereas strictly semistable objects live in nonclosed orbits whose extension data is transverse to the closed-orbit limit.
\end{quote}

There are therefore three possible responses.
\begin{enumerate}[label=(\roman*)]
\item One may pass to S-equivalence or coarse moduli.  Then the heat flow has a geometric target, but extension data is deliberately discarded.
\item One may work with the extension-generated stack.  Then semistable extension data is retained, but the target is no longer the ordinary groupoid of harmonic metrics on single polystable objects.
\item One may try to build a filtered or renormalized heat flow.  Such a flow would have to converge to data of the form in Definition~\ref{def:renormalized-heat-limit}, not merely to a positive harmonic metric on the original semistable Higgs bundle.
\end{enumerate}

The third option is a genuine analytic problem.  It would amount to giving a heat-flow interpretation of the extension completion rather than bypassing it.

\begin{problem}[Filtered heat flow and extension data]\label{prob:filtered-heat-flow}
Develop a Yang--Mills--Higgs heat-flow theory in which a semistable initial Higgs bundle produces a filtered asymptotic object consisting of:
\begin{enumerate}[label=(\alph*)]
\item the polystable harmonic metrics on the Jordan--H\"older graded factors;
\item asymptotic weights measuring collapse rates toward the closed orbit;
\item a renormalized filtered Maurer--Cartan datum representing the
successive extensions and their compatibility;
\item functoriality under pullback in parameter families.
\end{enumerate}
Determine whether the resulting filtered asymptotic object agrees with the image of the same semistable family in the extension-generated diffeological stack.
\end{problem}

This problem is deliberately phrased as a comparison, not as an attempt to avoid the stack.  A successful answer would show that the extension completion has an analytic heat-flow incarnation.  A negative answer would be equally informative: it would identify extension data that is genuinely stack-theoretic and not recoverable from asymptotic metric analysis alone.

\section{The \texorpdfstring{$\lambda$}{lambda}-direction and a diffeological Hodge family}\label{sec:lambda}

\subsection{\texorpdfstring{$\lambda$}{lambda}-connections and stable harmonic families}

Let $(E,\cD'',h)$ be a smooth family of harmonic bundles over $U\times X$, with
\[
  \cD''=\ddbar_E+\theta,
  \qquad
  \cD'_h=\partial_{E,h}+\theta^{\dagger_h}.
\]
The harmonic-bundle identities \eqref{eq:harmonic-bicomplex} give the standard Hodge family.  For $\lambda\in\C$ define a varying holomorphic structure and a holomorphic $\lambda$-connection by
\begin{equation}\label{eq:lambda-connection}
  \ddbar_{E,\lambda}
  :=\ddbar_E+\lambda\theta^{\dagger_h},
  \qquad
  D_{E,\lambda}
  :=\lambda\partial_{E,h}+\theta.
\end{equation}
The relevant Leibniz rule is
\[
  D_{E,\lambda}(fs)
  =\lambda(\partial f)\otimes s+fD_{E,\lambda}s
\]
for smooth scalar functions $f$; compatibility with $\ddbar_{E,\lambda}$ then makes $D_{E,\lambda}$ a holomorphic $\lambda$-connection on the holomorphic bundle defined by $\ddbar_{E,\lambda}$.

Package the pair into the total degree-one operator
\begin{equation}\label{eq:lambda-total-differential}
  \mathbb D_\lambda
  :=\ddbar_{E,\lambda}+D_{E,\lambda}
  =\cD''+\lambda\cD'_h.
\end{equation}
At $\lambda=0$, the pair is $(\ddbar_E,\theta)$, namely the original Higgs bundle.  For $\lambda\ne0$, the associated ordinary flat connection is
\begin{equation}\label{eq:lambda-flat-rescaled}
  \nabla_\lambda
  :=\ddbar_{E,\lambda}+\lambda^{-1}D_{E,\lambda}
  =\ddbar_E+\partial_{E,h}
   +\lambda^{-1}\theta+\lambda\theta^{\dagger_h}.
\end{equation}
At $\lambda=1$ this is the non-Abelian Hodge flat connection $\nabla_h$.

\begin{proposition}[Integrability of the Hodge family]\label{prop:lambda-flatness}
For every $\lambda\in\C$,
\begin{align}
  \ddbar_{E,\lambda}^2&=0,\label{eq:lambda-dbar-square}\\
  D_{E,\lambda}^2&=0,\label{eq:lambda-D-square}\\
  \ddbar_{E,\lambda}D_{E,\lambda}
  +D_{E,\lambda}\ddbar_{E,\lambda}&=0.\label{eq:lambda-compatibility}
\end{align}
Equivalently,
\begin{equation}\label{eq:lambda-square}
  \mathbb D_\lambda^2=0.
\end{equation}
For $\lambda\ne0$, \eqref{eq:lambda-flat-rescaled} is flat.  If the harmonic family is smooth in $u$, all coefficients are jointly smooth in $(u,\lambda,x)$ and polynomial in $\lambda$ before the rescaling by $\lambda^{-1}$.
\end{proposition}

\begin{proof}
Expanding the square of the total operator gives
\[
 \mathbb D_\lambda^2
 = (\cD'')^2
 +\lambda[\cD'',\cD'_h]_{\mathrm{gr}}
 +\lambda^2(\cD'_h)^2.
\]
Each coefficient vanishes by the harmonic-bundle identities
\[
  (\cD'')^2=0,
  \qquad
  [\cD'',\cD'_h]_{\mathrm{gr}}=0,
  \qquad
  (\cD'_h)^2=0.
\]
so \eqref{eq:lambda-square} holds for every complex $\lambda$.
Relative to the complex structure on $X$, $\ddbar_{E,\lambda}$ has type
$(0,1)$ and $D_{E,\lambda}$ has type $(1,0)$.  The $(0,2)$, $(2,0)$,
and $(1,1)$ components of $\mathbb D_\lambda^2=0$ are respectively
\eqref{eq:lambda-dbar-square}, \eqref{eq:lambda-D-square}, and
\eqref{eq:lambda-compatibility}.

For $\lambda\ne0$, divide the $(1,0)$ operator by $\lambda$.  The three
identities just proved are exactly the type components of the curvature
of $\ddbar_{E,\lambda}+\lambda^{-1}D_{E,\lambda}$, so that connection is
flat.  Finally, \eqref{eq:lambda-connection} is assembled from jointly
smooth family coefficients by addition and scalar multiplication; it
is polynomial in $\lambda$.  This proves the regularity assertion.
\end{proof}

For a stable plot, the smoothly varying normalized harmonic metric makes the preceding $\lambda$-connection construction smooth in both the original parameter and $\lambda$.  We spell out the resulting family before passing to the stack morphism.

Combine \cref{thm:intro-smooth-metrics} and \cref{prop:lambda-flatness}.

\begin{theorem}[Smooth Hodge family on the stable locus]\label{thm:stable-Hodge-family}
Let $(E,\cD'')$ be a smooth family of stable Higgs bundles satisfying the NAH numerical conditions.  There is a smooth family of integrable holomorphic $\lambda$-connections
\[
  (u,\lambda)\longmapsto
  (\ddbar_{E,u,\lambda},D_{E,u,\lambda})
\]
over $U\times\C$ such that:
\begin{enumerate}[label=(\roman*)]
\item at $\lambda=0$ it is the original Higgs family;
\item at $\lambda=1$ the associated ordinary connection is the flat family produced by the non-Abelian Hodge transform;
\item the construction commutes with smooth pullback in $U$;
\item the resulting Hodge pair is independent of the residual positive-scalar ambiguity in the harmonic metric.
\end{enumerate}
\end{theorem}

\begin{proof}
Choose a global smooth normalized harmonic metric $h$ by \cref{thm:intro-smooth-metrics}.  Define the Hodge pair by \eqref{eq:lambda-connection}; explicitly, the $(0,1)$ and $(1,0)$ pieces are formed from the original Higgs data, the Chern operator of $h$, and the adjoint Higgs field.  Every ingredient depends smoothly on $u$, and the dependence on $\lambda$ is polynomial.  Hence the resulting family is smooth on $U\times\C$.

Integrability for every $\lambda$ is \cref{prop:lambda-flatness}.  At $\lambda=0$, the terms involving the harmonic adjoint and Chern part drop out in exactly the way prescribed by \eqref{eq:lambda-connection}, leaving the original Higgs family.  At $\lambda=1$, the two pieces combine to the ordinary flat connection
\[
  \nabla_h=\cD''+\cD'_h,
\]
which is the stable transform of \eqref{eq:smooth-stack-map}.

For a smooth map of parameter manifolds $f:V\to U$, \cref{prop:pullback-solutions} identifies the harmonic metric of the pulled-back family with $f_X^*h$ after the determinant data are pulled back.  The Chern operator, Higgs adjoint, and formula \eqref{eq:lambda-connection} therefore commute with pullback.  Finally, replacing $h$ by $e^{c(u)}h$ leaves the relative Chern connection and $\theta^{\dagger_h}$ unchanged by \cref{prop:scalar-independence}.  Thus the Hodge pair is independent of the residual positive-scalar ambiguity.
\end{proof}

\subsection{The stable Hodge morphism and its first variation}

Let $\MdiffHod(X)\to\C$ denote the diffeological prestack of smooth families of integrable holomorphic $\lambda$-connections, where an object over a test manifold $S$ includes a smooth function $\lambda:S\to\C$.  The abstract $\lambda$-$d$ formalism in \cite{AzamRayan2026} gives a uniform definition.  We do not need a global geometricity theorem here; the point is the explicit family supplied by harmonic metrics.

The base parameter must be displayed explicitly.  A Higgs plot over $U$ produces a Hodge family over $U\times\C$, not an object over $U$ with an unrecorded extra variable.  Thus the natural statement is the following map over the $\lambda$-line.

\begin{corollary}[Stable Hodge lift]\label{cor:stable-Hodge-lift}
The construction of \cref{thm:stable-Hodge-family} defines a smooth morphism over $\C$
\[
  \MdiffDol^{\st,0}(X)\times\C
  \longrightarrow
  \MdiffHod(X).
\]
Its pullback along $\{0\}\hookrightarrow\C$ is the original Dolbeault family, while along $\{1\}\hookrightarrow\C$ the associated ordinary connection is the stable transform \eqref{eq:smooth-stack-map}.
\end{corollary}

\begin{proof}
Represent a plot of $\MdiffDol^{\st,0}(X)$ by a stable Higgs family over a parameter manifold $U$.  The global harmonic metric of \cref{thm:intro-smooth-metrics} gives a smooth Hodge object over $U\times\C$.  If another normalization is chosen, the two harmonic metrics differ by a positive scalar function of the parameter, constant along $X$.  By \cref{prop:scalar-independence}, the relative Chern operator and the adjoint Higgs field are identical for the two choices.  Thus the global Hodge object is independent of the auxiliary normalization.

The pullback compatibility in \cref{thm:stable-Hodge-family} shows that this assignment is natural in plots and respects descent; thus it defines a smooth morphism of diffeological stacks over $\C$.  Evaluating at $\lambda=0$ recovers the original Higgs family, while at $\lambda=1$ the associated ordinary connection is $\nabla_h$, the stable transform \eqref{eq:smooth-stack-map}.
\end{proof}

The Hodge family also has an explicit infinitesimal description.  Differentiating in a plot direction combines the metric correction computed earlier with the elementary variation of the $\lambda$-dependent pair.

Let $\eta=a+\varphi$ be the derivative of an actual smooth stable Higgs plot in a chosen local smooth gauge, and let $s_{\mathrm{tot}}$ be the corresponding logarithmic metric variation.  Differentiating the total operator \eqref{eq:lambda-total-differential} with respect to the family parameter gives
\begin{equation}\label{eq:lambda-family-variation}
  \dot{\mathbb D}_\lambda
  =\eta
  +\lambda\bigl(
    \eta^{\star_h}+\cD'_h s_{\mathrm{tot}}
  \bigr).
\end{equation}
In a fixed-determinant direction, substituting $s_0=-G_h\cS_h(\eta)$ yields
\begin{equation}\label{eq:lambda-differential}
  \dot{\mathbb D}_\lambda
  =\eta+
  \lambda\left(
  \eta^{\star_h}
  -\cD'_hG_h\cS_h(\eta)
  \right).
\end{equation}
At $\lambda=0$, this is the original Higgs deformation.  At $\lambda=1$, it is the infinitesimal NAH operator \eqref{eq:inf-NAH-operator}.

\begin{proposition}[Interpolation of infinitesimal deformations]\label{prop:lambda-inf}
In the chosen gauge, formula \eqref{eq:lambda-differential} gives a polynomial interpolation between the Dolbeault representative at $\lambda=0$ and the de Rham representative at $\lambda=1$.  For each $\lambda$, it satisfies the linearized integrability equation of the Hodge family.
\end{proposition}

\begin{proof}
Let $t\mapsto\mathbb D_{t,\lambda}$ be the Hodge family associated to an actual one-parameter stable plot.  For every $(t,\lambda)$ the integrability equation is
\[
  \mathbb D_{t,\lambda}^2=0.
\]
Differentiate with respect to $t$ at $t=0$.  If $\dot{\mathbb D}_{\lambda}$ denotes the expression in \eqref{eq:lambda-differential}, the Leibniz rule gives
\[
  [\mathbb D_{0,\lambda},\dot{\mathbb D}_{\lambda}]_{\mathrm{gr}}=0,
\]
which is precisely the linearized integrability equation for a deformation of the $\lambda$-connection.

Formula \eqref{eq:lambda-differential} is polynomial in $\lambda$ because the underlying Hodge family is polynomial in $\lambda$ and the metric derivative does not introduce additional $\lambda$-dependence.  Evaluating at $\lambda=0$ leaves the Dolbeault deformation representative, whereas at $\lambda=1$ the terms combine to the differentiated flat connection described in \cref{thm:differential-plot}.  Thus the formula interpolates between the two representatives while remaining tangent to the integrable Hodge locus for every $\lambda$.
\end{proof}

\subsection{Twistor normalization and extensions beyond the stable locus}

The usual twistor family is often written with a parameter $\zeta\in\C^\times$ as
\begin{equation}\label{eq:twistor-family}
  \nabla_\zeta
  =\zeta^{-1}\theta+D_h+\zeta\theta^{\dagger_h}.
\end{equation}
The Hodge pair \eqref{eq:lambda-connection} and the twistor family \eqref{eq:twistor-family} differ by the standard rescaling and choice of affine chart on $\mathbb P^1$.  For our purposes the Hodge form is preferable because it extends directly to $\lambda=0$ as a $\lambda$-connection.

\begin{remark}[Scope]
We do not claim here a new derived twistor structure.  Our result is a smooth family theorem over arbitrary finite-dimensional parameter plots.  Derived and shifted twistor constructions address a different level of geometry; see \cite{KryczkaTanakaYau2026}.  The relationship between those constructions and the diffeological family stack merits separate study.
\end{remark}

Beyond the stable locus the construction separates into two regimes.  Constant-type polystable families inherit the centralizer ambiguity already described, whereas semistable extension data requires a compatible $\lambda$-dependent lifting of the off-diagonal cocycles.

On a locally split constant-type polystable family, choose a smooth harmonic metric as in \cref{thm:constant-type} and apply \eqref{eq:lambda-connection}.  The result depends on the multiplicity-space metric choices but is naturally organized by the harmonic metric groupoid.

For a semistable family with a relative harmonic filtration, each harmonic quotient block has the $\lambda$-family constructed above.  It is tempting to transport the off-diagonal extension data simultaneously in $\lambda$.  In an upper-triangular splitting, the desired object would have the form
\begin{equation}\label{eq:lambda-upper-triangular}
  \cD_{\lambda,E}
  =
  \begin{pmatrix}
    \cD_{\lambda,Q_1} & \beta_{12}(\lambda) & \cdots\\
    0 & \cD_{\lambda,Q_2} & \cdots\\
    \vdots & \ddots & \ddots
  \end{pmatrix},
\end{equation}
where the off-diagonal terms solve the $\lambda$-dependent
Maurer--Cartan equations and recover the original filtered extension
datum at $\lambda=0$.  The extension-generated equivalence of
\cite{AzamRayan2026} identifies the endpoint categories, but by itself
it does not supply a canonical smooth choice of the functions
$\beta_{ij}(\lambda)$ for an arbitrary family.  We therefore regard
\eqref{eq:lambda-upper-triangular} as a target construction, not as a
theorem proved here.  For split families, or whenever a compatible
smooth $\lambda$-dependent lift of the full filtered Maurer--Cartan
datum is given, it does define a filtered Hodge family.  Establishing
such lifts in general is part of the global enhancement problem below.

This distinction leaves a natural global question: whether the extension-generated diffeological correspondence admits a canonical Hodge enhancement without imposing extra choices on semistable extension data.

\begin{question}[Diffeological Deligne--Hitchin enhancement]\label{q:global-Hodge}
Can the equivalence
\[
  \MHDol(X)\simeq\MHdR(X)
\]
be promoted to a global diffeological Hodge or Deligne--Hitchin stack whose local objects are smooth filtered $\lambda$-families and whose transition at $\lambda\ne0$ is governed by the diffeological Riemann--Hilbert correspondence?
\end{question}

The stable theorem supplies a concrete global family model.  The semistable difficulty is again the smooth assembly of filtrations and extension classes, not the formal existence of a $\lambda$-parameter.

\section{Examples and model calculations}\label{sec:examples}

\subsection{Abelian and stable model calculations}

For a line bundle, the nonlinear commutator term disappears and the metric equation becomes linear.  Let $(L_u,\ddbar_u)$ be a smooth degree-zero family.  Relative to a background metric $q^0_u$, write the normalized Hermitian--Einstein metric as
\[
  q_u=q^0_ue^{f_u},
  \qquad
  \int_X f_u\,\frac{\omega^n}{n!}=0.
\]
Then $f_u$ is determined by the scalar Poisson equation \eqref{eq:det-poisson}.  The linearized operator is the scalar Laplacian
\[
  \Delta_\omega:W^{k+2,2}_0(X)\longrightarrow W^{k,2}_0(X)
\]
on mean-zero functions.  If $\sigma_{\det}(\eta)$ denotes the derivative, at fixed metric, of the scalar curvature-contraction term on the right side of \eqref{eq:det-poisson}, then
\[
  \dot f=-\Delta_\omega^{-1}\sigma_{\det}(\eta)
\]
up to the fixed convention factor in \eqref{eq:det-poisson}.

This is the abelian analogue of the Green-operator formula.  It should not be confused with the trace-free source $\cS_h$: in rank one the trace-free endomorphism bundle is zero, so $\cS_h$ vanishes identically and all metric variation lies in the determinant equation.

\begin{example}[A family of degree-zero line bundles on a curve]\label{ex:line-curve}
Let $X$ be a compact Riemann surface and let
\[
  \alpha_u\in A^{0,1}(X)
\]
be a smooth family of $\ddbar$-closed forms.  Define
\[
  \ddbar_u=\ddbar+\alpha_u.
\]
Take a fixed background metric.  The Chern curvature is affine in $\partial\alpha_u$ and its conjugate.  The normalized harmonic metric correction $f_u$ is obtained by applying the fixed Green operator of the scalar Laplacian to the mean-zero curvature contraction.  Hence the dependence on $u$ is as smooth, or as real analytic, as the family $\alpha_u$.
\end{example}

The rank-one calculation has no Jacobi correction beyond the scalar normalization.  A stable rank-two family is the first example in which the Green-operator term is genuinely visible.

Let $X$ be a compact Riemann surface and $(E,\ddbar_E,\theta)$ a stable rank-two Higgs bundle of degree zero with fixed determinant.  Let an actual smooth one-parameter family through this point have trace-free first derivative
\[
  \eta=(a,\varphi),
  \qquad
  \cD''\eta=0.
\]
The normalized harmonic metric variation along this plot is
\begin{equation}\label{eq:rank2-s}
  s=-G_h\sqrt{-1}\Lambda\Bigl(
  [\eta,\cD'_h]_{\mathrm{gr}}
  +[\cD'',\eta^{\star_h}]_{\mathrm{gr}}
  \Bigr).
\end{equation}
The flat deformation is
\begin{equation}\label{eq:rank2-flat}
  \dot\nabla
  =\eta+\eta^{\star_h}+\cD'_hs.
\end{equation}

Even when $\eta$ is represented by a simple algebraic variation of the Higgs field, $s$ is nonlocal because it involves the Green operator.  This nonlocality is the analytic content hidden by the pointwise statement that a harmonic metric exists uniquely.

Two elementary specializations separate the sources of the first variation.  We begin by varying the Higgs field while holding the holomorphic structure fixed.

Suppose the holomorphic structure is fixed and
\[
  \theta_t=\theta+t\varphi+O(t^2),
  \qquad
  \ddbar_E\varphi+[\theta,\varphi]=0.
\]
Then
\[
  \eta=\varphi.
\]
The source simplifies to
\begin{equation}\label{eq:higgs-only-source}
  \cS_h(\varphi)
  =\sqrt{-1}\Lambda\Bigl(
  [\varphi,\cD'_h]_{\mathrm{gr}}
  +[\cD'',\varphi^{\dagger_h}]_{\mathrm{gr}}
  \Bigr)_0.
\end{equation}
In the standard type-expanded form, the zero-order commutators
\[
  [\varphi,\theta^{\dagger_h}]
  +[\theta,\varphi^{\dagger_h}]
\]
are the characteristic terms.  If $\theta=0$, the trace-free moment-map source vanishes to first order: the derivative terms in the invariant graded formula have pure types $(2,0)$ and $(0,2)$ and are killed by $\Lambda$, while the commutator $[\theta_t,\theta_t^{\dagger_h}]$ starts at quadratic order in $t$.  Thus the first nontrivial metric response to a Higgs-field variation through the zero section is generally second order.

We then reverse the roles and vary only the Dolbeault operator.  The comparison makes the sign in the companion deformation particularly transparent.

Suppose $\theta$ is fixed in a smooth trivialization and
\[
  \ddbar_{E,t}=\ddbar_E+ta+O(t^2).
\]
Then
\[
  \eta=a.
\]
The $(1,1)$ Chern-curvature contribution is
\[
  \sqrt{-1}\Lambda\bigl(
  \partial_{E,h}a-\ddbar_Ea^{\dagger_h}
  \bigr),
\]
with the remaining terms supplied by the interaction with the fixed Higgs field in the invariant formula \eqref{eq:source-operator}.  Again the metric correction is obtained by solving an elliptic equation.

\subsection{Polystable model calculations}

Let $F_u$ and $G_u$ be smooth stable degree-zero Higgs families, pairwise non-isomorphic for all $u$ in a parameter neighbourhood.  Set
\[
  E_u=F_u\oplus G_u.
\]
Choose normalized smooth harmonic metrics $h_F(u)$ and $h_G(u)$.  Then
\[
  h_u=h_F(u)\oplus h_G(u)
\]
is smooth and harmonic.  The kernel of the Jacobi operator on Hermitian endomorphisms is
\[
  \R\Id_F\oplus\R\Id_G.
\]
After fixing the total determinant there remains a one-dimensional trace-free kernel generated by
\[
  \frac1{\rk F}\Id_F
  \oplus
  -\frac1{\rk G}\Id_G.
\]
This is the infinitesimal freedom to rescale the two factor metrics inversely while preserving the total determinant.

The example shows why determinant normalization alone does not give uniqueness on the polystable locus.

A multiplicity-two stable factor exhibits the centralizer phenomenon in its simplest nontrivial form.  The cone of harmonic metric choices is already positive-dimensional even after the stable factor itself is fixed.

Let $(F,\cD''_F)$ be stable and consider
\[
  E=F\otimes\C^2.
\]
Choose a normalized harmonic metric $h_F$.  Every harmonic metric is
\[
  h_F\otimes q,
  \qquad
  q\in\operatorname{Herm}^+(\C^2).
\]
Even after fixing the determinant of $E$, the allowed $q$ form the symmetric space
\[
  \operatorname{SL}(2,\C)/SU(2).
\]
The trace-free Jacobi kernel has real dimension three, matching the tangent space of this symmetric space.  This is the simplest explicit model of the harmonic metric fibre over a polystable point.

\subsection{Explicit failures across a polystable stabilizer jump}\label{sec:explicit-polystable-failures}

The loss of a uniform Jacobi inverse at a stabilizer jump can reflect an actual failure of continuous harmonic metrics, rather than only a limitation of the proof.  The following two elementary families also distinguish that analytic failure from membership in the extension-generated stack.

Let $X$ be an elliptic curve and fix a nowhere-vanishing holomorphic one-form
\[
  \alpha\in H^0(X,\Omega_X^1).
\]
In both examples the underlying holomorphic bundle is
\[
  E=\mathcal O_X^{\oplus2}
\]
with its fixed trivial Dolbeault operator.  A matrix depending only on the parameter and multiplied by $\alpha$ therefore defines an integrable Higgs field.  All fibres have trivial Chern classes and satisfy the NAH numerical conditions.

\begin{proposition}[No continuous harmonic metric, but a harmonic filtration]\label{prop:coalescing-lines-filtered}
On a real interval $U$ about $0$, consider
\begin{equation}\label{eq:triangular-polystable-family}
  \theta_t
  =
  \begin{pmatrix}
    0&t\\
    0&t^2
  \end{pmatrix}\alpha.
\end{equation}
Every fibre $(E,\ddbar+\theta_t)$ is polystable, but no choice of fibrewise harmonic metrics extends continuously through $t=0$.  Nevertheless the family admits a smooth relative harmonic filtration and hence belongs to $\mathscr M_{\Dol,d}^{\cH}(X)$ for every $d\in\{0,1,2,\ldots,\infty\}$.
\end{proposition}

\begin{proof}
For $t\ne0$, the two eigensummands are
\[
  L_{0,t}=\mathcal O_X e_1,
  \qquad
  L_{t^2,t}=\mathcal O_X(e_1+t e_2),
\]
with Higgs fields $0$ and $t^2\alpha$, respectively.  They are non-isomorphic stable line Higgs bundles, so their direct sum is polystable.  At $t=0$ the Higgs field vanishes and the fibre is the polystable bundle $\mathcal O_X\oplus\mathcal O_X$.

Suppose that $h_t$ were a continuous family of harmonic metrics through $0$.  For $t\ne0$, \cref{thm:constant-type} applied to the fibrewise decomposition shows that the two non-isomorphic stable eigensummands are orthogonal for every harmonic metric.  Thus, at any fixed $x\in X$,
\[
  h_t(e_1,e_1+t e_2)=0.
\]
Letting $t\to0$ gives $h_0(e_1,e_1)=0$, contradicting positive definiteness.

On the other hand, the constant line subbundle $F=\mathcal O_Xe_1$ is invariant for every $t$.  Its Higgs field is zero, and the induced Higgs field on $E/F$ is $t^2\alpha$.  These are smooth families of harmonic line Higgs bundles.  Hence
\[
  0\subset F\subset E
\]
is a smooth relative harmonic filtration.  It has every finite parameter regularity, so \cref{thm:Cd-filtration-criterion} gives the last assertion.
\end{proof}

The preceding example shows why failure of a metric on the total polystable family does not by itself imply failure of the extension-generated correspondence.  The next family has the stronger property asked for in the family-level questions of \cite{AzamRayan2026}.

\begin{proposition}[A polystable family outside every extension-generated locus]\label{prop:square-root-polystable-failure}
Let $U\subset\C$ be a disk about $0$, and set
\begin{equation}\label{eq:square-root-polystable-family}
  \theta_z
  =
  \begin{pmatrix}
    0&z\\
    z^2&0
  \end{pmatrix}\alpha.
\end{equation}
This is a real-analytic family of polystable Higgs bundles.  On no neighbourhood of $0$ does it admit a continuous family of harmonic metrics.  It also admits no continuous relative harmonic filtration on such a neighbourhood.  Consequently
\[
  (E,\ddbar+\theta)\notin\mathscr M_{\Dol,d}^{\cH}(X)
\]
locally at $0$ for every $d\in\{0,1,2,\ldots,\infty\}$.
\end{proposition}

\begin{proof}
Write
\[
  A(z)=
  \begin{pmatrix}
    0&z\\
    z^2&0
  \end{pmatrix}.
\]
Then $A(z)^2=z^3\Id$.  For $z\ne0$ it has two distinct eigenlines.  In the affine chart of $\mathbf P^1$ they are generated by vectors $(1,a)$ with
\begin{equation}\label{eq:eigenline-square-root}
  a^2=z.
\end{equation}
Thus the corresponding rank-one Higgs summands are stable and non-isomorphic, while at $z=0$ the Higgs field is zero.  Every fibre is therefore polystable.

Along the positive real ray $z=r>0$, the eigenlines are generated by
\[
  (1,\sqrt r),
  \qquad
  (1,-\sqrt r).
\]
If harmonic metrics $h_z$ extended continuously to $z=0$, these two non-isomorphic stable summands would be $h_r$-orthogonal for $r>0$ by \cref{thm:constant-type}.  Both generators tend to $e_1$, so continuity would again give $h_0(e_1,e_1)=0$, a contradiction.

It remains to rule out a relative harmonic filtration.  A proper filtration of this rank-two family would contain a continuous invariant line subbundle $F$.  Fix $x\in X$.  Because $\alpha(x)\ne0$, the line $F|_{(z,x)}$ must be an eigenline of $A(z)$ for every $z\ne0$.  Its first coordinate is then nonzero, so it determines a continuous function $a$ on the punctured disk satisfying \eqref{eq:eigenline-square-root}.  Such a function would be a continuous square root of $z$ on a punctured disk, which is impossible: on a circle the identity $a^2=z$ would give $2\deg(a)=1$.  Hence no continuous invariant line subbundle exists on any neighbourhood of $0$.

A one-step filtration would instead assert that the whole family itself carries a continuous harmonic metric, which has already been excluded.  These are the only possibilities in rank two.  The finite-regularity filtration criterion \cref{thm:Cd-filtration-criterion} now shows that the family is outside $\mathscr M_{\Dol,d}^{\cH}(X)$ for every $d$, since every $C^d$ filtration or metric would in particular be continuous.
\end{proof}

\subsection{A weak \texorpdfstring{$C^0$}{C0} operator-level harmonic mediator}\label{subsec:weak-C0-mediator}

Proposition~\ref{prop:square-root-polystable-failure} rules out a continuous family of harmonic reductions and even a continuous relative harmonic filtration.  It does not rule out continuity of the operators obtained after the singular reduction has been used slicewise.  In the square-root family, the singularity of the metric is compensated by vanishing of the off-diagonal Higgs coefficients.  The resulting adjoint Higgs field, and hence the associated flat connection, remain continuous.

\begin{proposition}[A continuous flat lift with singular harmonic reduction]\label{prop:weak-flat-lift-square-root}
For the family \eqref{eq:square-root-polystable-family}, set $r=|z|$ and, for $z\ne0$, define
\begin{equation}\label{eq:singular-harmonic-metric-square-root}
  h_z=
  \begin{pmatrix}
    r^{1/2}&0\\
    0&r^{-1/2}
  \end{pmatrix},
  \qquad h_0=\Id.
\end{equation}
Then every $h_z$ is a smooth harmonic metric on the slice over $z$, with $\det h_z=1$, but $z\mapsto h_z$ is not continuous at $0$.  Nevertheless, the $h_z$-adjoint matrices extend continuously as
\begin{equation}\label{eq:continuous-adjoint-square-root}
  A(z)^{\dagger_{h_z}}
  =C(z)
  :=
  \begin{cases}
  \begin{pmatrix}
    0&\displaystyle\frac{\bar z^2}{|z|}\\[5pt]
    |z|\bar z&0
  \end{pmatrix},&z\ne0,\\[14pt]
  0,&z=0.
  \end{cases}
\end{equation}
Consequently
\begin{equation}\label{eq:continuous-flat-lift-square-root}
  \nabla_z=d+A(z)\alpha+C(z)\bar\alpha
\end{equation}
is a $C^0_UC^\infty_X$ family of flat connections with slicewise semisimple monodromy.  Its pointwise coarse isomorphism class is the classical non-Abelian Hodge image of $(E,\ddbar+\theta_z)$.

The original Higgs family therefore defines a holomorphic map from $U$ to the coarse Dolbeault moduli space, and \eqref{eq:continuous-flat-lift-square-root} is a continuous lift of its image under the classical coarse non-Abelian Hodge homeomorphism.

The resulting map to the coarse Betti moduli space is not $C^3$ at $z=0$.  In particular, no $C^3_UC^\infty_X$, and hence no smooth or real-analytic, family of flat connections can represent this same coarse map near $0$.
\end{proposition}

\begin{proof}
For $z\ne0$, choose $a$ with $a^2=z$.  The eigenlines of $A(z)$ are generated by
\[
  v_+=(1,a),
  \qquad
  v_-=(1,-a).
\]
Since $|a|^2=|z|=r$, formula \eqref{eq:singular-harmonic-metric-square-root} gives
\[
  h_z(v_+,v_-)=r^{1/2}-|a|^2r^{-1/2}=0.
\]
Thus the two stable rank-one Higgs summands are orthogonal.  The metric is constant in the $X$-direction, so its Chern connection is the trivial connection.  Direct calculation gives
\[
  h_z^{-1}A(z)^\dagger h_z
  =
  \begin{pmatrix}
    0&\bar z^2/r\\
    r\bar z&0
  \end{pmatrix}.
\]
This is \eqref{eq:continuous-adjoint-square-root}.  Its upper-right and lower-left entries have norms $r$ and $r^2$, respectively, and therefore extend continuously by zero.  The metric itself does not extend: its two eigenvalues tend to $0$ and $+\infty$.

For $z\ne0$, choose $\lambda$ with $\lambda^2=z^3$.  Then
\[
  C(z)=\frac{\bar\lambda}{\lambda}A(z),
\]
an expression independent of the sign of $\lambda$.  Hence $A(z)$ and $C(z)$ commute.  Their coefficients are constant on $X$, while $\alpha$ and $\bar\alpha$ are closed, so the curvature of \eqref{eq:continuous-flat-lift-square-root} vanishes.  On the two eigenlines the connection is
\[
  d+\lambda\alpha+\bar\lambda\bar\alpha,
  \qquad
  d-\lambda\alpha-\bar\lambda\bar\alpha.
\]
These are the rank-one non-Abelian Hodge transforms of $(\mathcal O_X,\lambda\alpha)$ and $(\mathcal O_X,-\lambda\alpha)$.  At $z=0$ the connection is $d\oplus d$.  Thus every fibre is semisimple and the coarse class is the classical non-Abelian Hodge image.  On the Dolbeault side the unordered pair is single-valued because changing $\lambda$ to $-\lambda$ interchanges its members; equivalently, its characteristic polynomial is
\[
  \eta^2-z^3\alpha^2.
\]

It remains to identify the intrinsic regularity loss.  Let $\gamma$ be a loop in $X$ with nonzero period
\[
  p_\gamma=\int_\gamma\alpha.
\]
Up to the immaterial sign convention for holonomy, the trace of the monodromy of \eqref{eq:continuous-flat-lift-square-root} around $\gamma$ is
\[
  T_\gamma(z)
  =2\cosh\bigl(\lambda p_\gamma+\bar\lambda\bar p_\gamma\bigr).
\]
Since $\lambda^2=z^3$ and $|\lambda|^2=|z|^3$, expansion at the origin gives
\begin{equation}\label{eq:coarse-trace-regularity}
  T_\gamma(z)
  =2+p_\gamma^2z^3+\bar p_\gamma^2\bar z^3
   +2|p_\gamma|^2|z|^3+O(|z|^6).
\end{equation}
The function $|z|^3$ is $C^2$ but not $C^3$ at the origin.  Holonomy traces of a $C^3_UC^\infty_X$ family of flat connections are $C^3$ in the parameter, by parameter dependence for the parallel-transport equation.  Formula \eqref{eq:coarse-trace-regularity} therefore rules out any $C^3$ representative of the same coarse Betti map.
\end{proof}

\begin{remark}[Slicewise heat flow and the nonuniform limit]\label{rem:weak-mediator-heat-flow}
One may reach the same operator-level conclusion by solving the harmonic-metric heat equation separately on the fibres.  Starting from a regular background family, the finite-time solutions retain parameter regularity, while standard convergence on each polystable slice produces a harmonic limit.  In the square-root family every such limit makes the two distinct rank-one summands orthogonal.  Changing the positive weights on those summands changes the limiting metric but not the adjoint matrix $C(z)$, and hence not the flat connection \eqref{eq:continuous-flat-lift-square-root}.  The loss of metric regularity is therefore a failure of uniform long-time convergence in the parameter, whereas the endpoint relative operator exhibits compensated continuity.  No particular heat-flow selection is built into the weak mediator below.
\end{remark}

The proposition suggests separating regularity of the relative operators from regularity of a harmonic reduction.  The mixed $C^0_UC^\infty_X$ category of \cref{sec:finite-regularity} is the natural setting.  Write $\mathscr M_{\Dol,0}(X)$ and $\mathscr M_{\dR,0}(X)$ for the stackifications, on the open-cover site of $C^0$ manifolds, of the evident prestacks of mixed $C^0_UC^\infty_X$ Higgs and flat families.  We retain $\mathscr M_{\Dol,0}^{\cH}(X)$ for the metric-regular extension-generated stack introduced in \cref{sec:finite-regularity} and write $\mathscr M_{\dR,0}^{\cH}(X)$ for its flat analogue.

\begin{definition}[Weak $C^0$ harmonic mediator]\label{def:weak-C0-mediator}
For a $C^0$ parameter manifold $U$, an object of
\[
  \mathscr H^{\mathrm{wk}}_0(X)(U)
\]
is a triple $(E,\cD'',\cD')$ with the following properties:
\begin{enumerate}[label=(\roman*)]
\item $E\to U\times X$ is a mixed $C^0_UC^\infty_X$ complex vector bundle;
\item $\cD''=\ddbar_E+\theta$ and $\cD'$ are relative degree-one operators with mixed $C^0_UC^\infty_X$ coefficients satisfying
\[
  (\cD'')^2=(\cD')^2=0,
  \qquad
  \cD''\cD'+\cD'\cD''=0;
\]
\item for every $u\in U$ there exists a smooth Hermitian metric $h_u$ on $E_u$ such that
\[
  \cD'_u=\partial_{E_u,h_u}+\theta_u^{\dagger_{h_u}}.
\]
\end{enumerate}
No parameter regularity is required of the assignment $u\mapsto h_u$.  Morphisms are mixed $C^0_UC^\infty_X$ bundle isomorphisms intertwining both $\cD''$ and $\cD'$.  The metrics $h_u$ are witnesses of pointwise metrizability, not additional rigid structure.
\end{definition}

Equivalently, one may display an object as $(E,\cD'',\nabla,\{h_u\})$, where
\[
  \nabla=\cD''+\cD'
\]
is continuous and flat, while the slicewise metrics may be discontinuous.  Treating the metrics only as witnesses removes artificial multiplicity coming from arbitrary discontinuous scalar rescalings, all of which determine the same relative operators.

\begin{proposition}[Pullback and endpoint functoriality]\label{prop:weak-mediator-functoriality}
The assignment $U\mapsto\mathscr H^{\mathrm{wk}}_0(X)(U)$ is a prestack on the open-cover site of $C^0$ manifolds.  It has natural endpoint functors
\begin{equation}\label{eq:weak-mediator-endpoints}
  \mathscr H^{\mathrm{wk}}_0(X)
  \longrightarrow \mathscr M_{\Dol,0}(X),
  \qquad
  \mathscr H^{\mathrm{wk}}_0(X)
  \longrightarrow \mathscr M_{\dR,0}(X),
\end{equation}
given by
\[
  (E,\cD'',\cD')\longmapsto(E,\cD''),
  \qquad
  (E,\cD'',\cD')\longmapsto(E,\cD''+\cD').
\]
The metric-regular $C^0$ harmonic mediator is the full subprestack for which the witnesses $h_u$ can be chosen as a mixed $C^0_UC^\infty_X$ metric.  Its inclusion in $\mathscr H^{\mathrm{wk}}_0(X)$ is strict.
\end{proposition}

\begin{proof}
For a $C^0$ map $f:V\to U$, pull back the bundle and both relative operators.  A pointwise harmonic witness pulls back slicewise by
\[
  (f^*h)_v=h_{f(v)}.
\]
No continuity of this assignment is needed.  The three operator identities in Definition~\ref{def:weak-C0-mediator} are preserved by pullback, and the Hom presheaves are sheaves because intertwining both relative operators is a local condition.  This proves the prestack assertion.  The total operator $\cD''+\cD'$ is flat by (ii), giving the second endpoint in \eqref{eq:weak-mediator-endpoints}; the first is immediate.

Every metric-regular harmonic family determines such a triple, and the inclusion is full with the stated choice of morphisms.  Proposition~\ref{prop:weak-flat-lift-square-root} supplies an object of the weak mediator whose Dolbeault endpoint is the family of Proposition~\ref{prop:square-root-polystable-failure}.  That proposition rules out any continuous harmonic witness, proving strictness.
\end{proof}

The earlier categorical construction can now be applied at the operator level.  Write
\[
  \cD'=\partial_E^{\mathrm{wk}}+\psi
\]
according to type.  The pointwise witnesses identify $\partial_E^{\mathrm{wk}}$ and $\psi$ with $\partial_{E,h_u}$ and $\theta_u^{\dagger_{h_u}}$ on each slice, but the two coefficients are continuous by definition even when the witnesses are not.  Therefore
\begin{equation}\label{eq:weak-Hodge-family}
  \ddbar_{E,\lambda}=\ddbar_E+\lambda\psi,
  \qquad
  D_{E,\lambda}=\lambda\partial_E^{\mathrm{wk}}+\theta
\end{equation}
is a mixed $C^0$ Hodge family.  This is the operator-level reason that the categorical construction survives the singular reduction.

\begin{proposition}[Weak extension-completed correspondence]\label{prop:weak-extension-correspondence}
Let
\[
  \mathscr M_{\Dol,0}^{\mathrm{wk}\cH}(X),
  \qquad
  \mathscr M_{\dR,0}^{\mathrm{wk}\cH}(X)
\]
be the stacky Dolbeault and de Rham endpoint images obtained from $\mathscr H^{\mathrm{wk}}_0(X)$ by finite iterated extension completion and stackification.  The common operator-level mediator induces an equivalence
\begin{equation}\label{eq:weak-extension-equivalence}
  \mathscr M_{\Dol,0}^{\mathrm{wk}\cH}(X)
  \simeq
  \mathscr M_{\dR,0}^{\mathrm{wk}\cH}(X).
\end{equation}
The original metric-regular extension-generated stacks embed into these weak endpoint stacks.
\end{proposition}

\begin{proof}
The finite-extension and stackification argument of \cite{AzamRayan2026} uses relative operators, Hom complexes, finite upper-triangular extension matrices, Maurer--Cartan equations, pullback, and descent.  These operations differentiate only in the $X$-direction and preserve mixed $C^0_UC^\infty_X$ coefficients.  Formula \eqref{eq:weak-Hodge-family} supplies the same $\lambda$-$d$ interpolation used in the smooth mediator.  Thus the common finite extension completion has Dolbeault and de Rham endpoint realizations, and evaluation at the two endpoints gives \eqref{eq:weak-extension-equivalence}.  The inclusion of the metric-regular mediator commutes with both endpoint functors, finite extensions, and stackification, producing the final embeddings.
\end{proof}

\begin{remark}[What the formal equivalence asserts]\label{rem:weak-correspondence-scope}
Proposition~\ref{prop:weak-extension-correspondence} concerns the two stacky images defined from the common weak mediator.  It does not identify either image with the full ambient $C^0$ Higgs or flat stack, and it does not turn a map into a coarse moduli space into a canonical family-level lift.  A continuous flat endpoint is additional relative data and may fail to exist or fail to be unique on other parameter spaces.
\end{remark}

\begin{corollary}[Strict enlargement at $C^0$ regularity]\label{cor:strict-weak-enlargement}
There are strict inclusions
\[
  \mathscr M_{\Dol,0}^{\cH}(X)
  \subsetneq
  \mathscr M_{\Dol,0}^{\mathrm{wk}\cH}(X),
  \qquad
  \mathscr M_{\dR,0}^{\cH}(X)
  \subsetneq
  \mathscr M_{\dR,0}^{\mathrm{wk}\cH}(X).
\]
The Dolbeault witness for strictness is the family of Proposition~\ref{prop:square-root-polystable-failure}; its de Rham witness is the flat family \eqref{eq:continuous-flat-lift-square-root}.
\end{corollary}

\begin{proof}
The Dolbeault assertion follows immediately from Propositions~\ref{prop:square-root-polystable-failure} and \ref{prop:weak-flat-lift-square-root}.  On the de Rham side, for $z\ne0$ the two rank-one flat summands of \eqref{eq:continuous-flat-lift-square-root} have the same eigenlines generated by $(1,\pm\sqrt z)$.  A continuous invariant line subbundle would therefore choose a continuous square root of $z$ on the punctured disk, which is impossible.  A one-step metric-regular presentation would instead provide a continuous harmonic metric; orthogonality of the two non-isomorphic rank-one flat summands along the positive real ray gives the same coalescing-lines contradiction as in Proposition~\ref{prop:square-root-polystable-failure}.  The flat analogue of Theorem~\ref{thm:Cd-filtration-criterion} excludes the family from $\mathscr M_{\dR,0}^{\cH}(X)$, while Definition~\ref{def:weak-C0-mediator} includes it before extension completion.
\end{proof}

The weak mediator enlarges the polystable--semisimple family correspondence, but it does not make extension completion unnecessary at a strictly semistable nonpolystable fibre.  Such a fibre has no harmonic metric even slicewise on the original object.  Extension data, or an analytically enhanced substitute retaining that data, is still required there.

\begin{corollary}[Sharp failure on the full polystable locus]\label{cor:full-polystable-failure}
The forgetful map from families of harmonic bundles to polystable Higgs families is not locally essentially surjective, even for real-analytic families over a disk.  Moreover, for every $d\in\{0,1,2,\ldots,\infty\}$, the extension-generated stack $\mathscr M_{\Dol,d}^{\cH}(X)$ is a proper substack of the stack of semistable $C^d_UC^\infty_X$ Higgs families satisfying the NAH numerical conditions.  Its intersection with the polystable substack is a proper substack of the latter.  In fact, neither $\mathscr M_{\Dol,d}^{\cH}(X)$ nor the full polystable substack contains the other.
\end{corollary}

\begin{proof}
The family \eqref{eq:square-root-polystable-family} is real analytic and
therefore defines a $C^d_UC^\infty_X$ polystable family for every listed
$d$.  By \cref{prop:square-root-polystable-failure}, no neighbourhood of
its central parameter admits even a continuous harmonic metric.  Hence
the forgetful harmonic-to-polystable map has no local lift of this
object and is not locally essentially surjective.

The same proposition rules out every continuous relative harmonic
filtration.  The filtration criterion therefore excludes the family
from $\mathscr M_{\Dol,d}^{\cH}(X)$ for all $d$, while it remains an
object of the ambient polystable, and hence semistable, $C^d$ stack.
This proves that the intersection of the extension-generated and
polystable substacks is proper, and also that the full polystable
substack is not contained in $\mathscr M_{\Dol,d}^{\cH}(X)$.

Conversely, harmonic Higgs bundles on the numerical locus are
polystable of slope zero, and an extension of semistable slope-zero
Higgs bundles is again semistable of slope zero.  Finite extension
completion and stackification therefore place
$\mathscr M_{\Dol,d}^{\cH}(X)$ inside the semistable stack.  The
square-root family above shows that this inclusion is proper.  For the
opposite noncontainment, the Simpson family recalled in
\cref{subsec:elliptic-degeneration} has a smooth relative harmonic
filtration and hence belongs to $\mathscr M_{\Dol,d}^{\cH}(X)$ for every
$d$, whereas its fibre at $t=0$ is a non-split extension of two
rank-one harmonic Higgs bundles and is strictly semistable rather than
polystable.  Thus the extension-generated substack is not contained in
the polystable substack either.
\end{proof}

\subsection{The elliptic degeneration and filtration data}\label{subsec:elliptic-degeneration}

We now revisit the family used in \cite{AzamRayan2026}, originating in Simpson's counterexample \cite{Simpson1994}.  Let $X$ be an elliptic curve with a holomorphic one-form $dz$.  Let
\[
  U=(\C\setminus\{0\})\times\R
\]
with coordinates $(a,t)$.  On the trivial rank-two bundle, consider the Higgs field
\begin{equation}\label{eq:Simpson-family-A}
  A(a,t)
  =
  \begin{pmatrix}
    0 & dz\\
    0 & at\,dz
  \end{pmatrix}.
\end{equation}
As explained in \cite{AzamRayan2026}, this family is globally an extension of two rank-one harmonic Higgs families.  Therefore it belongs to $\MHDol(X)$ by \cref{thm:filtration-criterion}.

The associated flat family in the extension correspondence has connection matrix
\begin{equation}\label{eq:Simpson-family-B}
  B(a,t)
  =
  \begin{pmatrix}
    0 & dz+\frac{\bar a}{a}\,d\bar z\\
    0 & at\,dz+\bar a t\,d\bar z
  \end{pmatrix},
\end{equation}
with the same convention as in \cite{AzamRayan2026,Simpson1994}.  In the coarse semistable Higgs moduli space,
\[
  \lim_{t\to0}A(a,t)
\]
is independent of $a$ after S-equivalence, whereas the corresponding flat limit retains $a$-dependence through $\bar a/a$.  This obstructs continuity of a coarse extension.

The elliptic degeneration also shows why the filtration language retains information that disappears under coarse S-equivalence.  We make that distinction explicit before turning to obstruction examples.

The family \eqref{eq:Simpson-family-A} has a visible invariant line subbundle generated by the first basis vector.  In the resulting filtration
\[
  0\subset L_1\subset E,
\]
the quotient is another line family.  The off-diagonal term $dz$ is a smooth relative extension cocycle.  Passing to S-equivalence discards this extension and remembers only the graded object
\[
  L_1\oplus L_2.
\]
The diffeological extension-generated stack instead retains the cocycle through the family.

\begin{proposition}[Family-level explanation of the coarse discontinuity]\label{prop:coarse-discontinuity-explanation}
In the elliptic family \eqref{eq:Simpson-family-A}, the parameter dependence lost by the coarse semistable Dolbeault quotient is carried by relative extension data and its harmonic transformation.  The diffeological stack correspondence remains defined because the relative harmonic filtration persists smoothly at $t=0$.
\end{proposition}

\begin{proof}
In the chosen trivialization of the elliptic family, the invariant line subbundle is independent of $(a,t)$.  It therefore defines a smooth relative subbundle through $t=0$, and the quotient is smooth as well.  Both graded pieces are line-bundle families satisfying the harmonic conditions.  Hence the two-step filtration persists as a relative harmonic filtration at the limiting parameter value.  By \cref{thm:filtration-criterion}, the whole family remains an object of $\MHDol(X)$.

The distinction from the coarse semistable quotient is now visible.  Passing to S-equivalence first replaces the extension by its associated graded object and discards the off-diagonal extension cocycle.  The extension-generated stack does not perform this collapse: it transports the smooth relative extension class through the harmonic mediator.  Consequently the transformed family retains parameter dependence that is invisible in the coarse Dolbeault quotient.  Formula \eqref{eq:Simpson-family-B} exhibits that dependence explicitly and shows that it survives at the level of the family stack even though the coarse-space map is discontinuous.
\end{proof}

\subsection{Obstruction models and approach to the polystable boundary}

Let a central semistable Higgs bundle split smoothly as
\[
  E_0=F\oplus Q
\]
and let $F$ be the desired first filtration term.  Consider an infinitesimal deformation
\[
  \eta=
  \begin{pmatrix}
    \eta_F & \eta_{FQ}\\
    \eta_{QF} & \eta_Q
  \end{pmatrix}.
\]
Then the first obstruction to preserving $F$ is
\begin{equation}\label{eq:block-obstruction}
  \operatorname{ob}_1(F,\eta)
  =[\eta_{QF}]
  \in\mathbf H^1\bigl(\cC^\bullet_{\Dol}(\Hom(F,Q))\bigr).
\end{equation}
The upper-right block changes the extension of $Q$ by $F$ but does not destroy invariance of $F$; the lower-left block attempts to move $F$ into $Q$ and is the obstruction.

The block form shows directly that only the lower-left component enters the first obstruction for this chosen subbundle.

The first two-step model admits a useful vanishing criterion: when the relevant hypercohomology group vanishes, the local invariant subbundle lifts smoothly.

Suppose
\[
  \mathbf H^1\bigl(\cC^\bullet_{\Dol}(\Hom(F,Q))\bigr)=0.
\]
Then every closed infinitesimal lower-left source is a coboundary and can be removed to first order by varying the subbundle.  Under the hypotheses of \cref{thm:smooth-lifting-subobject}, the invariant subbundle persists smoothly after shrinking the parameter neighbourhood.  If the resulting quotient families satisfy the hypotheses of \cref{cor:acyclic-JH}, this gives local membership in the extension-generated stack.

The complementary model shows how a nonzero class prevents such a lift even though invariant subspaces exist fibrewise.

Conversely, suppose
\[
  \mathbf H^1\bigl(\cC^\bullet_{\Dol}(\Hom(F,Q))\bigr)\ne0
\]
and choose a deformation whose lower-left block represents a nonzero class.  Then $F$ cannot extend to first order as an invariant subbundle.  This does not prove that the deformed family lies outside $\MHDol(X)$: a different filtration may exist.  It does prove that the chosen Jordan--H\"older term cannot be assembled through that deformation direction.

\begin{warning}
The obstruction theory is filtration-dependent.  Semistable objects can admit multiple Jordan--H\"older filtrations, and the failure of one chosen filtration to extend does not imply the failure of all relative harmonic filtrations.
\end{warning}

Finally, we return to the analytic boundary of the stable locus and relate degeneration of the Green operator to the appearance of a polystable stabilizer.  Let $u\mapsto(E_u,\cD''_u)$ be stable for $u\ne0$ and polystable at $u=0$.  For $u\ne0$,
\[
  s_u=-G_{h_u}\cS_{h_u}(\eta_u).
\]
If the smallest positive eigenvalue $\lambda_1(u)$ of $L_{h_u}$ tends to zero, the standard spectral estimate gives
\begin{equation}\label{eq:blowup-estimate}
  \norm{s_u}_{L^2}
  \le \lambda_1(u)^{-1}\norm{\cS_{h_u}(\eta_u)}_{L^2}.
\end{equation}
The upper bound alone does not prove blow-up; it only records deterioration of the inverse estimate.  The complementary lower bound of \cref{prop:spectral-parameter-singularity} identifies when loss of regularity is forced: a component of the source along a small-eigenvalue direction must decay at least at the same rate as the eigenvalue.  This turns the approach to a stabilizer jump into a quantitative compatibility problem and links the present boundary model directly to the finite-regularity and singular-parameter questions studied in \cref{sec:finite-regularity}.

\section{Scope and further questions}\label{sec:further}

The stable-family theorem gives global smooth normalized harmonic metrics for arbitrary finite-dimensional smooth parameter plots and local real-analytic dependence when the coefficient data are real analytic.  \Cref{sec:finite-regularity} strengthens this in a different direction: for every $d=0,1,2,\ldots,\infty$, a stable family that is $C^d$ in the parameter and smooth along $X$ has a global normalized harmonic metric of the same mixed regularity.  The ambient coefficient-space solution operator also pulls back to locally ambiently smooth reduced singular parameter spaces.  These results give a positive stable-locus answer, and a partial semistable answer through filtrations, to Question~5.3.4 of \cite{AzamRayan2026}.  The resulting transformed smooth plot has the first-variation formula of \cref{sec:first-variation}; at smooth or unobstructed moduli points, its cohomology class agrees with the classical Dolbeault--de Rham tangent comparison.  At obstructed points we make no assertion that every closed degree-one class is represented by a plot.

For polystable families, the positive theorem applies on the locally split constant-decomposition-type locus.  There the metric ambiguity is described by multiplicity-space Hermitian forms and compact isotropy.  The explicit families of \cref{sec:explicit-polystable-failures} show that no corresponding local theorem holds on the full polystable locus; one of them is outside every $C^d$ extension-generated stack.  The same example nevertheless belongs to the weak $C^0$ operator-level mediator of Definition~\ref{def:weak-C0-mediator}: its determinant-normalized metrics degenerate, but the associated flat operators extend continuously.  Thus the example separates three loci---metric-regular harmonic families, weakly operator-harmonic families, and arbitrary fibrewise polystable families---that coincide on points but differ for parameter families.  For semistable families, \cref{thm:filtration-criterion} identifies the extension-generated stack with families admitting local relative harmonic filtrations, and \cref{sec:obstructions} studies the problem of lifting the required invariant subbundles.

The limitations are equally important, since they determine the next analytic and stack-theoretic problems.

The paper disproves smooth local existence of harmonic metrics for arbitrary polystable families: the obstruction can already occur at the level of continuity across a stabilizer jump.  It likewise disproves the assertion that every smooth or $C^d$ polystable family belongs to the corresponding metric-regular extension-generated stack, because the required invariant subbundles may fail to assemble even continuously.  The relative harmonic filtration criterion, rather than fibrewise polystability, is the exact positive condition for that original stack.  The weak mediator changes neither statement: it enlarges the $C^0$ endpoint images by retaining regular relative operators with singular metric witnesses, but it does not manufacture continuous invariant subbundles or a regular harmonic reduction.  Finally, \cref{thm:ambient-singular-stable} concerns reduced singular parameter spaces with ambient extensions; it is not a theory of harmonic metrics over arbitrary nonreduced $C^\infty$-schemes.

The Hodge construction of \cref{sec:lambda} is likewise a smooth diffeological construction on the loci for which harmonic families have been produced.  No global derived or shifted twistor stack is constructed here.

\subsection{Stratified polystable families and Kuranishi geometry}

A possible next step is to combine the Jacobi operator with a stratification of the polystable stack by automorphism type.  On a stratum where
\[
  \dim\End_{\mathrm{Higgs}}(E_u)
\]
is locally constant and the kernel bundle of $L_{h_u}$ varies smoothly, the reduced Green operator \eqref{eq:reduced-green} should provide a transverse implicit-function theorem.  The finite-dimensional kernel directions should then be organized by the centralizer groupoid.

The principal technical issue is to construct slices compatibly with diffeological pullback and descent.  A successful theorem would extend \cref{thm:constant-type} within suitable strata beyond explicitly split families; \cref{cor:full-polystable-failure} shows that it cannot cover arbitrary transverse plots.

The Kuranishi description suggests a natural diffeological stratification by obstruction type.  Understanding how these finite-dimensional local models glue is one route toward characterizing the liftable polystable locus.

The weak mediator introduces a second, strictly larger liftable locus at $C^0$ regularity.  Its Dolbeault objects are polystable families for which the pointwise coarse non-Abelian Hodge image admits a continuous semisimple flat lift whose complementary harmonic operator is continuous, even though no continuous harmonic reduction need exist.  The coarse map alone does not supply such a lift, so its existence and nonuniqueness are relative stack questions rather than consequences of the coarse homeomorphism.

\begin{problem}[Weakly liftable polystable locus]\label{prob:weak-liftable-locus}
Characterize the essential image of
\[
  \mathscr H^{\mathrm{wk}}_0(X)
  \longrightarrow
  \mathscr M_{\Dol,0}(X)
\]
and its de Rham counterpart.  Determine when a continuous coarse non-Abelian Hodge map lifts to a continuous family of semisimple flat bundles, classify the ambiguity of such lifts, and determine the maximal parameter regularity of the lifted operators near changes of stabilizer type.
\end{problem}

At a polystable point, the local moduli problem is controlled by a differential graded Lie algebra and a Kuranishi map.  Diffeology offers a way to retain smooth parameter plots through singular quotient models.  It would be valuable to compare:
\begin{enumerate}[label=(\roman*)]
\item the Kuranishi obstruction map for Higgs deformations;
\item the finite-dimensional reduced harmonic-metric equation \eqref{eq:LS-reduced};
\item the filtration obstruction classes \eqref{eq:first-obstruction} and \eqref{eq:obN}.
\end{enumerate}
A useful comparison problem is to determine which parts of these three constructions are related by the harmonic-bundle identifications and which depend on distinct slice choices.

\subsection{Higher variation and moving geometric data}

The first derivative is
\[
  s_1=-G_h\cS_h(\eta_1).
\]
In a fixed local slice, differentiating the harmonic equation a second time and collecting terms gives an equation of the form
\begin{equation}\label{eq:second-metric-variation}
  L_hs_2
  =-\cS_h(\eta_2)-\cQ_h(\eta_1,s_1),
\end{equation}
where $\cQ_h$ denotes the Hessian contribution determined by the chosen local coordinates and slice.  Thus
\[
  s_2=-G_h\bigl(\cS_h(\eta_2)+\cQ_h(\eta_1,s_1)\bigr).
\]
Repeated differentiation gives analogous recursive equations, but we do not develop an invariant higher-jet formalism here.

\begin{problem}[Higher-order plotwise NAH]\label{prob:higher-order}
Develop an invariant higher-jet formalism for the diffeological non-Abelian Hodge transform and compare its obstruction tensors with higher-order isomonodromic deformation equations.
\end{problem}

Recent work on higher-order isomonodromic deformation makes this comparison particularly timely; see \cite{HuSunYangZuo2026}.

The higher-variation formulas were derived for fixed $X$.  Allowing the complex structure of the base manifold to move introduces additional Kodaira--Spencer terms and connects the present framework with universal moduli problems.

In this paper the K\"ahler manifold $X$ is fixed.  One can instead let the complex structure vary.  Then the Dolbeault operators, type decomposition, K\"ahler identities, and harmonic equation all depend on the parameter.  Hitchin's recent universal moduli construction over Teichm\"uller space provides a complementary perspective in the curve case \cite{Hitchin2026Universal}.

One may ask whether the analytic method extends after identifying the underlying smooth fibres and allowing the K\"ahler structure to vary.  The linearized operator is expected to remain elliptic in a suitable relative slice, while additional source terms should arise from variation of $\Lambda_\omega$ and of the type decomposition.  No theorem for varying $X$ is proved here.

\begin{problem}[Relative base manifold]\label{prob:varying-X}
Extend the parametric harmonic-metric theorem to a smooth proper family of compact K\"ahler manifolds
\[
  \mathcal X\to U
\]
and formulate the resulting non-Abelian Hodge transform as a morphism of diffeological stacks over the parameterized K\"ahler moduli problem.
\end{problem}

A second direction changes the asymptotic geometry on $X$ rather than the regularity or singularity of the parameter space.  We keep this terminology separate from \cref{sec:finite-regularity}: parabolic, tame, and wild harmonic bundles should require weighted versions of the parameter-dependent estimates used here.

The compact nonsingular case isolates the family mechanism.  For parabolic Higgs bundles, parameter dependence of harmonic metrics has already been studied in important settings, notably by Kim and Wilkin \cite{KimWilkin2018}.  Singular settings introduce weighted Sobolev spaces, model metrics, and possible wall crossing in parabolic weights.  The diffeological viewpoint suggests treating weight parameters and Higgs-family parameters simultaneously.

\begin{problem}[Tame and wild diffeological NAH]\label{prob:singular}
Construct diffeological moduli stacks of smooth families of tame or wild harmonic bundles with varying singular data, and determine the loci on which harmonic metrics vary smoothly across weight and residue parameters.
\end{problem}

The vector-bundle presentation is likewise not essential in principle.  For principal groups, however, the normalization and centralizer geometry must be reformulated in the adjoint bundle and deserve separate treatment.

We have written the paper for complex vector bundles.  An analogous argument is expected for principal $G$-Higgs bundles with $G$ complex reductive, but no principal-group extension is proved here.  The metric is replaced by a reduction to a maximal compact subgroup $K$, and the Jacobi operator acts on sections of the associated bundle with fibre the symmetric complement $\mathfrak g/\mathfrak k$.

On the stable/simple locus, one expects the kernel to be the appropriate compact central part, so that a central normalization should produce a transverse inverse under the usual simplicity hypotheses.  At polystable points, the kernel should be governed by the reductive centralizer.  Thus an analogue of the stable/polystable distinction developed here is natural in Lie-theoretic form, but establishing it with the required slices and parameter-uniform estimates lies beyond the vector-bundle argument proved in this paper.

\subsection{Derived structures and the role of diffeological geometry}

The diffeological stack remembers smooth finite-dimensional plots.  Derived stacks remember deformation complexes and higher automorphisms.  These perspectives should interact.  Recent work of Kryczka, Tanaka, and Yau establishes Lagrangian correspondences of non-Abelian Hodge type and shifted twistor structures for derived moduli of perfect complexes \cite{KryczkaTanakaYau2026}.  A comparison with the present family-level smoothness would require a functor from suitable derived analytic plots to diffeological plots or a common enhancement.

One concrete question is whether the Green-operator differential \eqref{eq:inf-NAH-operator} can be interpreted as a choice of real-analytic splitting in the tangent complex of a derived correspondence.

The derived perspective does not replace the earlier diffeological construction.  The latter remains the mechanism that retains smooth family data across non-polystable extension phenomena.

The first paper \cite{AzamRayan2026} constructed an equivalence before the analytic questions addressed here were settled.  The present results clarify that this order was useful.  The extension-generated stack exists categorically even where a smooth selected harmonic metric is unavailable.  The analytic theorem then identifies a large locus on which the mediator can be chosen smoothly and uniquely after normalization.

Thus the categorical and analytic constructions are complementary:
\[
  \text{harmonic families}
  \quad\leadsto\quad
  \text{smooth stable transform},
\]
while
\[
  \text{extension completion}
  \quad\leadsto\quad
  \text{semistable stack correspondence}.
\]
The obstruction theory measures the gap between them.

These directions point to a common issue: the interaction between analytic regime, stabilizer type, and the stack geometry of families matters more than any single universal regularity statement.

For families, the variation of the harmonic reduction separates three regimes.  On the stable locus the relevant Jacobi operator is invertible after normalization.  At polystable points its kernel records centralizers.  In the semistable problem one must additionally transport invariant subobjects and extension data.  The stack geometry reflects these different analytic mechanisms.

The diffeological framework does not require these regimes to fit a single smooth manifold model; smoothness is tested on the parameter spaces that occur as plots.  On the stable locus and on the controlled polystable strata treated here, harmonic metrics provide that smoothness.  Beyond those loci, the linearized equations expose the stabilizer and subobject obstructions that remain.

\appendix

\section{Parameter-dependent elliptic operators}\label{app:elliptic}

This appendix collects analytic facts used in the main text.  We include proofs to make clear exactly where finite-dimensionality of the parameter manifold and compactness of $X$ enter.

\subsection{Sobolev conventions and smooth coefficient families}

Let $V\to X$ be a smooth real or complex vector bundle over a compact $m$-dimensional manifold.  Fix a background metric and connection.  For integer $j\ge0$, let
\[
  W^{j,2}(V)
\]
be the corresponding Sobolev space.  Different background choices give equivalent norms.  If
\[
  j>m/2,
\]
then $W^{j,2}$ is a Banach algebra for scalar functions, and the standard Sobolev multiplication maps are bounded in the lower-order spaces used below.  If
\[
  j>m/2+1,
\]
then $W^{j,2}\hookrightarrow C^1$.

The nonlinear moment map uses two derivatives of the logarithmic metric coordinate, so we work with
\[
  s\in W^{k+2,2},
  \qquad
  \mu(s)\in W^{k,2},
\]
for $k>m/2+1$.

With the Sobolev conventions fixed, we record how differential operators depending smoothly on a finite-dimensional parameter act between the corresponding fixed Sobolev spaces.

Let $B\subset\R^d$ be open.  A family of differential operators
\[
  P_u=\sum_{\abs\alpha\le r}a_\alpha(u,x)\partial_x^\alpha
\]
is smooth when the coefficients are jointly smooth in $(u,x)$.  Such a family defines a smooth map
\[
  B\to\mathcal L(W^{j+r,2}(V),W^{j,2}(W))
\]
for every $j$ for which the products are defined.

\begin{lemma}[Operator-norm smoothness]\label{lem:operator-norm-smoothness}
If the coefficients $a_\alpha(u,\cdot)$ are smooth into $C^{j+r}$, then
\[
  u\mapsto P_u
\]
is smooth in the operator norm topology
\[
  \mathcal L(W^{j+r,2},W^{j,2}).
\]
\end{lemma}

\begin{proof}
For a multi-index $\alpha$ with $|\alpha|\le r$, multiplication by the coefficient $a_\alpha(u,\cdot)$ followed by $\partial^\alpha$ defines a bounded map
\[
  W^{j+r,2}\longrightarrow W^{j,2}.
\]
The corresponding operator norm is controlled by finitely many $C^{j+r}$ norms of the coefficient.  Hence on a relatively compact parameter set there is a uniform estimate
\[
  \|P_u\|_{\mathcal L(W^{j+r,2},W^{j,2})}
  \le C\sum_{|\alpha|\le r}\|a_\alpha(u,\cdot)\|_{C^{j+r}}.
\]

Differentiate the coefficients with respect to $u$.  Each parameter derivative $\partial_u^\beta P_u$ is again a differential operator of order at most $r$, with coefficients $\partial_u^\beta a_\alpha(u,x)$, and satisfies the same type of operator-norm estimate.  Taylor's formula for the coefficient functions, together with these uniform estimates, shows that the difference quotients converge in $\mathcal L(W^{j+r,2},W^{j,2})$.  Iterating the argument for all parameter derivatives proves smoothness in operator norm.
\end{proof}

\subsection{Uniform estimates and smooth inverses}

Suppose $P_u$ is a family of second-order strongly elliptic operators.  On a compact parameter subset $K\Subset B$, strong ellipticity is uniform: there is $c>0$ such that
\[
  \Re\ip{\sigma_2(P_u)(x,\xi)v}{v}
  \ge c\abs\xi^2\abs v^2
\]
for all $u\in K$.

\begin{proposition}[Uniform estimate]\label{prop:uniform-estimate-app}
For every $j\ge0$ there is $C_{j,K}$ such that
\[
  \norm{v}_{W^{j+2,2}}
  \le C_{j,K}
  \bigl(\norm{P_uv}_{W^{j,2}}+\norm v_{L^2}\bigr)
\]
for all $u\in K$.
\end{proposition}

\begin{proof}
Choose finitely many coordinate charts covering $X$ and trivialize the bundle on each chart.  Strong ellipticity gives a local G\aa rding-type estimate for compactly supported sections,
\[
  \|v\|_{W^{j+2,2}(U')}
  \le C\bigl(\|P_uv\|_{W^{j,2}(U)}+\|v\|_{L^2(U)}\bigr),
\]
whenever $U'\Subset U$.  The constant depends only on a positive lower bound for the principal-symbol ellipticity and on finitely many coefficient norms of $P_u$.

For $u\in K$, compactness of $K$ and smoothness of the coefficients give uniform bounds for all coefficient seminorms entering the local estimate.  The uniform ellipticity hypothesis supplies one lower ellipticity constant valid for every $u\in K$.  Thus the local constants can be chosen independently of $u$.  Apply the estimate to a finite partition of unity subordinate to the coordinate cover.  Commutators of $P_u$ with the cutoff functions have order at most one and are absorbed into the lower-order terms using the same uniform coefficient bounds.  Summing the finitely many local estimates and enlarging the constant yields
\[
  \|v\|_{W^{j+2,2}}
  \le C_{j,K}\bigl(\|P_uv\|_{W^{j,2}}+\|v\|_{L^2}\bigr)
\]
for all $u\in K$.
\end{proof}

Uniform estimates become particularly effective when the kernel vanishes: the inverse then persists and depends smoothly on the parameter.  This is the precise functional-analytic input used in the stable metric theorem.

Let
\[
  P_u:W^{j+2,2}(V)\to W^{j,2}(V)
\]
be a smooth family of bounded operators.  Suppose $P_{u_0}$ is invertible.  Then invertibility persists locally.

\begin{proposition}[Smooth inverse family]\label{prop:smooth-inverse-app}
After shrinking around $u_0$, $P_u$ is invertible and
\[
  u\mapsto P_u^{-1}
\]
is smooth in
\[
  \mathcal L(W^{j,2},W^{j+2,2}).
\]
Its derivative is
\begin{equation}\label{eq:inverse-derivative}
  \partial_{u_i}P_u^{-1}
  =-P_u^{-1}(\partial_{u_i}P_u)P_u^{-1}.
\end{equation}
\end{proposition}

\begin{proof}
Write
\[
  P_u=P_{u_0}\bigl(1+P_{u_0}^{-1}(P_u-P_{u_0})\bigr).
\]
Set
\[
 K_u=P_{u_0}^{-1}(P_u-P_{u_0})
 \in\mathcal L(W^{j+2,2},W^{j+2,2}).
\]
Operator-norm continuity gives $\|K_u\|<1$ after the neighbourhood is
shrunk.  Hence
\[
 (1+K_u)^{-1}=\sum_{m=0}^\infty(-K_u)^m
\]
converges in operator norm and
$P_u^{-1}=(1+K_u)^{-1}P_{u_0}^{-1}$.  This proves both persistence of
invertibility and continuity of the inverse.  Termwise differentiation
of the locally uniformly convergent series, or smoothness of inversion
on the open group of invertible bounded operators, proves smoothness to
all orders.  Finally, differentiating
$P_uP_u^{-1}=\Id$ gives
\[
 (\partial_{u_i}P_u)P_u^{-1}
 +P_u(\partial_{u_i}P_u^{-1})=0.
\]
Multiplication on the left by $P_u^{-1}$ yields
\eqref{eq:inverse-derivative}.
\end{proof}

\subsection{Joint regularity and fixed-dimensional kernels}

We give a detailed version of \cref{prop:parametric-bootstrap}.

\begin{theorem}[Joint regularity theorem]\label{thm:joint-regularity-app}
Let $P_u$ be a smooth family of second-order strongly elliptic operators and let
\[
  v:B\to W^{k+2,2}(V)
\]
be a $C^\infty$ map satisfying
\[
  P_uv(u)=f(u),
\]
where $f$ is jointly $C^\infty$ in $(u,x)$.  If the coefficients of $P$ are jointly smooth, then $v$ is jointly smooth.
\end{theorem}

\begin{proof}
We first improve $X$-regularity without differentiating in $u$.  The equation and elliptic regularity imply that if $f(u)$ and the coefficients are smooth in $x$, then $v(u)$ is smooth in $x$ for each fixed $u$.

For uniform parameter control, fix $K\Subset B$.  The estimate of \cref{prop:uniform-estimate-app} gives
\[
  \norm{v(u)}_{W^{j+2,2}}
  \le C_{j,K}
  \bigl(\norm{f(u)}_{W^{j,2}}+\norm{v(u)}_{L^2}\bigr).
\]
Continuity of $v$ in the original Sobolev topology bounds the $L^2$ term.  Induction gives local boundedness in every Sobolev norm.

Now differentiate in $u_i$:
\[
  P_u\partial_{u_i}v
  =\partial_{u_i}f-(\partial_{u_i}P_u)v.
\]
The right side is bounded in arbitrary Sobolev order because $v$ has already been bootstrapped in $x$.  Thus $\partial_{u_i}v$ has arbitrary $X$-regularity.  Repeating for a multi-index $\beta$ gives
\[
  P_u\partial_u^\beta v
  =\partial_u^\beta f
  -\sum_{0<\gamma\le\beta}
  \binom{\beta}{\gamma}
  (\partial_u^\gamma P_u)
  (\partial_u^{\beta-\gamma}v).
\]
Induction on $\abs\beta$ and elliptic estimates yield smooth dependence into every Sobolev space.  Sobolev embedding then implies joint smoothness of all mixed derivatives.
\end{proof}

For polystable strata the kernel need not vanish but may have constant dimension.  The same elliptic framework then gives smooth kernel projections and inverses on the orthogonal complement.

The polystable analysis uses a Fredholm family with nonzero kernel.  Let $P_u$ be self-adjoint elliptic and suppose
\[
  \dim\ker P_u=N
\]
is constant near $u_0$.  Assume a spectral gap separates zero from the rest of the spectrum.  Choose a small circle $\Gamma$ around zero containing no other spectrum.  The spectral projection is
\begin{equation}\label{eq:Riesz-projection}
  \Pi_u=\frac1{2\pi\sqrt{-1}}
  \int_\Gamma(z-P_u)^{-1}\,dz.
\end{equation}

\begin{proposition}[Smooth kernel bundle]\label{prop:smooth-kernel-bundle-app}
Under the hypotheses above, $u\mapsto\Pi_u$ is smooth in operator norm.  The kernels form a smooth rank-$N$ vector bundle over the parameter neighbourhood, and the reduced inverse
\[
  G_u^\perp=(P_u|_{(1-\Pi_u)H})^{-1}
\]
varies smoothly.
\end{proposition}

\begin{proof}
Choose the contour $\Gamma$ inside the fixed spectral gap.  For every $z\in\Gamma$, the operator $z-P_u$ is invertible for $u$ near $u_0$.  By \cref{prop:smooth-inverse-app},
\[
  u\longmapsto (z-P_u)^{-1}
\]
is smooth in operator norm, uniformly for $z$ on the compact contour.  Differentiation may therefore be passed under the contour integral in \eqref{eq:Riesz-projection}, proving that $u\mapsto\Pi_u$ is smooth in operator norm.

Because $\Pi_u$ is a smooth family of projections of constant finite rank $N$, its ranges form a smooth rank-$N$ subbundle: locally, a basis of $\operatorname{ran}\Pi_{u_0}$ remains linearly independent after applying $\Pi_u$ for $u$ close to $u_0$.  Self-adjointness gives the complementary smooth subbundle $\operatorname{ran}(1-\Pi_u)$.

The spectral gap also gives a uniform lower bound for $|P_u|$ on this complementary subspace.  Hence the restricted operator
\[
  P_u:(1-\Pi_u)W^{j+2,2}\to(1-\Pi_u)W^{j,2}
\]
is invertible.  Trivialize the moving complementary bundles by the smooth projections and apply the smooth inverse theorem to the resulting fixed-space family.  This proves smooth dependence of the reduced inverse $G_u^\perp$.
\end{proof}

The constant decomposition-type theorem gives a more explicit version of this spectral construction.

The preceding smooth-dependence statements have a real-analytic analogue whenever the coefficient family is analytic in operator norm.  The argument is local in the parameter and uses the same spectral separation as above.

Suppose $u\mapsto P_u$ is real analytic in operator norm.  Invertibility at $u_0$ implies analytic dependence of the inverse near $u_0$.  One can see this from the Neumann series
\[
  P_u^{-1}
  =\bigl(1+P_{u_0}^{-1}(P_u-P_{u_0})\bigr)^{-1}P_{u_0}^{-1}
\]
and analytic dependence of the geometric series.  The same argument applies to the resolvent and Riesz projection when the kernel dimension is fixed.

We finish the appendix by applying these general statements to the Jacobi family arising from the Hitchin--Simpson equation.

For the harmonic metric equation,
\[
  P_u=L_{h_u}=(\cD''_u)^{*h_u}\cD''_u.
\]
Its principal symbol is the scalar Laplace symbol, independent of the Higgs zero-order term.  On a stable normalized family it is invertible.  Therefore all results above apply.  On a constant-type polystable family, the kernel projection is explicitly determined by the centralizer and is smooth.

\section{Normalization, scalar ambiguity, and determinant data}\label{app:normalization}

\subsection{Scalar ambiguity and integral normalization}

Let $(E,\cD'')$ be stable and let $h_1,h_2$ be harmonic.  The standard uniqueness theorem implies
\[
  h_2=c\,h_1
\]
for some $c>0$.  In a family, the scalar may be a positive function $c(u)$.  Without normalization, a smooth family of harmonic metrics is therefore not unique.  This ambiguity is harmless for some constructions but not for a literal implicit-function inverse.

Uniqueness up to scalar is converted into an actual choice by fixing an integral normalization.  The formula below is convenient because it behaves smoothly in families.

Given a smooth background metric $\kappa$, define
\[
  N_\kappa(h)
  =\frac1{r\Vol_\omega(X)}
  \int_X\log\det(\kappa^{-1}h)\,\frac{\omega^n}{n!}.
\]
Then
\[
  N_\kappa(ch)=N_\kappa(h)+\log c.
\]
Hence every scalar class of metrics contains a unique representative satisfying
\[
  N_\kappa(h)=0.
\]
If $h(u)$ is smooth, the normalizing scalar
\[
  c(u)=e^{-N_\kappa(h(u))}
\]
is smooth.

\begin{proposition}[Equivalence of normalizations]\label{prop:equiv-normalizations}
On a stable family, determinant normalization and integral normalization differ by multiplication by a unique smooth positive scalar function of the parameter.  Either normalization yields exactly the same relative flat connection, because the two metrics differ by a positive scalar depending only on the parameter and hence constant in the $X$-directions.
\end{proposition}

\begin{proof}
Let $h_{\det}(u)$ and $h_{\mathrm{int}}(u)$ denote harmonic metrics selected by determinant and integral normalization, respectively.  Stability implies that for every $u$ there is a unique positive scalar $c(u)$ with
\[
  h_{\mathrm{int}}(u)=c(u)h_{\det}(u).
\]
The scalar is smooth.  For example, taking determinants gives
\[
  c(u)^r=\frac{\det h_{\mathrm{int}}(u)}{\det h_{\det}(u)},
\]
and the right side is a smooth positive function.  Equivalently, applying $N_\kappa$ gives
\[
  N_\kappa(c h)=N_\kappa(h)+\log c,
\]
so the integral normalization also determines $c$ smoothly and uniquely.

The scalar depends only on the parameter $u$.  Therefore it is constant along the $X$-directions, and the calculation of \cref{prop:scalar-independence} applies: the relative Chern connection is unchanged, as is the adjoint Higgs field.  Consequently
\[
  \nabla_{c(u)h}=\nabla_h.
\]
Thus the two normalization procedures choose different representatives of the same scalar class of harmonic metrics but produce exactly the same relative flat connection.
\end{proof}

\subsection{Determinant normalization in families}

If $h_t=he^{ts}$, then
\[
  \frac d{dt}\bigg|_0\log\det h_t=\tr(s).
\]
Thus fixing the determinant is exactly the linear condition
\[
  \tr(s)=0.
\]
The scalar kernel of $L_h$ is removed by restricting to this trace-free subspace.  By contrast, integral normalization imposes a mean-zero trace condition.  It also removes the scalar kernel but is slightly less local in $x$.

The determinant line packages the scalar normalization globally over the parameter space and separates it from the trace-free nonlinear equation.

The determinant holomorphic structure is induced from $\ddbar_E$.  The Higgs field on $\det E$ is $\tr\theta$.  The determinant moment-map equation is abelian.  Degree zero makes its Hermitian--Einstein constant vanish, and the scalar Poisson argument of \cref{prop:smooth-det} gives a smooth normalized solution on the determinant line over all of $U\times X$.  A global background metric exists without trivializing the line, and the fixed Green operator acts only on the $X$-variable; hence parameter-space monodromy creates no obstruction.  Under the NAH numerical hypotheses used in the main theorem, the subsequent Chern--Weil step is what places this determinant equation inside the flat harmonic correspondence.  Thus no separate nonlinear or descent difficulty is hidden in the determinant part of the implicit-function argument.

In a fixed-determinant problem this scalar part disappears from the tangent calculation, leaving only the trace-free Jacobi equation.

If the family is given in a fixed determinant moduli problem with a prescribed holomorphic determinant line and a prescribed Hermitian--Einstein determinant metric, the determinant step disappears.  The nonlinear map is then defined directly on trace-free Hermitian endomorphisms.  This is the simplest setting for the local analysis and is the usual one for fixed-determinant $SL_r$ problems.

The contrast with the polystable case is worth making explicit.  Fixing the determinant removes only one scalar direction and does not eliminate the higher-dimensional centralizer freedom.

For a polystable decomposition
\[
  E=\bigoplus_jE_j\otimes W_j,
\]
fixing $\det h$ imposes one scalar equation on
\[
  \prod_j\operatorname{Herm}^+(W_j).
\]
It does not collapse the symmetric-space fibre unless there is only one stable factor of multiplicity one.  Thus determinant normalization is sufficient precisely on the stable locus and not on the general polystable locus.

\section{Comparison with the first diffeological paper}\label{app:dictionary}

The notation and constructions used here build directly on \cite{AzamRayan2026}.  The following dictionary identifies the main objects in the two papers.

\subsection{Moduli stacks and harmonic families}

The earlier paper constructs diffeological stacks
\[
  \MdiffDol(X),\qquad\MdiffdR(X)
\]
from dg-prestacks of relative connections with respect to partial Dolbeault and partial exterior differentials.  The present paper uses the same stacks but works, for the analytic sections, in local smooth trivializations of a plot.

The notation
\[
  \MHDol(X),\qquad\MHdR(X)
\]
denotes the substacks generated from the harmonic mediator under finite iterated extensions and stackification.  The equivalence
\[
  \MHDol(X)\simeq\MHdR(X)
\]
is the principal theorem of \cite{AzamRayan2026}.

The mediator in the earlier paper is the prestack of harmonic families.  The present analytic theorem shows that every stable family lifts to an object of that mediator over the same parameter manifold.  Thus stable objects are already in the essential image of its Dolbeault forgetful map before extension completion or stackification; choosing a determinant normalization makes the metric unique, although that normalization is not canonical.

Definition~\ref{def:weak-C0-mediator} introduces a different enlargement at $C^0$ parameter regularity.  It retains mixed $C^0_UC^\infty_X$ operators $(\cD'',\cD')$ and requires only that $\cD'_u$ be induced by some smooth harmonic metric on each slice; the witnesses need not vary continuously.  The original $C^0$ harmonic mediator is the metric-regular subprestack.  Proposition~\ref{prop:weak-flat-lift-square-root} shows that the inclusion is strict, while Proposition~\ref{prop:weak-extension-correspondence} applies the earlier categorical extension construction to the continuous operator data.  This weak mediator was not part of \cite{AzamRayan2026}; it is motivated by the compensated regularity discovered in the square-root family.

A smooth harmonic family in the earlier paper consists of relative operators satisfying the harmonic bundle identities slicewise and smoothly in the parameter.  In the present paper such a family is written
\[
  (E,\cD'',h),
  \qquad
  \cD''=\ddbar_E+\theta,
\]
with
\[
  \cD'_h=\partial_{E,h}+\theta^{\dagger_h}.
\]
The flat connection is
\[
  \nabla_h=\cD''+\cD'_h.
\]

\subsection{Open questions, extension completion, and geometricity}

Question 5.3.1 of \cite{AzamRayan2026} asks whether fibrewise harmonic metrics on a smooth polystable family can be chosen smoothly.  The present paper proves:
\begin{enumerate}[label=(\roman*)]
\item a global smooth-existence theorem for stable families, obtained from local analytic dependence and normalized gluing, \cref{thm:intro-smooth-metrics};
\item a theorem for locally split constant-type polystable families, \cref{thm:constant-type};
\item a groupoid description of the residual ambiguity, \cref{prop:metric-groupoid-poly};
\item explicit real-analytic polystable families with no continuous family of harmonic metrics, \cref{prop:coalescing-lines-filtered,prop:square-root-polystable-failure};
\item a stratumwise question for characterizing the remaining liftable polystable loci, Question~\ref{ques:stratified-poly}.
\end{enumerate}

Question 5.3.2 asks whether all semistable or polystable families are locally finite iterated extensions of harmonic families.  The answer is negative: \cref{prop:square-root-polystable-failure} is polystable and is outside the extension-generated stack even locally at its central parameter.  The exact positive statement is the relative harmonic filtration criterion \cref{thm:filtration-criterion}; the obstruction theory developed here studies the assembly of the required fibrewise filtrations.  The contrasting family in \cref{prop:coalescing-lines-filtered} shows that a total harmonic metric is not necessary for extension-generated membership.

Question 5.3.3 asks about the relation between fibrewise quasi-fullness and local extension from harmonic bundles.  Our obstruction theory suggests one link: failure to lift an invariant subobject is detected in the hypercohomology of a Hom deformation complex, exactly the type of complex on which quasi-fullness questions are formulated.  A complete equivalence remains open.

Question 5.3.4 asks about $C^d$ parameter families and, beyond them, families over $C^\infty$-schemes.  \Cref{thm:Cd-stable-metrics} gives a global positive answer on the stable locus for every $d=0,1,2,\ldots,\infty$, with no loss of parameter regularity.  \Cref{thm:Cd-filtration-criterion,cor:Cd-semistable-criterion} give the corresponding positive criterion when a relative filtration of the required regularity exists, while \cref{prop:square-root-polystable-failure} shows that there are real-analytic polystable families outside $\mathscr M_{\Dol,d}^{\cH}(X)$ for every $d$.  \Cref{thm:ambient-singular-stable} treats a class of reduced singular parameter spaces by pulling the ambient coefficient-space solution operator back from a smooth presentation.  The genuinely nonreduced $C^\infty$-scheme problem remains open and is isolated in \cref{prob:nonreduced-Cinfty}.  The heat-flow discussion of \cref{sec:heat-flow-geometricity} gives a conceptual reason for the same boundary: an unenhanced harmonic-metric limit reaches the polystable shadow, while the extension-generated stack remembers the transverse extension data.

The finite extension completion is the categorical operation that enlarges these harmonic families to the semistable substack described geometrically by relative harmonic filtrations in this paper.

In the earlier paper, finite iterated extensions are categorical.  The present paper's \cref{prop:filtration-extension} identifies them with geometric filtrations by smooth invariant subbundles.  In a chosen splitting, the extension completion is represented by upper-triangular matrices \eqref{eq:upper-triangular-D} satisfying the Maurer--Cartan equations \eqref{eq:upper-MC}.

The resulting comparison is genuinely diffeological: plots retain parameter-dependent extension data that need not descend to a continuous map of coarse moduli spaces.

The earlier paper proves that the family moduli stacks are presented by diffeological groupoids.  We do not repeat those constructions.  The new analytic contribution is instead to show that on stable plots the harmonic metric and transformed connection are genuinely smooth in the parameter.  Thus the harmonic mediator is not merely a slicewise existence device on that locus; it defines a smooth operation on plots.

\section{Detailed variation formulas in local coordinates}\label{app:local-formulas}

This appendix expands several invariant formulas used in the text, principally to make the sign conventions checkable.

\subsection{Chern and adjoint variations}

Fix a holomorphic structure $\ddbar_E$ and a Hermitian metric $h$.  Let
\[
  h_t=he^{ts}
\]
with $s=s^{\dagger_h}$.  In a local frame, write the $(1,0)$ Chern connection matrix as
\[
  A_h^{1,0}=H^{-1}\partial H+\text{terms from }A^{0,1}.
\]
Then
\[
  \dot A_h^{1,0}=\partial_hs.
\]
Consequently
\[
  \frac d{dt}\bigg|_0F_{D_{h_t}}
  =\ddbar_E\partial_hs
\]
when the holomorphic structure is fixed.

The corresponding variation of the adjoint Higgs field is computed in the same local frame and uses the metric logarithm introduced earlier.

The adjoint is characterized by
\[
  h(\theta v,w)=h(v,\theta^{\dagger_h}w).
\]
Under $h_t=he^{ts}$,
\[
  \dot\theta^{\dagger}
  =[\theta^{\dagger_h},s].
\]
Therefore
\begin{align*}
  \frac d{dt}\bigg|_0[\theta,\theta^{\dagger_{h_t}}]
  &=[\theta,[\theta^{\dagger_h},s]].
\end{align*}
The combination with the Chern curvature variation is the Higgs Laplacian $L_hs$ after contraction with $\sqrt{-1}\Lambda$.

\subsection{Expanded Jacobi and simultaneous variations}

On a Hermitian endomorphism $s$, contraction with $\Lambda$ leaves the $(1,1)$ terms
\begin{equation}\label{eq:expanded-Jacobi}
  L_hs
  =\sqrt{-1}\Lambda\left(
  \ddbar_E\partial_hs
  +[\theta,[\theta^{\dagger_h},s]]
  \right).
\end{equation}
Indeed, the remaining terms in $\cD''\cD'_hs$ have types $(2,0)$ and $(0,2)$ and are annihilated by $\Lambda$.  The invariant identity
\[
  L_h=(\cD'')^*\cD''
\]
keeps the positivity and convention dependence transparent.

Combining the two elementary variations gives the full source term for a general Higgs deformation.

Let
\[
  \ddbar_{E,t}=\ddbar_E+ta+O(t^2),
  \qquad
  \theta_t=\theta+t\varphi+O(t^2).
\]
At fixed $h$, the companion variation of the Chern--Higgs operator is
\[
  \eta^{\star_h}=-a^{\dagger_h}+\varphi^{\dagger_h}.
\]
Accordingly, the fixed-metric derivative of the moment map is
\begin{align}\label{eq:expanded-source-app}
  \cS_h(a,\varphi)
  =\sqrt{-1}\Lambda\Bigl(&
  \partial_{E,h}a-\ddbar_Ea^{\dagger_h}
  +[\varphi,\theta^{\dagger_h}]
  +[\theta,\varphi^{\dagger_h}]
  \Bigr)_0.
\end{align}
Formula \eqref{eq:source-operator} remains the invariant definition.

\subsection{Differential and second variation of the transform}

In a fixed local smooth gauge and for a fixed-determinant plot direction,
\[
  s=-G_h\cS_h(a,\varphi)
\]
and a connection-valued representative of the transformed derivative is
\begin{align}\label{eq:expanded-dNAH-app}
  \dot\nabla
  =&\ a+\varphi-a^{\dagger_h}+\varphi^{\dagger_h}
  \\
  &\ -\cD'_hG_h\cS_h(a,\varphi).
\end{align}
The first line is the fixed-metric companion of the Higgs deformation; the second is the correction produced by variation of the harmonic metric.  A change of smooth gauge adds a $d_{\nabla_h}$-exact term.

A second differentiation shows where derivatives of the Green operator and quadratic source terms enter.  The schematic form needed for the higher-variation discussion is given below.

Let $p_t$ denote the Higgs data and $s_t$ the logarithmic metric coordinate, and set
\[
  p_i=\frac{d^ip_t}{dt^i}\bigg|_{t=0},
  \qquad
  s_i=\frac{d^is_t}{dt^i}\bigg|_{t=0}.
\]
The equation is
\[
  \Phi(p_t,s_t)=0.
\]
At first order,
\[
  Ls_1=-\Phi_p p_1.
\]
At second order,
\begin{equation}\label{eq:second-derivative-app}
  Ls_2
  =-\Phi_pp_2
  -\Phi_{pp}(p_1,p_1)
  -2\Phi_{ps}(p_1,s_1)
  -\Phi_{ss}(s_1,s_1).
\end{equation}
This formula is coordinate- and slice-dependent through the derivatives of $\Phi$.  It records the recursive structure used in \cref{sec:further}; no invariant higher-jet theorem is asserted here.

\section{Dependencies among the main results}\label{app:dependency}

The following table records the principal inputs used by the main results.

\begin{longtable}{@{}p{0.29\textwidth}p{0.62\textwidth}@{}}
\toprule
Result & Main input \\
\midrule
\endfirsthead

\toprule
Result & Main input \\
\midrule
\endhead

\midrule
\multicolumn{2}{r@{}}{\small\itshape Continued on next page} \\
\endfoot

\bottomrule
\endlastfoot
\Cref{prop:smooth-det} & Scalar Poisson equation and fixed Green operator \\
\Cref{prop:Jacobi} & Harmonic-bundle K\"ahler identities \\
\Cref{prop:stable-kernel} & Energy identity plus stability implies simplicity \\
\Cref{thm:intro-smooth-metrics} & Global determinant Poisson solve, Banach implicit-function theorem, parameter elliptic bootstrap, and normalized gluing \\
\Cref{thm:smooth-stable-transform} & Smooth metric family plus pullback uniqueness \\
\Cref{thm:first-variation-metric} & Differentiated moment map plus Green operator \\
\Cref{thm:differential-plot} & Variation of $\cD'_h$ and metric first variation \\
\Cref{thm:constant-type} & Stable theorem on each factor plus polystable uniqueness and centralizer decomposition \\
\Cref{prop:coalescing-lines-filtered} & Orthogonality of distinct stable eigensummands plus the visible invariant line filtration \\
\Cref{prop:square-root-polystable-failure} & Coalescing orthogonal eigenlines plus the monodromy obstruction to a continuous square root \\
\Cref{prop:weak-flat-lift-square-root} & Explicit singular determinant-normalized metric, compensated continuity of the adjoint Higgs field, and holonomy-trace regularity \\
\Cref{prop:weak-mediator-functoriality} & Pullback of continuous relative operators and slicewise harmonic witnesses \\
\Cref{prop:weak-extension-correspondence} & $C^0$ operator-level $\lambda$-$d$ family plus finite extension completion and stackification \\
\Cref{thm:filtration-criterion} & Extension completion plus filtration--extension equivalence \\
\Cref{thm:infinitesimal-lifting} & Linearized invariant-projection equation \\
\Cref{thm:kuranishi-subobject} & Hodge decomposition of the elliptic Hom complex, graph Bianchi identity, and implicit-function theorem \\
\Cref{thm:smooth-lifting-subobject} & Vanishing of the finite-dimensional obstruction space \\
\Cref{thm:Cd-stable-metrics} & Ambient coefficient-space implicit-function theorem plus finite-regularity elliptic bootstrap \\
\Cref{thm:Cd-filtration-criterion} & Finite extension completion in mixed $C^d_UC^\infty_X$ regularity \\
\Cref{thm:ambient-singular-stable} & Pullback of the ambient solution operator plus normalized uniqueness \\
\Cref{prop:ordinary-heat-polystable} & Hitchin--Simpson correspondence and polystability direction \\
\Cref{prop:heat-flow-factors-shadow} & Closed complex-gauge orbit represented by the Jordan--H\"older graded object \\
\Cref{prop:extension-completion-receptacle} & Relative harmonic filtration criterion and finite extension completion \\
\Cref{thm:stable-Hodge-family} & Smooth metric theorem plus harmonic bicomplex identities \\
\end{longtable}

The chart also makes clear which statements are unconditional theorems and which directions remain open.  In particular, Question~\ref{ques:stratified-poly} is not used as an input to any theorem.

\pagebreak

\end{document}